\documentclass[pagesize]{scrartcl}
\usepackage{mystyle}

\newcommand{\MW}[1]{{#1}}

\title{A Probabilistic Approach to\\ Shape Derivatives}
\author{
    Luka Schlegel\thanks{University of Trier, Department IV -- Mathematics, Universit\"{a}tsring 19, 54296 Trier, Germany.\\ L.\ Schlegel gratefully acknowledges financial support from the German Research Foundation (DFG)  within the Priority program SPP 1962 "Non-smooth and Complementarity-based Distributed Parameter Systems: Simulation and Hierarchical Optimization".\\ M.\ Würschmidt gratefully acknowledges financial support from the German Research Foundation (DFG) within the Research Training Group 2126: Algorithmic Optimization.}
    \and 
    Volker Schulz\footnotemark[1]
    \and 
    Frank T. Seifried\footnotemark[1] 
    \and 
    Maximilian Würschmidt\footnotemark[1]
}
\date{}

\begin{document}
\allowdisplaybreaks
\maketitle

\begin{abstract}
    We introduce a novel mesh-free and direct method for computing the shape derivative in PDE-constrained shape optimization problems. Our approach is based on a probabilistic representation of the shape derivative and is applicable for second-order semilinear elliptic PDEs with Dirichlet boundary conditions and a general class of target functions. The probabilistic representation derives from an extension of a boundary sensitivity result for diffusion processes due to Costantini, Gobet and El~Karoui \cite{Karoui_bdry}. Moreover, we present a simulation methodology based on our results that does not necessarily require a mesh of the relevant domain, and provide Taylor tests to verify its numerical accuracy.


\end{abstract}

\section{Introduction}
The optimization of shapes is a challenging task to be solved in ubiquitous application problems. This research area is quite mature but nevertheless a very active field. A brief and current overview can be gained from \cite{Alllaire2021}.
Foundational monographs on shape optimization include, e.g., \cite{Delfour_Shape, Haslinger2003,Pironneau_Shape, Simon1980, Sokolowski_Shape}.
Although shapes do not define a vector space in a straightforward way, most shape optimization algorithms perform a descent algorithm based on shape derivatives. The notion of the shape derivative is based on shape sensitivities dating back to Hadamard's visionary publication \cite{Hadamard1908}. These sensitivities are based on shape variations, where in most cases the perturbation of identity is used and in some cases the more general speed method; for both we refer to the monographs mentioned above. If the set of admissible shapes defines a Riemannian manifold of sufficiently smooth shapes, the notion of a covariant derivative can be used as sensitivities \cite{VHS-shape-Riemann}. The paper \cite{Schmidt-Schulz-2023} discusses shape derivatives of second order and their usage in shape optimization algorithms within a vector space framework for the set of perturbations. 

Many applications of shape optimization methods involve a state equation formed by partial differential equations, which must be solved in a computational domain characterized by the shape under interest. Thus the objective criterion implicitly depends on this state equation. This dependency can be evaluated numerically by perturbing each mesh point of the shape, leading to so-called mesh sensitivities. This leads to a tremendous usage of memory and computing resources. A more efficient and mostly used alternative to treat this implicit dependency is applying a Lagrangian technique involving Lagrange multipliers. A detailed discussion of the Lagrangian approach can be found in \cite{Kunisch2008} and in particular a foundational discussion in the shape context in \cite{laurain2016distributed}. The adjoint approach necessitates the solution of the adjoint equation in addition to the state equation.

After this conceptual discussion of the literature background, we outline the \textbf{main novelty of this paper:} based on probabilistic representations for the solutions of semilinear elliptic partial differential equations, we derive expressions for the shape derivative that do not require Lagrangian multipliers or adjoint equations. Thus, we call this a direct method. On a computational level, the evaluation of these expressions is based on Monte-Carlo simulation and has the potential to be more efficient than mesh sensitivities or the Lagrangian approach. Since it is direct, our approach also does not necessitate a computational mesh, which means that we provide a mesh-free method for the evaluation of the shape derivative, although a rather general elliptic PDE defines the state equation for the shape optimization problem under investigation. 

The remainder of this section is devoted to a literature review of further aspects of our approach.

\textbf{Probabilistic methods regarding Eulerian shape derivatives.}
The Eulerian shape derivative of a solution of a second order \textit{parabolic} partial differential equation has been investigated from a probabilistic perspective in the literature. Specifically, the linear case is considered in \cite{Karoui_bdry} where a probabilistic representation of the Eulerian shape derivative is derived as boundary sensitivity result of suitable diffusion processes. 
Moreover, for a special type of linear equations; namely Poisson type equations with constant source term and vanishing boundary condition, the solution identifies with the expectation of the first exit time of an appropriate diffusion. In this case the Eulerian shape derivative corresponds to the $L^1$-derivative of the corresponding exit times.
In this context, for the linear parabolic case an asymptotic equivalence between boundary perturbations and the simulation error of the corresponding exit times is discussed in \cite{Gobet2010}. 
In \cite{Dokuchaev2004, Dokuchaev2015} bounds for the $L^1$-distance of the exit times from two bounded domains are provided. This might be seen as a first step towards the linear elliptic case. Nonetheless, the setting does not precisely match that of shape calculus, and the bounds are not sufficient to establish shape derivatives.
To the best of our knowledge, a probabilistic representation of shape derivatives for \textit{semilinear elliptic} PDE, as considered in this paper, has not yet been investigated in the literature.

\textbf{Probabilistic methods regarding shape functionals.}
A probabilistic representation of the shape derivative of a shape functional as derived in this work is a novel contribution to the literature. We are only aware of two related contributions that are connected in a broad sense: \cite{GatarekSokolowski1992} discuss the derivative of a shape functional consisting of the expectation of an $L^2$-norm of solutions of parabolic and hypoelliptic stochastic evolution equations, and \cite{Rousset2010} provides a probabilistic interpretation of shape functional derivatives in the context of quantum groundstates.

\textbf{Shape optimization under uncertainty.}
There is also a literature on shape optimization under uncertainty, where shape optimization problems are augmented by exogenous random shocks; we refer to \cite{Allaire2015} for a general overview.
For instance, these shocks may occur in the form of randomness in the target functional or PDE coefficients; see e.g.\ \cite{Allaire2005,Conti2009,Dambrine2015, GeiersbachLoayzaRomeroWelker2021,Guan2023,Martinez2015};
or as random geometric disturbances, see e.g.\ \cite{Chen2011, Schillings2011}.
We emphasize that, by contrast, in this paper we investigate classical shape derivatives in the absence of any random perturbations. Probabilistic arguments and methods are merely used as mathematical tools to analyze these (deterministic) problems.

\textbf{Outline.}
The paper is organized as follows. In Section~\ref{Section:pde_constrained_shape_derivatives} we introduce the general setting, provide the basic definitions concerning shape derivatives and give an informal description of our main results.

Section~\ref{Section:ProbRepresentationShapeDerivative} provides the stochastic framework and the rigorous mathematical analysis for our main results: the probabilistic representation of the Eulerian shape derivative for the solution of a semilinear elliptic PDE in Theorem~\ref{theorem:ProbabilisticRepresentationShapeDerivative} and the probabilistic representation of the shape derivative for a shape functional in Theorem~\ref{theorem:ProbabilisticRepresentationShapeFunctionalDerivative}.
Section~\ref{section:ProofProbabilisitcRepresentation} provides the proof of Theorem~\ref{theorem:ProbabilisticRepresentationShapeDerivative}.

In Section~\ref{Section:SimulationProbShape} and Section~\ref{Section:TaylorTest} we discuss a numerical implementation of our probabilistic representation of shape derivatives. In particular, we propose a mesh-free method as well as a hybrid approach based on the results of Section~\ref{Section:ProbRepresentationShapeDerivative}. We conclude with numerical results in a benchmark application verifying the accuracy of our methodology using a Taylor test.

Readers primarily interested in shape optimization and the application of our results may focus on Sections~\ref{Section:pde_constrained_shape_derivatives}, \ref{Section:SimulationProbShape} and \ref{Section:TaylorTest} and the statements of Theorem~\ref{theorem:ProbabilisticRepresentationShapeDerivative} and Theorem~\ref{theorem:ProbabilisticRepresentationShapeFunctionalDerivative}.

\section{Discussion of Main Results}\label{Section:pde_constrained_shape_derivatives}
In the following we provide the fundamental definitions concerning shape derivatives with PDE constraints used throughout this article. Thus let $\Omega\subset\R^d$ be a bounded domain,\footnote{We use standard terminology and refer to a subset of $\R^d$ as a domain if it is open and connected.} and denote by
\[
    \pertSpace\defined\big\{\dir\colon\R^d\rightarrow\R^d\,|\,\dir\text{ is of class $\C^2$ and bounded}\big\} 
\]
the space of admissible distortions. Each $\dir\in\pertSpace$ canonically induces a perturbed domain $\pertDomain$ via
\[
    \pertDomain \defined \big\{ x \in\R^d \,\big|\, T_\varepsilon^\dir (x)\in\Omega \big\}
\]
where $\varepsilon$ is a distortion factor and $T_\varepsilon^\dir\colon x\mapsto x+\varepsilon\dir(x)$ denotes the shift operator in direction $\dir$, also referred to as a perturbation of the identity; see e.g.\ \cite[Chapter 3]{Delfour_Shape}. For readers from a shape optimization background, we point out that the above definition of $\pertDomain$ as the pre-image of the perturbation may seem unusual, as the shape optimization literature usually defines $\pertDomain$ as the image of the perturbation. This is not a substantial difference, but should be taken into account to avoid confusion. Our definition of $\pertDomain$ avoids technical difficulties in the proofs for Section~\ref{Section:ProbRepresentationShapeDerivative}.

\begin{remark}\label{remark:Epsilon0}
    Note that, for each fixed $\dir\in\pertSpace$, since $\dir$ is of class $\mathcal{C}^2$ and $\Omega$ is bounded, there exists $\varepsilon^\dir_0>0$ such that $T_\varepsilon^V$ is bijective for all $\varepsilon\in[-\varepsilon^\dir_0,\varepsilon^\dir_0]$. Throughout this article, whenever we use a direction $\dir\in\pertSpace$ and a distortion factor $\varepsilon$, we implicitly assume that $\varepsilon\in[-\varepsilon^\dir_0,\varepsilon^\dir_0]$. For later reference we define the corresponding bounded hold-all domain
    \begin{equation}\label{def:HoldAllDomain}
        \HoldAll^\dir \defined \bigcup_{\varepsilon \in [-\varepsilon^\dir_0,\varepsilon^\dir_0]} \pertDomain
    \end{equation}
    and write $\HoldAll=\HoldAll^\dir$ and $\varepsilon=\varepsilon^\dir$ when $\dir$ is clear from the context; similarly we denote $\|\cdot\|_\infty \defined \|\cdot\|_{L^\infty (\HoldAll)}$.
    If the domain of a function is smaller, the norm is taken as that of the corresponding restriction.\close
\end{remark}
Let $\SemiDiffOp$ denote the differential operator defining the PDE constraint.
Under suitable conditions, specified in detail in Section~\ref{Section:ProbRepresentationShapeDerivative} below, for each distortion $\dir\in\pertSpace$ and sufficiently small distortion factor $\varepsilon$ there exists a unique solution $u_\varepsilon^\dir\colon\pertDomain\rightarrow\R$ of the PDE
\begin{align}\label{PDE with general elliptic differential operator}
    \SemiDiffOp[u_\varepsilon^\dir] =&\ 0 \qquad \text{on } \pertDomain,\notag\\
    u_\varepsilon^\dir =&\ g  \qquad \text{on } \partial\pertDomain.
\end{align}
Given a function $\phi\colon\R^d\times\R\rightarrow\R$ of class $\mathcal{C}^1$, we may thus consider the functional
\begin{align}\label{Def:Shape_Functional}
    \Phi\colon u\mapsto\int_{\text{dom}(u)}\phi\big(x,u(x)\big)\dLebesgue{d}{x}
\end{align}
where $\text{dom}(u)$ denotes the domain of definition of $u$, provided the integral is well-defined.\footnote{Note that the integral depends on $u$ both via the integrand and via the integration region $\text{dom}(u)$. For instance, since $\text{dom}(u^\dir_\varepsilon)=\Omega^\dir_\varepsilon$ we have
\[
    \Phi(u_\varepsilon^\dir) = \int_{\pertDomain}\phi\big(x,u^\dir_\varepsilon(x)\big)\dLebesgue{d}{x}\qquad\text{for all }\dir\in\pertSpace\ \text{and }\varepsilon\in[-\varepsilon_0,\varepsilon_0].
\]}
With this notation, we can state the definitions of the shape derivative of the PDE solution $u$ and the shape derivative of the functional $\Phi$.
\begin{definition}\label{Def:ShapeDerivative}
    Suppose that for each $\dir\in\pertSpace$ the Gateaux derivative
    \begin{equation*}
        \D u[\dir](x) \defined \frac{\de}{\de \varepsilon}\Big|_{\varepsilon=0} u_\varepsilon^\dir(x)\qquad \text{exists for every }x\in\overline{\Omega}.
    \end{equation*}
    Then the Euler shape derivative of $u$ (briefly, the shape derivative of $u$) is defined as the map\footnote{The domain of $\D u[\dir]$ does not depend on $\dir$ since $\Omega = \Omega_0^\dir$ for any $\dir\in\pertSpace$. $\mathfrak{B}(\overline{\Omega})$ denotes the space of Borel measurable functions $\overline{\Omega}\to\R$.}
    \begin{equation*}
        \D u\colon \pertSpace \rightarrow \mathfrak{B}(\overline{\Omega}),
        \quad \dir\mapsto \D u[\dir].
    \end{equation*}
    If the corresponding limit exists, the shape derivative of the functional $\Phi$ is defined as the map
    \begin{align}
        &\D\Phi \colon \pertSpace \rightarrow \R,\\
        &\D\Phi [\dir] \defined\frac{\de}{\de \varepsilon}\Big|_{\varepsilon=0} \Phi(u_\varepsilon^\dir). \closeEqn
    \end{align} 
\end{definition}

\begin{remark}
    For the sake of completeness, we first point out that the shape derivative typically considered in shape optimization is based on a push-forward definition of the perturbed domain;
    \begin{equation*}
        \mathcal{D}\Phi [\dir] \defined \lim_{\varepsilon\to 0} \tfrac{1}{\varepsilon} \big(\Phi\big(T_\varepsilon^\dir(\Omega)\big) - \Phi(\Omega)\big) \quad\text{for every}\quad\dir\in\pertSpace .
    \end{equation*}
    Since our definition is based on the corresponding pre-image, these two definitions coincide only up to the sign, i.e.
    \begin{equation*}
        \mathcal{D}\Phi[\dir] = -\D\Phi[\dir],\quad \dir\in\pertSpace .
    \end{equation*}
    Second, the recent literature on shape optimization also investigates the weaker notion of semi-derivatives of shape functionals, see e.g.\ \cite[Definition 3.2]{Delfour_Shape} or \cite[Section 2.5.1]{SturmDissertation}. Under the regularity conditions of this article, the notion of a semi-derivative does not add generality since the shape derivatives in the sense of Definition~\ref{Def:ShapeDerivative} exist.\close
\end{remark}
In Section~\ref{Section:ProbRepresentationShapeDerivative} we establish, under suitable conditions, existence of the shape derivatives in the sense of Definition~\ref{Def:ShapeDerivative}, and we provide probabilistic representations of $\D u[\dir]$ and $\D\Phi [\dir]$.
Importantly, the probabilistic representation of $\D\Phi[\dir]$ is a boundary representation that is amenable to direct Monte Carlo simulation.
In the following, we spell this out in more detail:
To wit, consider a semilinear elliptic convection-diffusion equation of the form\footnote{See Section~\ref{Section:ProbRepresentationShapeDerivative} for the exact formulation including all necessary assumptions.}
\begin{align}\label{PDE Anwender Darstellung}
    v(x)^\top\nabla u(x)+\divergence\big(K\nabla u\big)(x)+f\big(x,u(x)\big)=&\ 0, \quad & x&\in\Omega,\notag\\
    u(x) =&\ g(x), \quad &x&\in\partial\Omega.
\end{align}
Theorem \ref{theorem:ProbabilisticRepresentationShapeDerivative} provides a probabilistic representation of $\D u[\dir]$ via 
\begin{align*}
    \D u[\dir] (x) = \E\bigg[\exp\Big(\int_0^{\tau^x} \Diff_u f\big(X_s^x,u(X_s^x)\big)\de s\Big) \big\langle\nabla u-\nabla g, \dir\big\rangle\big(X_{\tau^x}^x\big)\bigg],
\end{align*}
where $X^x$ is a suitable diffusion process with first exit $\tau^x$ from $\Omega$.
Based on this, Theorem \ref{theorem:ProbabilisticRepresentationShapeFunctionalDerivative} yields a probabilistic representation of the shape derivative $\D\Phi[\dir]$ via
\begin{multline*}
    \D\Phi [\dir] = \constant{+}\E\Big[\big\langle\dir,\nabla u-\nabla g\big\rangle (\enX^+)\Big] - \constant{-} \E\Big[\big\langle\dir,\nabla u-\nabla g\big\rangle (\enX^-)\Big]\\
    -\int_{\partial\Omega} \big\langle \dir , \phi(\cdot, u)n \big\rangle \dSph{d-1}.
\end{multline*}
Here $\enX^\pm$ are random variables taking values in $\partial\Omega$, the constants $\constant{\pm}\geq 0$ are given explicitly by 
\[
    \constant{\pm}\defined\pm\int_{\Omega^\pm}\Dphiu\phi\big(x,u(x)\big)\dLebesgue{d}{x}
\]
and $\Dphiu\phi$ denotes the derivative w.r.t.\ the second component of $\phi$. In particular, except possibly for the computation of the constants $\constant{\pm}$, the probabilistic representation of $\D\Phi [\dir]$ depends only on evaluations of $u$ at the boundary of $\Omega$. Moreover, $\constant{\pm}$ in turn do not depend on $\dir$, and hence have to be determined only once to obtain $\D\Phi[\dir]$ for all $\dir\in\pertSpace$.
Note further that, since $\enX^\pm$ takes values in $\partial\Omega$, the expectations in the probabilistic representation represent boundary integrals; specifically,
\[
\E\Big[\big\langle\dir,\nabla u-\nabla g\big\rangle (\enX^\pm)\Big] = \int_{\partial\Omega} \big\langle\dir,\nabla u-\nabla g\big\rangle \de\nu^\pm
\]
where $\nu^\pm$ denotes the distribution of $\enX^\pm$.
Finally, observe that the PDE coefficients $v$, $K$ and $f$ from \eqref{PDE Anwender Darstellung} do not appear explicitly in the probabilistic representation of the shape derivative $\D\Phi$; they are implicit in $u$ and the dynamics of $\enX^\pm$.

In Section~\ref{Section:ProbRepresentationShapeDerivative} we postulate standing assumptions and formally present rigorous statements of our main results. Readers with a focus on shape optimization and applications of our probabilistic representations may focus on the statements of Theorem~\ref{theorem:ProbabilisticRepresentationShapeDerivative} and Theorem~\ref{theorem:ProbabilisticRepresentationShapeFunctionalDerivative} and move on to Section~\ref{Section:SimulationProbShape}.

\section{Probabilistic Representation of Shape Derivatives}\label{Section:ProbRepresentationShapeDerivative}
This section presents our main result, a probabilistic representation of the shape derivative $\D\Phi$ of $\Phi$. To obtain this, we proceed in two steps:
First, we split the shape derivative $\D \Phi$ into a boundary integral and a term that involves (an integral over the entire domain $\Omega$ of) the shape derivative $\D u$ of $u$. Second, we use a Feynman-Kac representation to transform the term involving $\D u$ into a probabilistic boundary representation.

To begin with, we state the relevant regularity conditions. For $k\in\nat_0$ and $\gamma\in(0,1)$ we say that a function is of class $\HoelderSpace{k}{\gamma}$ if it is $k$-times continuously differentiable with $\gamma$-Hölder continuous derivatives, and we denote the space of functions of class $\HoelderSpace{k}{\gamma}$ on $\Omega$ by $\HoelderSpace{k}{\gamma}(\Omega)$. We refer to Definition~\ref{Definition Hoelder Space} in the Appendix or \cite[p.52]{GilbargTrudinger} for further details.

\begin{domain_assumption}\label{assumption:domain}
	The domain $\Omega$ is bounded and its boundary $\partial\Omega$ is of class $\HoelderSpace{2}{\gamma}$ for some $\gamma\in(0,1)$, i.e.\ $\partial\Omega$ admits a representation via maps of class $\HoelderSpace{2}{\gamma}$. \close
\end{domain_assumption}
This condition is standard in the literature on elliptic PDEs; we refer to \cite[p.64]{friedman1964} and to \cite[p.94]{GilbargTrudinger}.\footnote{Note that in \cite[p.64]{friedman1964} this property is called \textit{property $\Bar{E}$} and imposes the equivalent condition that $\partial\Omega$ can be represented as a graph of a function of class $\HoelderSpace{2}{\gamma}$.}
Note that \hyperref[assumption:domain]{\emph{(Dom)}} implies in particular that $\Omega$ satisfies an exterior sphere condition, and that the outer normal vector field $n\colon\partial\Omega\rightarrow\R^d$ is well-defined; see e.g.\ \cite[Proposition 10.39]{LeeSmoothManifolds}. 

Concerning the PDE constraint, we consider the second-order differential operator
\begin{equation}\label{Def_SemilinearPDEOperator}
	\SemiDiffOp [u]\defined\InfGen[u] + f(\cdot, u)
\end{equation}
where $\InfGen$ denotes the linear elliptic operator
\begin{equation}\label{eqnInfGen}
	\InfGen[u](x)=\mu(x)^\top\Diff u(x)+\tfrac{1}{2}\trace\bigl[\sigma(x)\sigma(x)^\top\Diff^2u(x)\bigr].
\end{equation}
Thus the PDE constraint is given by
\begin{align}\label{eqnSemilinearPDE}
	\InfGen[u] + f(\cdot, u) =&\ 0 \qquad \text{on } \Omega,\notag \\
	u =&\ g         \qquad \text{on } \partial\Omega.
\end{align}
To ensure existence and uniqueness of solutions to \eqref{eqnSemilinearPDE} we impose the following standard assumptions on the PDE coefficients $\mu$, $\sigma$, $f$ and $g$.
\begin{pde_assumption}\label{assumption:LadyzhenskayaPardoux} 
	$\mu\colon\R^d\rightarrow\R^d$ is of class $\C^1$, $\sigma\colon\R^d\rightarrow\R^{d\times d}$ is of class $\C^2$, and $\sigma\sigma^\top$ is strictly elliptic on $\Omega$.\footnote{The matrix $\sigma\sigma^\top (x)$ is positive definite for each $x\in\Omega$ and the eigenvalues are uniformly bounded away from zero; see \cite[p.31]{GilbargTrudinger}.}
	Moreover $f\colon\R^d\times\R\rightarrow\R$ is of class $\C^{1,2}$ and satisfies\footnote{The condition $\Diff_u f\leq 0$ is used in the proof of Theorem \ref{theorem:ProbabilisticRepresentationShapeFunctionalDerivative}. The representation of the shape derivative $\D u$ as stated in Theorem \ref{theorem:ProbabilisticRepresentationShapeDerivative} is valid if $\Diff_u f$ is merely bounded above.} $\Diff_u f\leq 0$, and there is a constant $C \geq 0$ such that
	\begin{align*}
		\sign (u) f(x,u) \leq C, \quad (x,u)\in\Omega\times\R.
	\end{align*}
	Finally $g\colon\R^d\rightarrow\R^d$ is of class $\HoelderSpace{2}{\gamma}(\overline{\Omega})$ for some $\gamma\in (0,1)$. \close
\end{pde_assumption}

\MW{
	Now we introduce the stochastic setting: Let $(\mathfrak{X},\mathfrak{A},\mathfrak{F},\prob)$ be a filtered probability space, where we assume that the filtration $\mathfrak{F}$ is generated by a $d$-dimensional Brownian motion $W$ augmented by all $\prob$-nullsets. The diffusion associated\footnote{Existence and uniqueness of $X^x$ is ensured under much weaker conditions than our standing assumptions, see e.g.\ \cite[Theorem 5.2.5]{KaratzasShreve}.} to $\InfGen$ is characterized by 
	\begin{equation}\label{Def:ForwardDiffusion}
		X_t^x = x + \int_0^t \mu(X_s^x)\de s + \int_0^t \sigma(X_s^x)\de W_s, \qquad t\geq 0,
	\end{equation}
	where $x\in\R^d$ is fixed, and we denote the first exit time of $X^x$ from $\Omega$ by
	\[
	\tau^x \defined \tau^x_{\Omega} \defined\inf\big\{t\geq0\,|\,X_t^{x}\notin \Omega\big\}.
	\]
	\begin{holdAll_exp_assumption}\label{assumption:holdAll_exp}
		Assume there is $\rho > 2$ such that $\sup_{x\in\HoldAll} \E [\exp(\rho\tau_{\HoldAll}^x)]<\infty$. \close
	\end{holdAll_exp_assumption}
	In the following, we take $\gamma\in(0,1)$ such that both \hyperref[assumption:domain]{\emph{(Dom)}} and \hyperref[assumption:LadyzhenskayaPardoux]{\emph{(PDE)}} are satisfied, and assume \hyperref[assumption:holdAll_exp]{\emph{(E)}}. Under these assumptions, we have the following well-known result:
	\begin{proposition}\label{PropExistenceOfPDESolution}
		The PDE \eqref{eqnSemilinearPDE} admits a unique solution $u\in\HoelderSpace{2}{\gamma}(\overline{\Omega})$. \close
	\end{proposition}
	\begin{proof}	
		By \cite[Theorem 15.10]{GilbargTrudinger} there exists at least one solution $u$ of \eqref{eqnSemilinearPDE} of class $\HoelderSpace{2}{\gamma}(\overline{\Omega})$. Uniqueness follows from a viscosity argument as in \cite[Section 6]{PardouxRandomTimeBSDE}. More precisely, for $x\in\overline{\Omega}$ consider the backward stochastic differential equation (BSDE)
		\begin{align*}
			Y_t^x = g(X_{\tau^x}^x) + \int_{t\wedge\tau^x}^{\tau^x} f(X_s^x, Y_s^x)\de s - \int_{t\wedge\tau^x}^{\tau^x} Z_s^x \de W_s, \qquad t\in[0,\tau^x]
		\end{align*}
		Observe that, by Assumption \hyperref[assumption:holdAll_exp]{\emph{(E)}}, \cite[Theorem 3.4]{PardouxRandomTimeBSDE} applies and hence there is a unique solution $(Y^x,Z^x)$ of the above BSDE for each $x\in\overline{\Omega}$. Finally, \cite[Theorem 6.5]{PardouxRandomTimeBSDE} implies uniqueness via the Feynman-Kac correspondence $u(x) = Y_0^x$, $x\in\Omega$.
	\end{proof}
}
We next present our first main result, a probabilistic representation of the shape derivative $\D u[\dir]$. This may be seen as a semilinear elliptic version of \cite[Theorem 2.2]{Karoui_bdry}, where the parabolic linear case with a bounded terminal time is investigated. While the general strategy of the proof is similar to that in \cite{Karoui_bdry}, several complications arise due to the nonlinearity of \eqref{eqnSemilinearPDE} and the fact that elliptic equations give rise to BSDEs on unbounded time horizons.
\begin{theorem}[Probabilistic Representation of Shape Derivative]\label{theorem:ProbabilisticRepresentationShapeDerivative}
	Let $\dir\in\pertSpace$. Then the shape derivative $\D u[\dir]$ exists, that is for all $x\in\overline{\Omega}$, the map\footnote{Recall that $u_\varepsilon^\dir$ denotes the solution of \eqref{PDE with general elliptic differential operator} where the domain is shifted in direction $\dir$ by $\varepsilon$.} $\varepsilon\mapsto u_\varepsilon^\dir(x)$ is differentiable at $\varepsilon=0$, and we have the probabilistic representation
	\begin{align*}
		\D u[\dir] (x) = \E\bigg[\exp\Big(\int_0^{\tau^x} \Diff_u f\big(X_s^x,u(X_s^x)\big)\de s\Big) \big\langle\nabla u-\nabla g, \dir\big\rangle\big(X_{\tau^x}^x\big)\bigg]. \closeEqn
	\end{align*}
\end{theorem}
\begin{proof}
	The proof, and all auxiliary results required for it, are provided in Section~\ref{section:ProofProbabilisitcRepresentation}; the assertion is then an immediate consequence of Theorem~\ref{theorem:StronConvergenceShapeDerivative}.
\end{proof}    
\begin{remark}\label{remark:ShapeDerivativeContinuous}
	In Theorem~\ref{theorem:StronConvergenceShapeDerivative} we in fact establish a stronger result: The map $\varepsilon\mapsto u_\varepsilon^\dir$ is differentiable at $\varepsilon=0$ in $(\C (\overline{\Omega}),\|\cdot\|_\infty)$, i.e. 
	\begin{align*}
		\lim_{\varepsilon\to 0} \big\|\tfrac{1}{\varepsilon} (u_\varepsilon^\dir - u) - \D u[\dir] \big\|_\infty = 0
	\end{align*}
	for any $\dir\in\pertSpace$.
	In particular the map
	\begin{equation*}
		x \mapsto \D u[\dir] (x)
	\end{equation*}
	is continuous.
	Since the probabilistic representation in Theorem~\ref{theorem:ProbabilisticRepresentationShapeDerivative} further implies that 
	\begin{align*}
		\D u [\dir] \leq \sup_{y\in\overline{\Omega}} \Big|\big\langle(\nabla u-\nabla g) (y), \dir (y) \big\rangle\Big| \leq C \|\dir\|_\infty,
	\end{align*}
	it follows that $\D [u]\colon (\pertSpace, \|\cdot\|_\infty)\rightarrow (\C (\overline{\Omega}),\|\cdot\|_\infty)$ is a bounded, linear operator. 
	\close
\end{remark}

As a direct illustration of Theorem~\ref{theorem:ProbabilisticRepresentationShapeDerivative} in a probabilistic context, we consider $L^1$-derivatives of exit times; these can be regarded as asymptotic extensions of the corresponding $L^1$-bounds provided by \cite{Dokuchaev2004, Dokuchaev2015}.
\begin{example}[$L^1$-derivative of exit times]\label{Example:L1DerivativesOfExitTimes}
	Suppose $\Omega\subseteq\R^d$ satisfies \hyperref[assumption:domain]{\emph{(Dom)}} and consider the problem\footnote{Here we assume that $\mu, \sigma$ satisfy \hyperref[assumption:LadyzhenskayaPardoux]{\emph{(PDE)}}; no further restrictions are imposed.}
	\begin{align*}
		\InfGen [u] + 1 = 0 \quad \text{on}\ \Omega,\qquad u=0\quad \text{on}\ \partial\Omega .
	\end{align*}
	If $x\in\Omega$ and $\dir\in\pertSpace$ is inflating, i.e.\ $\Omega\subseteq\pertDomain$ for all sufficiently small $\varepsilon>0$, then we have
	\begin{align*}
		\lim_{\varepsilon\to 0} \tfrac{1}{\varepsilon}\E\big[|\pertExit - \tau^x|\big] &= \lim_{\varepsilon\to 0} \tfrac{1}{\varepsilon} \big(u_\varepsilon^\dir (x) - u(x)\big) = \E \big[\langle\nabla u,\dir\rangle (X_{\tau^x}^x)\big].
	\end{align*}
	Analogously, if $\dir$ is deflating, i.e.\ $\Omega\supseteq\pertDomain$ for sufficiently small $\varepsilon>0$, then
	\begin{align*}
		\lim_{\varepsilon\to 0} \tfrac{1}{\varepsilon} \E\big[|\pertExit - \tau^x|\big] &= - \lim_{\varepsilon\to 0} \tfrac{1}{\varepsilon} \big(u_\varepsilon^\dir (x) - u(x)\big) = -\E\big[\langle\nabla u,\dir\rangle (X_{\tau^x}^x)\big].\closeEqn
	\end{align*}
\end{example}

The second main result of this article provides a probabilistic representation of the shape derivative $\D\Phi[\dir]$. Before we present this, the following result recalls the well-known connection between the shape derivative $\D u[\dir]$ of $u$ and the shape functional derivative $\D\Phi[\dir]$ as introduced in Definition~\ref{Def:ShapeDerivative}. In the literature, this result is also referred to as the Reynolds transport theorem. Note, however, that in shape calculus domain perturbations are typically defined as images under some perturbation of the identity; see e.g.\ \cite{Delfour_Shape, Sokolowski_Shape, SturmDissertation}. By contrast, in this article the distorted domains $\pertDomain$ are defined as pre-images of those mappings.

For the corresponding identity with reversed perturbations, we refer to \cite[p.2097]{HarbrechtShape}. Moreover, we mention \cite[Theorem 3.3]{Simon1980}, where the result is derived for a shape functional with integrand of the form $\phi (x,u(x)) = \widetilde{\phi}(u(x))$ under rather strong differentiability assumptions, and \cite[Section 4.4.1]{BittnerDissertation} or \cite[Section 2.31]{Sokolowski_Shape}, where it is derived for the shape functional $\phi(x,u(x)) = u(x)$.

The result as used in this paper reads as follows; for completeness, we provide a proof in Appendix~\ref{section:Supplements}.
\begin{proposition}\label{Prop:ShapeDerivativeSplitt}
	For any $\dir\in\pertSpace$ we have
	\[
	\D\Phi[\dir] = \int_\Omega\Dphiu\phi(\cdot,u)\D u[\dir]\dLeb{d} - \int_{\partial\Omega} \big\langle\dir, \phi(\cdot,u)n \big\rangle\dSph{d-1}.\closeEqn
	\]
\end{proposition}

Next observe that the inner product in the expectation in Theorem~\ref{theorem:ProbabilisticRepresentationShapeDerivative} is evaluated at the exit time of the diffusion, i.e.\ exclusively at points that are located on the boundary. This motivates a reformulation as a boundary integral.
Informally, this may be interpreted as collecting the information of initial points and trajectories in a scalar weight factor for each boundary point $y\in\partial\Omega$.

To make this precise, we use the theory of doubly stochastic Poisson processes, see e.g.\ \cite[Section II.1]{bremaud1981}, \cite[p. 3-15]{Grandell1976} or \cite[Section 3]{Lawrance1972}. Formally, consider an enlargement of $\partial\Omega\subset\R^d$: Let $\dagger \notin \overline{\HoldAll}$, set 
\begin{align}
	\enOmega \defined \partial\Omega \cup \big\{\dagger\big\}
\end{align}
and extend any map $\varphi\colon\partial\Omega\rightarrow\R$ to $\enOmega$ by setting $\varphi(\dagger)\defined 0$. Since $\dagger$ is isolated, this preserves continuity and smoothness properties. Next let $\zeta\colon\R^d\to[0,\infty)$ be given and introduce a family $\{\enX^x \,|\, x\in\Omega\}$ of $\enOmega$-valued random variables as follows: For each $x\in\Omega$ we set
\begin{align}\label{def - ExitOrKilledRandomVariable}
	\enX^{x} &\defined \enX_{\tau^x}^x = \begin{cases}
		X_{\tau^x}^x,& \qquad \tau^x<\killTime\\
		\,\dagger,& \qquad \tau^x\geq\killTime,
	\end{cases}
\end{align}
where $\enX^x$ denotes the killed process
\begin{align}\label{def - killed process}
	\enX_t^x &\defined \begin{cases}
		X_{t}^x,& \qquad t<\killTime\\
		\,\dagger,& \qquad t\geq\killTime,
	\end{cases} 
\end{align}
and
\[
\killTime\defined\inf\Big\{t\geq0\,\Big|\, \int_0^t \genericStochasticIntensity(X_r)\de r \geq E\Big\}
\]
with $E\sim\text{Exp}(1)$ independent of $\sigma(X_s^x \,|\, s\geq 0)$.
As before, $\tau^x$ denotes the exit time of $X^x$ from $\Omega$.
Thus $\enX^x$ represents the value of the process $X^x$ at the killing time or at the first exit from $\Omega$, whichever happens first; we refer to $\enX^x$ as exit-kill random variables.

We have the following result\footnote{This is a standard result from the theory of doubly stochastic processes. A proof of the first identity can be found in \cite[Lemma 4.1]{JeanblancDefaultRisk}, and a proof of the second in \cite[Lemma 4.3]{JeanblancDefaultRisk}.}.
\begin{lemma}\label{lemma - properties killed process}
	Let $x\in\Omega$ and $t\geq 0$.
	Then we have
	\begin{align}
		\prob\big[\killTime > t\,\big|\, \sigma(X_s^x \,|\, s\geq 0) \big] =&\, \prob\big[\killTime > t\,\big|\, \sigma(X_s^x \,|\, t\geq s\geq 0) \big] =\, \exp\Big(-\int_0^t \genericStochasticIntensity(X_r^x)\de r\Big),\\
		\E \big[\eta (\enX^x)\big] =&\, \E\Big[\exp\Big(-\int_0^{\tau^x} \genericStochasticIntensity(X_r^x) \de r\Big) \eta(X_{\tau^x}^x)\Big],
	\end{align}
	whenever $\eta\colon\R^d\rightarrow\R$ is such that $\E[|\eta(X^x_{\tau^x}|]<+\infty$.\close
\end{lemma}

We define
\[
\Omega^+ \defined \big\{x\in\Omega \,\big|\, \Dphiu\phi\big(x,u(x)\big)\geq 0\big\}, \qquad \Omega^- \defined \Omega\setminus\Omega^+
\]
and the probability densities\footnote{This assumes $\Omega^+$ and $\Omega^-$ both have positive $d$-dimensional Lebesgue measure; otherwise, one part of the construction is void.}
\[
\rho^\pm\colon\Omega^\pm\rightarrow [0,\infty);\quad x\mapsto \frac{\Dphiu\phi\big(x,u(x)\big)}{\int_{\Omega^\pm} \Dphiu\phi\big(r,u(r)\big)\dLebesgue{d}{r}} 
\]
with corresponding probability distributions
\begin{equation}\label{Def:InitialPointDistribution}
	\mu^{\pm}\big[\de x\big]\defined \rho^\pm (x)\dLebesgue{d}{x}
\end{equation}
on $(\Omega^\pm,\B(\Omega^\pm))$.

The second main result of this article now provides the probabilistic representation of the shape derivative $\D\Phi[\dir]$ of $\Phi$. This is the basis for the simulation approach in Sections~\ref{Section:SimulationProbShape} and \ref{Section:TaylorTest} below. 
\begin{theorem}[Probabilistic Representation of Shape Functional Derivative]\label{theorem:ProbabilisticRepresentationShapeFunctionalDerivative}
	Let $\enX^{X_0^\pm}$ denote the exit-kill random variables defined in \eqref{def - ExitOrKilledRandomVariable} with killing intensity $-\Diff_u f(X,u(X))$ and initial distributions $X_0^\pm \sim \mu^\pm$. Then for every $\dir\in\pertSpace$ we have
	\begin{multline*}
		\D\Phi[\dir] = \constant{+} \E\Big[\big\langle \dir,\nabla u-\nabla g\big\rangle \big(\enX^{X_0^+}\big)\Big] -\constant{-}\E\Big[\big\langle\dir ,\nabla u-\nabla g \big\rangle \big(\enX^{X_0^-}\big)\Big]\\
		- \int_{\partial\Omega}\big\langle\dir,\phi\big(\cdot,u\big)n\big\rangle\dSph{d-1}
	\end{multline*}
	where the constants $\constant{\pm} \geq 0$ are given by 
	\[
	\constant{\pm}\defined\pm \int_{\Omega^\pm}\Dphiu\phi\big(x,u(x)\big)\dLebesgue{d}{x}. \closeEqn
	\]
\end{theorem}
\begin{proof}
	Fix $\dir\in\pertSpace$ and recall from Proposition \ref{Prop:ShapeDerivativeSplitt} that 
	\begin{align*}
		\D\Phi[\dir] = \int_\Omega\Dphiu\phi(\cdot,u)\D u[\dir]\dLeb{d} - \int_{\partial\Omega} \big\langle\dir, \phi(\cdot,u)n \big\rangle\dSph{d-1}.
	\end{align*}
	Note in particular that the second summand already takes the asserted form. Concerning the first, the probabilistic representation of $\D u[\dir]$ in Theorem~\ref{theorem:ProbabilisticRepresentationShapeDerivative} and Lemma~\ref{lemma - properties killed process} yield
	\begin{align*}
		& \int_\Omega \Dphiu\phi\big(x,u(x)\big)\D u[\dir](x)\dLebesgue{d}{x}\\
		& \hspace*{0.5cm} = \int_\Omega\Dphiu\phi\big(x,u(x)\big)\E\Big[\beta_{\tau^x}\,\big\langle(\nabla u-\nabla g),\dir\big\rangle (X_{\tau^x}^x)\Big]\dLebesgue{d}{x}\\
		& \hspace*{0.5cm} = \int_\Omega\Dphiu\phi\big(x,u(x)\big)\E\Big[\big\langle\dir,(\nabla u-\nabla g)\big\rangle (\enX^x)\Big]\dLebesgue{d}{x}.
	\end{align*}
	Splitting the integral over $\Omega$ into integrals over $\Omega^\pm$ and rescaling with $\constant{\pm}$, it follows that
	\begin{multline*}
		\int_\Omega\Dphiu\phi\big(x,u(x)\big)\E\Big[\big\langle\dir,(\nabla u-\nabla g)\big\rangle (\enX^x)\Big]\dLebesgue{d}{x}\\
		= \constant{+}\int_{\Omega^+} \E\Big[\big\langle\dir,(\nabla u - \nabla g)\big\rangle\big(\enX^{x}\big)\Big]\mu^+ (\de x) - \constant{-}\int_{\Omega^-} \E\Big[\big\langle\dir,(\nabla u-\nabla g)\big\rangle \big(\enX^{x}\big)\Big]\mu^- (\de x)
	\end{multline*}
	where $\mu^\pm$ are given by \eqref{Def:InitialPointDistribution}.
	The remainder of the argument is analogous for $\mu^+$ and $\mu^-$, so let $\constant{-}=0$. Denoting the distribution of the exit-kill random variables $\enX^{X_0^+}$ by $\nu^{+}$,\footnote{Formally, this is a distribution on $(\enOmega,\B(\enOmega))$, but we only consider measurable sets $A\in\B(\partial\Omega)$.} we have for any $A\in\mathcal{B}(\partial\Omega)$
	\begin{align}\label{def:DistributionExitOrKilledVariablesWithRandomizedStart}
		\nu^+ [A] \defined\, & \prob \Big[\enX^{X_0^+}\in A\Big] = \int_{\Omega^+} \prob\big[\enX^{x}\in A\big]\mu^{+} (\de x).
	\end{align}
	Thus a monotone class argument implies that
	\begin{align}\label{prop distribution killed process}
		\int_{\partial\Omega} \eta(y) \nu^+ (\de y) = \int_{\Omega^+} \E\Big[ \eta(\enX^x)\Big] \mu^{+} (\de x)
	\end{align}
	for every $\eta\in L^1 (\partial\Omega, \nu^+)$ and hence
	\begin{align*}
		\int_{\Omega^+} \E\Big[\big\langle\dir,(\nabla u - \nabla g)\big\rangle\big(\enX^{x}\big)\Big]\mu^+ (\de x) & = \int_{\partial\Omega}\big\langle\dir,\nabla u-\nabla g\big\rangle (y)\,\nu^+(\de y) \\
		& = \E\Big[\big\langle \dir,\nabla u-\nabla g\big\rangle \big(\enX^{X_0^+}\big)\Big].
	\end{align*}
	This completes the proof.
\end{proof}

\begin{remark}
	We briefly recall the well-known connection between the PDE formulations \eqref{PDE Anwender Darstellung} ("convection-diffusion" notation) and \eqref{eqnSemilinearPDE}: Given \eqref{PDE Anwender Darstellung} with $K$ symmetric and positive definite everywhere, for every $x\in\R^d$ there is $\widetilde{\sigma}(x)\in\R^{d,d}$ such that $K(x)=\widetilde{\sigma}(x)\widetilde{\sigma}(x)^\top$. Setting
	\begin{align*}
		\sigma(x)&\defined\sqrt{2}\widetilde{\sigma}(x), & \mu (x)&\defined v(x) + \divergence(K)(x),
	\end{align*}
	yields the equivalent formulation \eqref{eqnSemilinearPDE}.
	We understand that \eqref{PDE Anwender Darstellung} satisfies \hyperref[assumption:LadyzhenskayaPardoux]{\emph{(PDE)}} if the equivalent formulation with $\mu,\sigma$ as defined above satisfy \hyperref[assumption:LadyzhenskayaPardoux]{\emph{(PDE)}}. \close
\end{remark}

\section{Proof of Theorem~\ref{theorem:ProbabilisticRepresentationShapeDerivative}}\label{section:ProofProbabilisitcRepresentation}
Throughout this section we assume that \hyperref[assumption:domain]{\emph{(Dom)}} and \hyperref[assumption:LadyzhenskayaPardoux]{\emph{(PDE)}} are satisfied.

\begin{proposition}\label{PROP:BoundaryRegularity}
	For every perturbation $\dir\in\pertSpace$ and any distortion factor $\varepsilon\in [-\varepsilon_0,\varepsilon_0]$ the perturbed domain $\Omega_\varepsilon^\dir$ is bounded and satisfies an exterior sphere condition.\footnote{For every $y\in\partial\Omega$ exists an open ball $\mathcal{U}$ satisfying $\mathcal{U}\cap\overline{\Omega}=\{y\}$, see e.g.\ \cite[p.27]{GilbargTrudinger}.} \close\end{proposition}
\begin{proof}
	Recalling Remark \ref{remark:Epsilon0}, we have that $T_\varepsilon^\dir\colon\Omega_\varepsilon^\dir\rightarrow\Omega$ is bijective and of class $\C^2(\R^d)$. Moreover, since
	\[
	\Diff T_\varepsilon^\dir = \mathcal{I}+\varepsilon\Diff\dir
	\]
	the inverse function theorem implies that $(T_\varepsilon^\dir)^{-1}$ is also of class $\mathcal{C}^2$.
	Hence Lemma~\ref{lemma:BoundaryPreservation} implies
	\begin{equation}\label{eqn:proofBdryRegularity}
		(T_\varepsilon^\dir)^{-1}(\partial\Omega) = \partial\big((T_\varepsilon^\dir)^{-1}(\Omega)\big) = \partial\pertDomain.
	\end{equation}
	We proceed by showing that the boundary of $\pertDomain$ is of class $\C^2$ via construction of the corresponding parametrizations; see \cite[p.94]{GilbargTrudinger}. Fix $y\in\partial\pertDomain$ and set $x\defined T_\varepsilon^\dir (y)\in\partial\Omega$, where we used \eqref{eqn:proofBdryRegularity}. Denote by $\psi^x$ the $\HoelderSpace{2}{\gamma}$ parametrization for $x\in\partial\Omega$, which by definition is bijective on an open ball $\mathcal{U}^x$ with center $x$ with inverse of class $\HoelderSpace{2}{\gamma}$. We define
	\[
	\psi^y\colon\mathcal{U}^y \rightarrow \R^d; \quad \psi^y \defined \psi^x\circ T_\varepsilon^\dir 
	\]
	where $\mathcal{U}^y \subseteq (T_\varepsilon^\dir)^{-1}(\mathcal{U}^x)$ is an open ball around $y$. Observing\footnote{Using the half-space notation $\R^d_+\defined \{x\in\R^d\,|\,x_d>0\}$ from \cite[p.9]{GilbargTrudinger}.} the following two inclusions
	\begin{align*}
		\psi^y\big(\mathcal{U}^y\cap\pertDomain\big) &\subseteq \psi^y\big((T_\varepsilon^\dir)^{-1}(\mathcal{U}^x)\cap\pertDomain\big) = \psi^x\big(\mathcal{U}^x\cap\Omega\big)\subseteq\R^d_+\\
		\psi^y\big(\mathcal{U}^y\cap\partial\pertDomain\big) & \subseteq  \psi^y\big((T_\varepsilon^\dir)^{-1}(\mathcal{U}^x)\cap (T_\varepsilon^\dir)^{-1}(\partial\Omega)\big) = \psi^x \big(\mathcal{U}^x\cap\partial\Omega\big)\subseteq\partial\R^d_+
	\end{align*}
	where for the second inclusion we used \eqref{eqn:proofBdryRegularity}. It follows that $\Omega_\varepsilon^\dir$ has boundary of class $\mathcal{C}^2$, hence in particular satisfies an exterior sphere condition.
\end{proof}

\begin{proposition}\label{PROP:ExistenceLadyzhenskaya}
	The PDE \eqref{eqnSemilinearPDE} admits a unique solution $u_\varepsilon^\dir$ of class $\C(\overline{\pertDomain})\cap\C^2(\pertDomain)$ on the perturbed domain $\pertDomain$.\close
\end{proposition}
\begin{proof}
	By Proposition \ref{PROP:BoundaryRegularity} the perturbed domain $\pertDomain$ satisfies an exterior sphere condition. Thus, \cite[Theorem 15.18]{GilbargTrudinger} yields existence of a solution $u_\varepsilon^\dir$ of class $\C(\overline{\pertDomain})\cap\C^2(\pertDomain)$. Uniqueness follows from the same viscosity argument as used in the proof of Proposition \ref{PropExistenceOfPDESolution}.
\end{proof}

Similar to \cite{Karoui_bdry}, for each $\varepsilon$ and $\dir\in\pertSpace$ we introduce the perturbed process
\begin{equation}\label{DEf:PerturbedProcess}
	\pertX \defined X^{\varepsilon,\dir,x}\defined T_\varepsilon^\dir (X^x) = X^x+\varepsilon\dir(X^x)
\end{equation}
as well as the associated first exit times
\begin{equation}\label{def:PertubedExit}
	\pertExit \defined \tau_\varepsilon^{\dir,x} \defined \inf \big\{t\geq 0\,|\ X^x_t \notin \pertDomain\big\}.
\end{equation}
We collect some properties of the perturbed process and exit times in the following lemma.

\begin{lemma}\label{lemma:PropertiesPerturbedNew}
	For any $x\in\overline{\Omega}$ the perturbed exit time satisfies
	\begin{equation*}
		\pertExit = \inf\big\{t\geq 0\,|\ \pertX_t \notin\Omega\big\}.
	\end{equation*}
	Moreover, we have
	\begin{equation*}
		\LimSup\sup_{t\geq 0} |\pertX_t - X_t^x| = 0
	\end{equation*}
	as well as
	\begin{equation*}
		\LimSup \E\big[|\pertExit-\tau^x|\big] = 0. \closeEqn
	\end{equation*}
\end{lemma}

\begin{proof}
	The alternative representation of $\pertExit$ follows directly from the definitions of $\pertX$, $\pertDomain$ and $\pertExit$ via
	\begin{equation*}
		\inf\big\{t\geq 0\,|\ \pertX_t \notin \Omega\big\} = \inf\big\{t\geq 0\,|\ T_\varepsilon^\dir(X^x_t) \notin \Omega\big\} = \inf\big\{t\geq 0\,|\ X^x_t \notin \big(T_\varepsilon^\dir\big)^{-1}(\Omega)\big\} = \pertExit
	\end{equation*}
	and the first convergence statement is immediate since for every $t\geq 0$ we have
	\begin{equation*}
		\pertX_t - X_t^x = \varepsilon\dir(X_t^x)
	\end{equation*}
	where $\dir$ is uniformly bounded.
	
	In order to establish the uniform $L^1$-convergence of exit times, we introduce auxiliary exit times $\tau_{\varepsilon,\pm}$ such that $\tau_{\varepsilon,-}\leq \pertExit\leq \tau_{\varepsilon,+}$ and demonstrate that their a.s.\ limit is $\tau^x$, where the convergence is monotone. Using a continuity argument and Dini's convergence theorem we infer the desired uniform convergence.
	
	We start by introducing the auxiliary exit times. Fix $x\in\overline{\Omega}$ and for ease of notation write $\tau_{\varepsilon}\defined\pertExit$ and $\tau\defined\tau^x$. Let $V_{\max} \defined \sup_{y\in\R^d}\|V(y)\|$ and introduce the sets
	\begin{align*}
		\Omega^{\varepsilon,+} \defined \big\{y\in\R^d\,|\, \dist(y,\Omega)\leq \varepsilon V_{\max}\big\} \quad\text{and}\quad \Omega^{\varepsilon,-} \defined \big\{y\in\Omega\,|\, \dist(y,\partial\Omega)\geq \varepsilon V_{\max}\big\}
	\end{align*}
	so by construction $\Omega^{\varepsilon,-}\subseteq\Omega\subseteq\overline{\Omega}\subseteq\Omega^{\varepsilon,+}$ as well as $\Omega^{\varepsilon,-}\subseteq\pertDomain\subseteq\Omega^{\varepsilon,+}$. This implies in particular that $\tau_{\varepsilon,-}\leq\tau_{\varepsilon}\leq\tau_{\varepsilon,+}$.
	
	\step{1} We show $\tau_{\varepsilon,\pm}\rightarrow\tau$ a.s. as well as in $L^p$ for any $p\geq 1$. 
	
	We first show that $\tau_{\varepsilon,-}\rightarrow\tau$ a.s. Since $\Omega^{\varepsilon_2,-}\subseteq\Omega^{\varepsilon_1,-}$ whenever $\varepsilon_1\leq\varepsilon_2$ it follows that $\tau_{\varepsilon_2,-}\leq\tau_{\varepsilon_1,-}$, i.e.\ $(\tau_{1/n,-})_n$ is an increasing sequence with upper bound $\tau$. Thus 
	\begin{equation*}
		\tau_- \defined \lim_{\varepsilon\rightarrow 0} \tau_{\varepsilon,-}\leq\tau.
	\end{equation*}
	By definition of $\tau_{\varepsilon,-}$ we have $0\leq\dist (X_{\tau_{\varepsilon,-}}^x, \partial\Omega)\leq\varepsilon V_{\max}$ and since $X^x$ has continuous paths, it follows that $\dist(X_{\tau_-}^x,\partial\Omega)=0$. Thus $X_{\tau_-}^x\in\partial\Omega$, whereas by definition of $\tau=\tau^x$ we have $\tau\leq\tau_-$; we conclude that $\tau_-=\tau$ a.s. 
	
	Next, to show that $\tau_{\varepsilon,+}\rightarrow\tau$ a.s.\ we set
	\begin{equation*}
		\tau_+ \defined\inf\big\{t\geq 0\,|\, X_t^x\notin\overline{\Omega}\big\}.    
	\end{equation*}
	As above, since $\Omega^{\varepsilon_1,+}\subseteq\Omega^{\varepsilon_2,+}$ for $\varepsilon_1\leq\varepsilon_2$ it follows that $(\tau_{1/n,+})_n$ is a decreasing sequence with lower bound $\tau^x$ and
	\begin{equation*}
		\lim_{\varepsilon\rightarrow 0} \tau_{\varepsilon ,+} = \inf_{\varepsilon >  0} \tau_{\varepsilon ,+} = \tau_+,
	\end{equation*}
	where the second identity is due to continuity of the paths of $X^x$ and
	\begin{equation*}
		\big\{X_t^x \notin \Omega^{\varepsilon,+}\big\} = \big\{\dist(X_t^x, \Omega) > \varepsilon V_{\max}\big\}.
	\end{equation*}
	We have established $\tau_+ \geq \tau$, and proceed by showing that the latter estimate a.s.~holds with equality. For this we use the strong Markov property\footnote{The solution of the forward SDE \eqref{Def:ForwardDiffusion} has the strong Markov property, since the coefficients $\mu$ and $\sigma$ are globally Lipschitz and bounded, see e.g.~\cite[Theorem 4.20]{KaratzasShreve}.} of $X^x$ and a Blumenthal $0$-$1$ argument. Let $(\mathcal{C}([0,\infty),\R^d),\B(\mathcal{C}([0,\infty),\R^d)))$ denote the path space and denote by $\pi$ the canonical projection process, i.e.
	\[
	\pi\colon [0,\infty)\times\mathcal{C}([0,\infty),\R^d) \rightarrow \R^d;\ (t,x)\mapsto x_t,
	\]
	with this we can compute
	\begin{align*}
		\prob[\tau =\tau_+] &= \E\Big[\prob_\tau[\tau=\tau_+]\Big] \\
		&= \E\Big[\prob_\tau\big[\text{For all }\varepsilon >0\text{ exists }t\in(0,\varepsilon)\text{ s.t. } X_{\tau+t}^x\notin \overline{\Omega} \big]\Big]\\
		&= \E\Big[\prob^{X_{\tau}^x}\big[\text{For all }\varepsilon >0\text{ exists }t\in(0,\varepsilon)\text{ s.t. } \pi_t \notin \overline{\Omega} \big]\Big]
	\end{align*}
	where in the third step we used the strong Markov property. Since $\Omega$ satisfies an external cone condition, for every $z\in\partial\Omega$ there is a cone $C_z$ with $\overline{\Omega}\cap C_z = \{z\}$. Denote by $\tau_{C_z}$ the first time $X^x$ leaves $\R^d\setminus C_z$, then
	\begin{align*}
		\prob^z \big[\text{For all }\varepsilon >0\text{ exists }t\in(0,\varepsilon)\text{ s.t. } \pi_t \notin \overline{\Omega}\big] & \geq \prob^z \big[\text{For all }\varepsilon >0\text{ exists }t\in(0,\varepsilon)\text{ s.t. } \pi_t \in C_z\big] \\ 
		&= \prob^z \big[\tau_{C_z} =0\big] = 1,
	\end{align*}
	where the last step is due to \cite[Corollary III.3.2]{BassDiffusionsAndEllipticOperators} and we make use of the fact that the complement of the cone $\R^d\setminus C_z$ satisfies an external cone condition. To establish convergence in $L^p$, set
	\begin{equation*}
		\Omega^* \defined \Omega^{\varepsilon_0,+} \cup \Omega^{\varepsilon_0,+}  
	\end{equation*}
	and note that $\Omega_{\varepsilon,\pm}\subset\Omega^*$ for any $\varepsilon\in [-\varepsilon_0,\varepsilon_0]$. Letting $\tau^*$ denote the first exit time of $X^x$ from $\Omega^*$, we have $\tau_{\varepsilon,\pm}\leq\tau^*$ and $\E[(\tau^*)^p]<\infty$, hence dominated convergence implies convergence in $L^p$ for any $p\geq 1$.
	
	\step{2} We demonstrate that
	\begin{equation*}
		\LimSup \E\big[|\pertExit-\tau^x|\big] = 0.
	\end{equation*}
	
	By \cite[Proposition 5.76]{Pardoux2014} for every $x\in\overline{\Omega}$ and any sequence $(x_n)_n\subset\overline{\Omega}$ s.t.\ $x_n\to x$ there exists $N_x \in \mathfrak{A}$ with $\prob [N_x] = 0$ such that $\tau_{\varepsilon,\pm}^{x_n} \to \tau_{\varepsilon,\pm}^x$ as $n\to\infty$ outside of $N_x$. Thus, for every $\varepsilon\in[-\varepsilon_0,\varepsilon_0]$ the map
	\begin{equation*}
		h_\varepsilon \colon \overline{\Omega} \rightarrow [0,\infty),\quad
		x \mapsto \E\big[\tau_{\varepsilon,+}^x - \tau_{\varepsilon,-}^x \big]
	\end{equation*}
	is continuous. The construction of $\tau_{\varepsilon,\pm}$ immediately implies that $h_\varepsilon$ is non-negative. Let $(\varepsilon_k)_k \subseteq [-\varepsilon_0,\varepsilon_0]$ be a monotone vanishing sequence. By the first step
	\begin{equation*}
		h_{\varepsilon_k} (x) = \E\big[\tau_{\varepsilon_k,+}^x - \tau_{\varepsilon_k,-}^x \big] = \E\big[\tau_{\varepsilon_k,+}^x - \tau^x\big] + \E \big[\tau^x - \tau_{\varepsilon_k,-}^x \big] \to 0
	\end{equation*}
	for every $x\in\overline{\Omega}$. Hence Dini's convergence theorem, see \cite[Theorem 7.13]{Rudin1976}, implies that $(h_\varepsilon)_\varepsilon$ converges uniformly.
	Hence
	\begin{equation*}
		\LimSup \E\big[|\pertExit-\tau^x|\big] 
		\leq \LimSup \E\big[\tau_{\varepsilon,+}^x - \tau_{\varepsilon,-}^x]
		\leq \LimSup \big|h_\varepsilon (x)\big| = 0 . \qedhere
	\end{equation*}
\end{proof}

\begin{theorem}\label{theorem:StronConvergenceShapeDerivative}
	For any $\dir\in\pertSpace$ we have
	\begin{align*}
		\LimSup\bigg|\frac{1}{\varepsilon}\Big(u_\varepsilon^\dir (x)-u(x)\Big) - \E\Big[\,\beta_{\tau^x}\big\langle\nabla u-\nabla g, \dir\big\rangle (X_{\tau^x}^x)\Big]\bigg| = 0
	\end{align*}
	where $u_\varepsilon^\dir (x)\defined g(x)$ for $x\in\Omega\setminus\pertDomain$ and
	\begin{equation*}
		\beta_t \defined \exp\Big(\int_0^t \Diff_u f\big(X_s^x,u(X_s^x)\big)\de s\Big),\quad t\in[0,\tau^x].\closeEqn
	\end{equation*}
\end{theorem}
\begin{proof}
	Let $\dir\in\pertSpace$ be fixed. Proposition~\ref{PROP:BoundaryRegularity} and Proposition~\ref{PROP:ExistenceLadyzhenskaya} imply that there exists a unique solution $u_\varepsilon^\dir\in \mathcal{C} (\overline{\pertDomain})\cap \mathcal{C}^2 (\pertDomain)$ of the Dirichlet problem \eqref{eqnSemilinearPDE} on $\pertDomain$, and for $x\in \pertDomain$
	\begin{equation}\label{eqn:ProofStrongConvergenceEQ1}
		u_\varepsilon^\dir(x) = \E\Big[g(X_{\pertExit}^x) + \int_0^{\pertExit} f\big(X_s^x, u_\varepsilon^\dir(X_s^x)\big)\de s\Big].
	\end{equation}
	Thus, we can express 
	\begin{align*}
		u_\varepsilon^\dir (x) - u(x) &= \E\Big[g(X_{\pertExit}^x) + \int_0^{\pertExit} f\big(X_s^x, u_\varepsilon^\dir(X_s^x)\big)\de s - u(x) \Big],\quad x\in\Omega.
	\end{align*}
	Upon extending $\beta$ from $[0,\tau^x]$ to $[0,\pertExit]$ via
	\begin{equation*}
		\beta_t \defined \exp\Big(\int_0^t \Diff_u f\big(X_s^x,\uExt (X_s^x)\big)\de s\Big),\quad t\in[0,\pertExit]
	\end{equation*}
	we have
	\begin{align*}
		& g(X_{\pertExit}^x) + \int_0^{\pertExit} f\big(X_s^x, u_\varepsilon^\dir(X_s^x)\big)\de s - u(x) \\
		& \hspace*{1.0cm} = \beta_{\pertExit} \Big(g\big(X_{\pertExit}^x\big) - g\big(\pertX_{\pertExit}\big)\Big) + \beta_{\pertExit\wedge\tau^x} \Big( u\big(\pertX_{\pertExit\wedge\tau^x}\big)-u\big(X_{\pertExit\wedge\tau^x}^x\big)\Big) + \theItoDifference
	\end{align*}
	with
	\MW{
		\begin{equation}\label{eqn:ProofSemilinearSplit}
			\begin{aligned}%
				\theItoDifference &\defined
				g\big(X_{\pertExit}^x\big) -\beta_{\pertExit} \Big(g\big(X_{\pertExit}^x\big) - g\big(\pertX_{\pertExit}\big)\Big) \\
				& \hspace*{1.0cm} - \beta_{\pertExit\wedge\tau^x} \Big( u\big(\pertX_{\pertExit\wedge\tau^x}\big)-u\big(X_{\pertExit\wedge\tau^x}^x\big)\Big)\\
				& \hspace{2.5cm} + \int_0^{\pertExit} f\big(X_s^x, u_\varepsilon^\dir(X_s^x)\big)\de s - u(x).
			\end{aligned}	
		\end{equation}
	}
	Here $\Tilde{u}\colon\HoldAll\rightarrow\R$ denotes a $\C^2$ extension of $u$, i.e.\ $\Tilde{u}|_{\overline{\Omega}} = u$, where $\HoldAll$ is defined in Remark~\ref{remark:Epsilon0}. This is feasible since Proposition~\ref{PropExistenceOfPDESolution} implies that $u$ is of class $\C^2$ on $\overline{\Omega}$ so the classical extension lemma, see e.g.~\cite[Lemma 2.20]{LeeSmoothManifolds} applies.
	Thus we have the upper bound
	\MW{
		\begin{equation}\label{eqn:ProofStrongConvergenceBoundOfClaim}
			\begin{aligned}
				& \bigg|\tfrac{1}{\varepsilon}\Big(u_\varepsilon^\dir (x)-u(x)\Big) - \E\Big[\,\beta_{\tau^x}\big\langle\nabla u-\nabla g, \dir\big\rangle (X_{\tau^x}^x)\Big]\bigg| \\
				& \hspace*{0.5cm}\leq \bigg|\E\Big[\, \tfrac{1}{\varepsilon}\Big( \beta_{\pertExit} \Big(g\big(X_{\pertExit}^x\big) - g\big(\pertX_{\pertExit}\big)\Big) + \beta_{\pertExit\wedge\tau^x} \Big( u\big(\pertX_{\pertExit\wedge\tau^x}\big)-u\big(X_{\pertExit\wedge\tau^x}^x\big)\Big)  \Big)  \\
				& \hspace*{2.5cm} - \beta_{\tau^x}\big\langle\nabla u-\nabla g, \dir\big\rangle (X_{\tau^x}^x)\Big]\bigg| + \tfrac{1}{\varepsilon}\big|\E[\theItoDifference]\big| \\
				& \hspace*{0.5cm}\leq \big|\E[\thegTerms]\big| + \big|\E[\theuTerms]\big| + \tfrac{1}{\varepsilon}\big|\E[\theItoDifference]\big|
			\end{aligned}
		\end{equation}
	}
	where
	\begin{align*}
		\thegTerms &\defined \tfrac{1}{\varepsilon} \beta_{\pertExit} \Big(  g\big(\pertX_{\pertExit}\big) - g\big(X_{\pertExit}^x\big) \Big) - \beta_{\tau^x}\big\langle\nabla g, \dir\big\rangle (X_{\tau^x}^x),\\
		\theuTerms &\defined \tfrac{1}{\varepsilon} \beta_{\pertExit\wedge\tau^x} \Big( u\big(\pertX_{\pertExit\wedge\tau^x}\big) - u\big(X_{\pertExit\wedge\tau^x}^x\big)\Big) - \beta_{\tau^x}\big\langle\nabla u, \dir\big\rangle (X_{\tau^x}^x).
	\end{align*}
	We proceed by showing that $\E[\thegTerms]\to 0$, $\E[\theuTerms]\to 0$ and $\tfrac{1}{\varepsilon}\E[\theItoDifference]\to 0$ as $\varepsilon\rightarrow 0$, uniformly with respect to $x\in\Omega$.
	To show that $\LimSup|\E[\thegTerms]|=0$ note that
	\begin{align*}
		|\thegTerms| &\leq \big|\beta_{\pertExit}\big| \Big|\tfrac{1}{\varepsilon}\Big( g\big(\pertX_{\pertExit}\big)-g\big(X_{\pertExit}^x\big) \Big)- \big\langle \nabla g, \dir\big\rangle (X_{\tau^x}^x)\Big| + \big|\beta_{\pertExit}- \beta_{\tau^x}\big| \Big|\big\langle \nabla g, \dir\big\rangle (X_{\tau^x}^x)\Big|\\
		&\leq \Big|\tfrac{1}{\varepsilon}\Big( g\big(X_{\pertExit}^x+\varepsilon\dir (X_{\pertExit}^x)\big)-g\big(X_{\pertExit}^x\big) \Big) - \big\langle \nabla g, \dir\big\rangle (X_{\tau^x}^x)\Big| + C \big|\beta_{\pertExit}- \beta_{\tau^x}\big|
	\end{align*}
	where we use the fact that $\beta_t \leq 1$, continuity of $\nabla g$ on $\R^d$ and boundedness of $\dir$. Since $\Diff^2 g$ is continuous, Lemma \ref{lemma:UniformConvergenceDirectionalDerivative} yields
	\begin{align*}
		\LimSup \E\bigg[\Big|\tfrac{1}{\varepsilon}\Big( g\big(X_{\pertExit}^x+\varepsilon\dir (X_{\pertExit}^x)\big)-g\big(X_{\pertExit}^x\big) \Big) - \big\langle \nabla g, \dir\big\rangle (X_{\tau^x}^x)\Big| \bigg] = 0.
	\end{align*}
	On the other hand, by definition of $\beta_t$
	\MW{
		\begin{equation}
			\begin{aligned}
				\big|\beta_{\pertExit}- \beta_{\tau^x}\big| &\leq \bigg| \exp\Big(\int_{\tau^x\wedge\pertExit}^{\tau^x} \Diff_u f\big(X_s^x, \uExt (X_s^x)\big)\de s\Big) - \exp\Big(\int_{\tau^x\wedge\pertExit}^{\pertExit} \Diff_u f\big(X_s^x, \uExt (X_s^x)\big)\de s\Big)\bigg| \\
				& = 1 - \exp\Big(\int_{\tau^x\wedge\pertExit}^{\tau^x\vee\pertExit} \Diff_u f\big(X_s^x, \uExt (X_s^x)\big)\de s\Big) \\
				& \leq 1 - \exp \big(-\|\Diff_u f\|_\infty |\pertExit - \tau^x|\big).
			\end{aligned}
		\end{equation}
	}
	Since $y\mapsto 1-\exp(-Cy)$ is concave, Jensen's inequality yields
	\begin{align*}
		\E \Big[\big|\beta_{\pertExit}- \beta_{\tau^x}\big|\Big] \leq 1 - \exp\Big(-C \E\big[|\pertExit-\tau^x|\big]\Big) \leq 1 - \exp\Big(-C \,\sup_{y\in\Omega}\,\E\big[|\tau_\varepsilon^y -\tau^y|\big]\Big),
	\end{align*}
	and using Lemma \ref{lemma:PropertiesPerturbedNew} it follows that
	\begin{align*}
		\LimSup \E\Big[\big|\beta_{\pertExit} - \beta_{\tau^x}\big|\Big] \leq 1 - \exp\Big(-C \,\LimSup \E\big[|\pertExit -\tau^x|\big]\Big) = 0.
	\end{align*}
	Combining the preceding two convergence statements, we have
	\MW{
		$\LimSup\big|\E[\thegTerms]\big| = 0.$
	}
	Analogously, we conclude that
	\MW{
		$\LimSup\big|\E[\theuTerms]\big| = 0$
	}
	where Lemma~\ref{lemma:UniformConvergenceDirectionalDerivative} applies to $\tilde{u}$.
	Thus to complete the proof it remains to demonstrate that
	\begin{align}\label{eqn:ProofStrongConvergenceUniformConvergenceVanisher}
		\LimSup\tfrac{1}{\varepsilon}\big|\E[\theItoDifference]\big| = 0.
	\end{align}
	This will be accomplished using Lemma~\ref{lemma:ProofStep1}, Lemma~\ref{lemma:ProofStep2}, Lemma~\ref{lemma:ProofStep3} and Lemma~\ref{lemma:ProofStep4} below.
	
	\begin{lemma}\label{lemma:ProofStep1}
		Using the notation in the proof of Theorem~\ref{theorem:StronConvergenceShapeDerivative}, there exist a constant $C>0$ (not depending on $x\in\Omega$) and a process $\theRemainder$ such that
		\begin{align*}
			u(\pertX_t) = \E_t\Big[g(\pertX_{\pertExit}) + \int_t^{\pertExit} f\big(\pertX_s,u(\pertX_s)\big)\de s \Big] - \E_t\Big[\int_t^{\pertExit}\theRemainder_s\de s\Big],\quad t\in[0,\pertExit]
		\end{align*}
		and
		\begin{align*}
			\sup_{x\in\Omega} \sup_{t\geq 0}|\theRemainder_t| \leq C\varepsilon . \closeEqn
		\end{align*}
	\end{lemma}
	Note that, in contrast to the classical Feynman-Kac representation of $u$ along $X^x$, i.e.
	\begin{align*}
		u(X_t^x) = \E_t\Big[g(X_{\tau^x}^x) + \int_t^{\tau^x} f\big(X_s^x,u(X_s^x)\big)\de s \Big], \quad t\in[0,\tau^x]
	\end{align*}
	Lemma~\ref{lemma:ProofStep1} provides a representation of $u$ along the \emph{perturbed} process $\pertX$ (equivalently, a representation of $u\circ T_\varepsilon^\dir$ along $X^x$).
	\begin{proof}[Proof of Lemma~\ref{lemma:ProofStep1}]
		For ease of notation, let $T\defined T_\varepsilon^\dir$, note that $\Diff T = \mathcal{I}_d + \varepsilon \Diff \dir$ and set
		\begin{align}\label{def:GAMMAandDTMatrix}
			\Gamma \defined \tfrac{1}{\varepsilon}\Big(\trace\big[\sigma^\top\Diff^2 T^j\sigma\big] \Big)_{j=1,\dots,d}
			= \Big(\trace\big[\sigma^\top\Diff^2\dir^j\sigma\big] \Big)_{j=1,\dots,d} \in\R^d 
		\end{align}
		where $T^j$ denotes the $j$\textsuperscript{th} coordinate of $T$. By It\=o and recalling that $u(\pertX_t)=u\circ T(X^x_t)$;
		\MW{
			\begin{equation}
				\begin{aligned}
					u(\pertX_t) - u (\pertX_0) &= \int_0^t \Diff u(\pertX_s) \Diff T^\top (X_s^x) \sigma (X_s^x) \de W_s \\
					& \hspace*{0.5cm} + \int_0^t \Diff u (\pertX_s) \Diff T^\top (X_s^x) \mu (X_s^x) + \Diff u (\pertX_s) \tfrac{1}{2}\varepsilon\Gamma (X_s^x) \\
					& \hspace*{1.5cm} + \tfrac{1}{2} \trace\Big(\big(\Diff T^\top \sigma\big)^\top (X_s^x)\Diff^2 u(\pertX_s)\big(\Diff T^\top\sigma\big)(X_s^x)\Big) \de s \\
					&= \int_0^t \Diff u(\pertX_s) \Diff T^\top (X_s^x) \sigma (X_s^x) \de W_s \\
					& \hspace*{0.5cm} + \int_0^t \Diff u (\pertX_s) \mu (\pertX_s) + \tfrac{1}{2}\trace \Big(\big(\sigma^\top\Diff^2u\sigma\big)(\pertX_s)\Big) + \theRemainder_s \,\de s\\
					&= \int_0^t \Diff u(\pertX_s) \Diff T^\top (X_s^x) \sigma (X_s^x) \de W_s \\
					& \hspace*{0.5cm} + \int_0^t \InfGen[u](\pertX_s)\de s + \int_0^t \theRemainder_s \,\de s,\quad t\in[0,\pertExit]
				\end{aligned}
			\end{equation}
		}
		where $\theRemainder=(\theRemainder_s)_s$ is given by
		\begin{multline}\label{def:theRemainder}
			\theRemainder_s \defined \Diff u(\pertX_s) \Big\{\Diff T^\top(X_s^x)\mu(X_s^x)-\mu(\pertX_s)+ \tfrac{1}{2}\varepsilon\Gamma (X_s^x)\Big\} \\
			+ \tfrac{1}{2} \trace\Big(\big(\Diff T^\top \sigma\big)^\top (X_s^x)\Diff^2 u(\pertX_s)\big(\Diff T^\top\sigma\big)(X_s^x)\Big) -\tfrac{1}{2}\trace \Big(\big(\sigma^\top\Diff^2u\sigma\big)(\pertX_s)\Big)
		\end{multline}
		on $\{s\leq\pertExit\}$ and $\theRemainder_s\defined 0$ otherwise.
		By Lemma~\ref{lemma:PropertiesPerturbedNew} the perturbed exit time $\pertExit$ coincides with the first exit time of $\pertX$ from $\Omega$. Thus taking conditional expectations, using Lemma~\ref{lemma:WaldLemmaForBrownianMotion} and the fact that $u$ solves \eqref{eqnSemilinearPDE} on $\Omega$ it follows that
		\begin{align*}
			u(\pertX_t) = \E_t\Big[g(\pertX_{\pertExit}) + \int_t^{\pertExit} f\big(\pertX_s,u(\pertX_s)\big)\de s\Big] - \E_t\Big[\int_t^{\pertExit}\theRemainder_s\de s \Big].
		\end{align*}
		It remains to establish the bound for $\theRemainder$.
		Expanding the first line of $\theRemainder$ and using \eqref{def:GAMMAandDTMatrix}
		\begin{align*}
			& \Diff u(\pertX_s)^\top\Big(\Diff T^\top (X_s^x)\mu(X_s^x)-\mu(\pertX_s)+\tfrac{1}{2}\varepsilon\Gamma(X_s^x)\Big) \\
			& \hspace*{0.5cm} = \Diff u(\pertX_s)^\top\big(\mu(X_s^x)-\mu(\pertX_s)\big) \\
			& \hspace*{1.5cm} + \varepsilon\Diff u(\pertX_s)^\top\Diff\dir(X_s^x)^\top\mu(X_s^x) + \tfrac{1}{2} \varepsilon \Diff u(\pertX_s)^\top \Gamma(X_s^x).
		\end{align*}    
		Since $\Diff u$ and $\Diff\dir$ are bounded and $\mu$ is continuous, we obtain
		\begin{equation*}
			\varepsilon\Diff u(\pertX_s)^\top\Diff\dir(X_s^x)^\top\mu(X_s^x) \leq \varepsilon\big\|\Diff u(\pertX_s)\big\| \big\|\Diff\dir(X_s^x)\big\| \big\|\mu(X_s^x)\big\| \leq C_1\varepsilon.
		\end{equation*}
		Moreover, since $\sigma$ and $\Diff^2\dir$ are continuous
		\begin{equation*}
			\varepsilon\Diff u(\pertX_s)^\top \Gamma(X_s^x) \leq \varepsilon\big\|\Diff u(\pertX_s)\big\| \Bigg(\sum_{j=1}^d \Big|\trace \Big(\big[\sigma^\top\Diff^2\dir^j\sigma\big](X_s^x)\Big) \Big|^2\Bigg)^{\frac{1}{2}} \leq C_2\varepsilon
		\end{equation*}
		and since $\mu$ is Lipschitz continuous
		\begin{align*}
			\Diff u(\pertX_s)^\top\big(\mu (X_s^x)-\mu (\pertX_s)\big) & \leq \big\|\Diff u(\pertX_s)\big\| \big\|\mu (X_s^x)-\mu (\pertX_s)\big\| \leq C_3\varepsilon
		\end{align*}
		where $C_1,C_2,C_3>0$ are constants that do not depend on $x$.
		We next provide an upper bound for the second line of \eqref{def:theRemainder}.
		Recalling \eqref{def:GAMMAandDTMatrix} and using elementary properties of the trace operator
		\begin{align*}
			& \trace\Big(\big(\Diff T^\top\sigma\big)^\top(X_s^x)\Diff^2 u(\pertX_s)\big(\Diff T^\top\sigma\big)(X_s^x) - \big(\sigma^\top\Diff^2 u\sigma\big)(\pertX_s)\Big) \\
			&\hspace*{0.5cm} = \trace\Big(\big(\sigma(X_s^x)\sigma^\top(X_s^x) -\sigma(\pertX_s)\sigma^\top(\pertX_s)\big)\Diff^2u(\pertX_s)\Big) + \varepsilon \trace(A)
		\end{align*}
		where
		\begin{align*}
			&A \defined \sigma^\top(X_s^x)\Diff^2u(\pertX_s)\big(\Diff\dir^\top\sigma\big)(X_s^x) + \big(\sigma^\top\Diff\dir\big)(X_s^x)\Diff^2u(\pertX_s)\sigma(X_s^x) \\
			& \hspace*{0.5cm} + \varepsilon\big(\sigma^\top\Diff\dir\big)(X_s^x)\Diff^2u(\pertX_s)\big(\Diff\dir^\top\sigma\big)(X_s^x).
		\end{align*}
		Since $\Diff^2 u$ is bounded and $\sigma$ and $\Diff\dir$ are continuous, $A$ is uniformly bounded. By the Cauchy-Schwarz inequality and Lipschitz continuity of $\sigma\sigma^\top$ on $\overline{\Omega}\cup\overline{\pertDomain}$
		\begin{align*}
			\trace\Big(\big(\big[\sigma\sigma^\top\big](X_s^x)-\big[\sigma\sigma^\top\big](\pertX_s)\big)\Diff^2u(\pertX_s)\Big) \leq C_4\varepsilon.
		\end{align*}
		Combining the preceding estimates we obtain
		\begin{align*}
			\big|\theRemainder_t\big| \leq C\varepsilon
		\end{align*}
		where $C>0$ does not depend on $x\in\Omega$, completing the proof of Lemma~\ref{lemma:ProofStep1}.
	\end{proof}
	
	\begin{lemma}\label{lemma:ProofStep2}
		Using the notation in the proof of Theorem~\ref{theorem:StronConvergenceShapeDerivative}, there exists a constant $C>0$ (not depending on $x\in\Omega$) such that
		\begin{align}\label{eqn:ProofStrongConvergenceGronwallPreparation}
			\big|u_\varepsilon^\dir (X_t^x)-\uExt (X_t^x)\big| \leq C \E_t\Big[ \varepsilon\big(\pertExit +1\big) + \int_t^{\pertExit}\big|u_\varepsilon^\dir (X_s^x)-\uExt (X_s^x)\big|\de s \Big], \quad t\in[0,\pertExit].\quad \closeEqn
		\end{align}
	\end{lemma}
	\begin{proof}[Proof of Lemma~\ref{lemma:ProofStep2}]
		Note that $\pertX_t\in\Omega$ on $\{t\leq\pertExit\}$. Since $\tilde{u}$ is Lipschitz on $\HoldAll$
		\begin{align*}
			\big|u_\varepsilon^\dir (X_t^x)-\uExt(X_t^x)\big| & \leq \big|u_\varepsilon^\dir (X_t^x) - u(\pertX_t)\big| + \big|u(\pertX_t)-\uExt(X_t^x)\big| \\
			&\leq \big|u_\varepsilon^\dir (X_t^x)-u(\pertX_t)\big| + C\varepsilon.
		\end{align*}
		Since $u_\varepsilon^\dir$ is a solution of the PDE \eqref{eqnSemilinearPDE} on $\pertDomain$ we have
		\begin{align*}
			u_\varepsilon^\dir (X_t^x) = \E_t\Big[ g(X_{\pertExit}^x) + \int_t^{\pertExit} f\big(X_s^x, u_\varepsilon^\dir(X_s^x)\big) \de s \Big],\quad t\in[0,\pertExit]
		\end{align*}
		and thus using the representation of $u(\pertX_t)$ in Lemma~\ref{lemma:ProofStep1} we have
		\MW{
			\begin{equation}
				\begin{aligned}
					\big|u_\varepsilon^\dir (X_t^x)-u (\pertX_t)\big| &= \bigg| \E_t\Big[ \int_t^{\pertExit} f\big(X_s^x, u_\varepsilon^\dir(X_s^x)\big) -f\big(\pertX_s,u(\pertX_s)\big) \de s \Big] \\
					& \hspace*{0.5cm} + \E_t\Big[ g(X_{\pertExit}^x) - g(\pertX_{\pertExit}) + \int_t^{\pertExit}\theRemainder_s\de s \Big] \bigg|\\
					&\leq \E_t\Big[ \int_t^{\pertExit} \big|f\big(X_s^x, u_\varepsilon^\dir(X_s^x)\big) -f\big(\pertX_s,\uExt(X_s^x)\big)\big| \de s  \Big] \\
					& \hspace*{0.5cm} + \E_t\Big[ \int_t^{\pertExit} \big|f\big(\pertX_s,\uExt(X_s^x)\big) -f\big(\pertX_s,u(\pertX_s)\big)\big| \de s \Big]\\
					& \hspace*{1.5cm} + \E_t\Big[ \big|g(X_{\pertExit}^x)-g(\pertX_{\pertExit})\big| + \int_t^{\pertExit}|\theRemainder_s|\de s \Big] \bigg|\\
					&\leq C_1 \E_t\Big[ \int_t^{\pertExit}\big|u_\varepsilon^\dir (X_s^x)-\uExt (X_s^x)\big|\de s \Big]  + C_2 \varepsilon + C_3 \varepsilon \E_t[\pertExit]
				\end{aligned}
			\end{equation}
		}
		where $C_1$ and $C_2$ are Lipschitz constants and $C_3$ is the constant from Lemma~\ref{lemma:ProofStep1}.
	\end{proof}
	
	\begin{lemma}\label{lemma:ProofStep3}
		Using the notation in the proof of Theorem~\ref{theorem:StronConvergenceShapeDerivative}, there exists a constant $C>0$ (not depending on $x\in\Omega$) such that
		\begin{equation}\label{eqn:ProofStrongConvergenceGronwallBound}
			\big|u_\varepsilon^\dir (X_t^x)-\uExt (X_t^x)\big| \leq C\varepsilon\, \E_t\big[(\pertExit+1)\exp(\pertExit)\big],\quad t\in[0,\pertExit]. \closeEqn
		\end{equation}
	\end{lemma}
	\MW{%
		\begin{proof}[Proof of Lemma~\ref{lemma:ProofStep3}]
			By Assumption \hyperref[assumption:holdAll_exp]{\emph{(E)}} we have
			\begin{align*}
				\E\bigg[\Big(\int_0^{\pertExit}\big|u_\varepsilon^\dir (X_s^x)-\uExt (X_s^x)\big| \de s \Big)^2\bigg] 
				\leq \sup_{y\in\overline{\pertDomain}}\Big(\big|u_\varepsilon^\dir(y)\big| + \big|\uExt(y)\big|\Big)^2\,\E\big[{(\pertExit)}^2\big] 
				< \infty.
			\end{align*}
			Using Lemma~\ref{lemma:ProofStep2} it follows that the stochastic Gronwall bound stated in Lemma~\ref{lemma:StochasticGronwallRandomTimes} applies (with $\alpha=1$) and the assertion holds.
		\end{proof}
	}
	%
	\begin{lemma}\label{lemma:ProofStep4}
		There is a constant $C>0$ (not dependent on $x\in\Omega$) such that
		\begin{align}\label{eqn:ProofStrongConvergenceBoundZeps}
			\big|\E[\theItoDifference]\big| \leq C\varepsilon\, \E\big[\pertExit -(\pertExit\wedge\tau^x)\big] + C\,\E\Big[\int_0^{\pertExit} \big| u_\varepsilon^\dir (X_s^x)-\uExt (X_s^x)\big|^2 \de s\Big]. \closeEqn
		\end{align}
	\end{lemma}
	\begin{proof}[Proof of Lemma~\ref{lemma:ProofStep4}]
		Since $X_{\pertExit}^x\in\partial\pertDomain$ and $\pertX_{\pertExit}\in\partial\Omega$ we have $g(X^x_{\pertExit})=u^\dir_\varepsilon(X_{\pertExit}^x)$ and $g(\pertX_{\pertExit})=u(\pertX_{\pertExit})$.
		Thus by definition of $\theItoDifference$, see \eqref{eqn:ProofSemilinearSplit}, and noting that $\beta_0=1$ it follows that
		\begin{align}\label{eqn:ProofStrongConvergenceStep1Eq1}
			\theItoDifference = &\, (1-\beta_{\pertExit})u_\varepsilon^\dir\big(X_{\pertExit}^x\big) - (1-\beta_{0})u_\varepsilon^\dir\big(X_{0}^x\big)\notag\\
			&\hspace*{2.0cm} + \beta_{\pertExit\wedge\tau^x} u\big(X_{\pertExit\wedge\tau^x}^x\big) - \beta_0 u(x)\notag\\
			& \hspace*{4.0cm} +\beta_{\pertExit} u(\pertX_{\pertExit}) - \beta_{\pertExit\wedge\tau^x} u\big(\pertX_{\pertExit\wedge\tau^x}\big)\notag\\
			&\hspace*{6.0cm} + \int_0^{\pertExit} f\big(X_s^x, u_\varepsilon^\dir(X_s^x)\big)\de s.
		\end{align}
		We proceed by expanding each line using It\=o's formula. For the first, since $u_\varepsilon^\dir$ solves the PDE on $\pertDomain$ we have
		\MW{
			\begin{equation}\label{eqn:ProofStrongConvergenceStep1Eq2}
				\begin{aligned}
					& (1-\beta_{\pertExit})u_\varepsilon^\dir\big(X_{\pertExit}^x\big) - (1-\beta_{0})u_\varepsilon^\dir\big(X_{0}^x\big) \\
					& \hspace*{1.0cm} = \int_{0}^{\pertExit} (1-\beta_s)\InfGen [u_\varepsilon^\dir](X_s^x) - \beta_s u_\varepsilon^\dir(X_s^x) \Diff_u f\big(X_s^x,\uExt(X_s^x)\big)\de s \\
					& \hspace*{2.0cm} + \int_0^{\pertExit} (1 -\beta_s)\Diff u_\varepsilon^\dir (X_s^x)\sigma(X_s^x)\de W_s \\
					& \hspace*{1.0cm} = - \int_{0}^{\pertExit} f\big(X_s^x,u_\varepsilon^\dir(X_s^x)\big)\de s \\
					& \hspace*{2.0cm} + \int_0^{\pertExit\wedge\tau^x}\beta_s\Big(f\big(X_s^x,u_\varepsilon^\dir(X_s^x)\big) - u_\varepsilon^\dir(X_s^x) \Diff_u f\big(X_s^x,\uExt(X_s^x)\big)\Big)\de s\\
					& \hspace*{3.0cm} + \int_{\pertExit\wedge\tau^x}^{\pertExit}\beta_s\Big(f\big(X_s^x,u_\varepsilon^\dir(X_s^x)\big) - u_\varepsilon^\dir(X_s^x) \Diff_u f\big(X_s^x,\uExt(X_s^x)\big)\Big)\de s\\
					& \hspace*{4.0cm} + \int_0^{\pertExit} (1 -\beta_s)\Diff u_\varepsilon^\dir (X_s^x)\sigma(X_s^x)\de W_s.
				\end{aligned}
			\end{equation}
		}
		Notice that the time integral in the fourth line of \eqref{eqn:ProofStrongConvergenceStep1Eq1} appears with negative sign in the last line here, hence cancels out in \eqref{eqn:ProofStrongConvergenceStep1Eq1}.
		Moreover, due to Lemma~\ref{lemma:ExpectationVanish} the stochastic integral vanishes in expectation.
		Hence it remains to consider the second and third lines of \eqref{eqn:ProofStrongConvergenceStep1Eq1} and \eqref{eqn:ProofStrongConvergenceStep1Eq2}.
		
		We first address the second lines, covering $[0,\pertExit\wedge\tau^x]$ and compute
		\begin{align*}
			& \beta_{\pertExit\wedge\tau^x} u\big(X_{\pertExit\wedge\tau^x}^x\big) - \beta_0 u(x) = \int_0^{\pertExit\wedge\tau^x}\beta_s \Diff u(X_s^x)\sigma(X_s^x)\de W_s \\
			& \hspace*{3.5cm} + \int_0^{\pertExit\wedge\tau^x} \beta_s \Big(\InfGen [u](X^x_s) + u(X_s^x)\Diff_u f\big(X_s^x, u(X_s^x)\big)\Big)\de s. 
		\end{align*}
		Thus the difference of the second lines of \eqref{eqn:ProofStrongConvergenceStep1Eq1} and \eqref{eqn:ProofStrongConvergenceStep1Eq2} is given by $I_{\tau^x}+M_{\pertExit\wedge\tau^x}$ where $I=(I_t)_t$ is given by
		\begin{align}\label{eqn:NewIepsilon}
			&I_t\defined\int_0^{t\wedge\pertExit} \beta_s \Big(f\big(X_s^x,u_\varepsilon^\dir(X_s^x)\big) -f\big(X_s^x,\uExt(X_s^x)\big) \notag\\
			&\hspace*{2.5cm} + \big(\uExt(X_s^x)-u_\varepsilon^\dir(X_s^x)\big)\Diff_u f\big(X_s^x,\uExt(X_s^x)\big)\Big)\de s.
		\end{align}
		and $\E[M_{\pertExit\wedge\tau^x}]=0$ by Lemma~\ref{lemma:WaldLemmaForBrownianMotion}. Next, we consider the third lines of \eqref{eqn:ProofStrongConvergenceStep1Eq1} and \eqref{eqn:ProofStrongConvergenceStep1Eq2}.
		As before, we use It\=o's formula and the notation introduced in \eqref{def:GAMMAandDTMatrix}, to obtain
		\begin{align}\label{eqn:ProofStrongConvergenceStep1Eq3}
			& \beta_{\pertExit} u(\pertX_{\pertExit}) - \beta_{\pertExit\wedge\tau^x} u\big(\pertX_{\pertExit\wedge\tau^x}\big) = \int_{\pertExit\wedge\tau^x}^{\pertExit} \beta_s \Diff u(\pertX_s) \Diff T^\top (X_s^x) \sigma (X_s^x) \de W_s \notag\\
			& \hspace*{4.5cm} + \int_{\pertExit\wedge\tau^x}^{\pertExit} \beta_s \Big\{\Diff u(\pertX_s) \Big(\Diff T^\top  (X_s^x)\mu (X_s^x)+ \tfrac{1}{2}\varepsilon\Gamma(X_s^x) \Big) \notag\\
			& \hspace*{5.5cm} + \tfrac{1}{2} \trace \Big(\big(\Diff T^\top \sigma\big)^\top (X_s^x) \Diff^2 u(\pertX_s)\big(\Diff T^\top\sigma\big)(X_s^x)\Big) \notag\\
			& \hspace*{6.5cm} + u(\pertX_s)\Diff_u f\big(X_s^x,\uExt(X_s^x)\big)\Big\}\de s
		\end{align}
		and observe that the corresponding difference is given by $\theVanisher + I_{\pertExit}-I_{\tau^x}+M'_{\pertExit}$, where
		\begin{align*}
			\theVanisher \defined & \int_{\pertExit\wedge\tau^x}^{\pertExit} \beta_s \bigg\{\Diff u(\pertX_s)^\top \Big(\Diff T^\top (X_s^x)\mu (X_s^x)+ \tfrac{1}{2}\varepsilon\Gamma (X_s^x)\Big) \notag \\
			& \hspace*{1.5cm} + \tfrac{1}{2} \trace \Big(\big(\Diff T^\top \sigma\big)^\top (X_s^x) \Diff^2 u(\pertX_s)\big(\Diff T^\top\sigma\big)(X_s^x)\Big) \notag \\
			& \hspace*{2.5cm} + \big(u(\pertX_s)-\uExt (X_s^x)\big) \Diff_u f\big(X_s^x,\uExt(X_s^x)\big) \\
			& \hspace*{3.5cm} + f\big(X_s^x, \uExt (X_s^x)\big)\bigg\}\de s
		\end{align*}
		and $\E[M'_{\pertExit}]=0$ by Lemma~\ref{lemma:WaldLemmaForBrownianMotion}. 
		Using the mean value theorem it follows that
		\begin{align*}
			\big|\E[\theItoDifference]\big| = \E\big[|I_{\pertExit}| + |\theVanisher|\big] \leq C \E \Big[\int_0^{\pertExit} \big|u_\varepsilon^\dir (X_s^x) - \uExt (X_s^x) \big|^2 \de s \Big] + \E\big[|\theVanisher|\big] .
		\end{align*}
		Thus to complete the proof it remains to show that
		\begin{align*}
			|\theVanisher| \leq C\varepsilon\big(\pertExit -(\pertExit\wedge\tau^x)\big).
		\end{align*}
		For this, using that $\pertX \in \overline{\Omega}$ until $\pertExit$, we can use that $u$ solves the PDE \eqref{eqnSemilinearPDE} and observe 
		\begin{align}\label{eqn:theVanisherEqn}
			\theVanisher & = \int_{\pertExit\wedge\tau^x}^{\pertExit} \beta_s \bigg\{\Diff u(\pertX_s)^\top \Big(\Diff T^\top (X_s^x)\mu (X_s^x) - \mu(\pertX_s) + \frac{1}{2}\Gamma (X_s^x)\Big) \notag \\
			& \hspace*{1.5cm} + \frac{1}{2} \trace \Big(\big(\Diff T^\top \sigma\big)^\top (X_s^x) \Diff^2 u(\pertX_s)\big(\Diff T^\top\sigma\big)(X_s^x)-\big(\sigma^\top\Diff^2\sigma\big)(\pertX_s)\Big)\notag \\
			& \hspace*{2.5cm} + \big(u(\pertX_s)-\uExt (X_s^x)\big) \Diff_u f\big(X_s^x,\uExt(X_s^x)\big) \notag \\
			& \hspace*{3.5cm} + f\big(X_s^x, \uExt (X_s^x)\big)-f\big(\pertX_s, u(\pertX_s)\big)\bigg\}\de s \notag \\
			& = \int_{\pertExit\wedge\tau^x}^{\pertExit} \beta_s \Big\{\theRemainder_s + \big(u(\pertX_s)-\uExt (X_s^x)\big) \Diff_u f\big(X_s^x,\uExt(X_s^x)\big) \notag \\
			& \hspace*{3.5cm} + f\big(X_s^x, \uExt (X_s^x)\big)-f\big(\pertX_s, u(\pertX_s)\big)\Big\}\de s,
		\end{align}
		where the second identity is due to \eqref{def:theRemainder}. From here we proceed by bounding the integrand of $\theVanisher$ directly. Clearly, $|\beta_s|\leq 1$ and from Lemma~\ref{lemma:ProofStep1}, we have $\sup_{y\in\Omega} \sup_{t\geq 0}|\mathcal{E}^{y,\varepsilon}_t| \leq C\varepsilon$. Moreover, with the Lipschitz continuity of $f$ and $\uExt$ together with boundedness of $\Diff_u f$, we have
		\begin{align*}
			& \big(u(\pertX_s)-\uExt (X_s^x)\big) \Diff_u f\big(X_s^x,\uExt(X_s^x)\big) + f\big(X_s^x, \uExt (X_s^x)\big)-f\big(\pertX_s, u(\pertX_s)\big)\leq C\varepsilon.\qedhere
		\end{align*}
	\end{proof}
	
	We now complete the proof of Theorem~\ref{theorem:StronConvergenceShapeDerivative} by showing that \eqref{eqn:ProofStrongConvergenceUniformConvergenceVanisher} holds, i.e.
	\begin{align}
		\LimSup\tfrac{1}{\varepsilon}\big|\E[\theItoDifference]\big| = 0.
	\end{align}
	Lemma~\ref{lemma:ProofStep4} and Lemma~\ref{lemma:ProofStep3} yield
	\begin{align}\label{eqn:ProofStronConvergenceStep5}
		\big|\E[\mathcal{Z}_\varepsilon]\big| & \leq C\varepsilon\, \E\big[\pertExit -(\pertExit\wedge\tau^x)\big] + C\,\E\Big[\int_0^{\pertExit} \big|u_\varepsilon^\dir (X_s^x)-\uExt (X_s^x)\big|^2\de s\Big] \notag\\
		& \leq C\varepsilon\, \E\big[\pertExit -(\pertExit\wedge\tau^x)\big] + C\varepsilon^2\,\E\Big[\int_0^{\pertExit} \E_s\big[(\pertExit+1)\exp (\pertExit)\big]^2 \de s\Big]
	\end{align}
	where by Tonelli's theorem
	\MW{
		\begin{equation}
			\begin{aligned}
				\E\Big[\int_0^{\pertExit}\E_s\big[(\pertExit+1)\exp(\pertExit)\big]^2 \de s\Big] & = \E\Big[\int_0^\infty \1_{\{s\leq\pertExit\}} \E_s\big[(\pertExit+1)\exp(\pertExit)\big]^2 \de s\Big] \\
				&\leq \E\Big[\int_0^\infty \1_{\{s\leq\pertExit\}}\E_s\big[(\pertExit+1)^2 \exp(2\pertExit)\big] \de s\Big] \\
				& = \int_0^\infty\E\Big[\1_{\{s\leq\pertExit\}}\E_s\big[(\pertExit+1)^2 \exp(2\pertExit)\big]\Big]\de s \\
				& = \int_0^\infty \E\Big[\1_{\{s\leq\pertExit\}} (\pertExit+1)^2 \exp(2\pertExit)\Big]\de s\\
				& = \E\Big[(\pertExit+1)^2\, \pertExit \exp(2\pertExit) \Big].
			\end{aligned}
		\end{equation}
	}
	\MW{
		Clearly $\pertExit\leq\tau_{\HoldAll}^x$ for all $x\in\HoldAll$, so by Assumption \hyperref[assumption:holdAll_exp]{\emph{(E)}} and an elementary bound (see Lemma~\ref{lemma:elementaryBoundEXP}) we have
		\begin{align}\label{eqn:proofStrongConvergenceUniformExitBound}
			\sup_{\varepsilon\in[-\varepsilon_0,\varepsilon_0]}\,\sup_{x\in\pertDomain}\, \E\Big[(\pertExit+1)^2\, \pertExit \exp(2\pertExit) \Big] 
			\leq C \E\big[\exp(\rho\tau^x_{\HoldAll})\big] < \infty 
		\end{align}
	}
	Thus we obtain from \eqref{eqn:ProofStronConvergenceStep5} using Lemma~\ref{lemma:PropertiesPerturbedNew}
	\begin{align*}
		\LimSup \tfrac{1}{\varepsilon} \big|\E[\theItoDifference]\big| \leq C \LimSup \E\big[\pertExit -(\pertExit\wedge\tau^x)\big] + C\varepsilon = 0.
	\end{align*}
	This establishes \eqref{eqn:ProofStrongConvergenceUniformConvergenceVanisher} and therefore completes the proof of Theorem~\ref{theorem:StronConvergenceShapeDerivative}.
\end{proof}
\begin{lemma}\label{lemma:ExpectationVanish}
	For any $\varepsilon \in[-\varepsilon_0,\varepsilon_0]$, we have
	\begin{align*}
		& \E\Big[(1-\beta_{\pertExit}^\varepsilon)u_\varepsilon^\dir\big(X_{\pertExit}^x\big) - (1-\beta_{0}^\varepsilon)u_\varepsilon^\dir\big(X_{0}^x\big)\Big] \\
		& \hspace*{1.5cm} = \E\Big[\int_{0}^{\pertExit} (1-\beta_s^\varepsilon)\InfGen [u_\varepsilon^\dir](X_s^x) - \beta_s^\varepsilon u_\varepsilon^\dir(X_s^x) \Diff_u f\big(\pertX_s,u(\pertX_s)\big)\de s\Big]. \closeEqn
	\end{align*}
\end{lemma}
\begin{proof}
	Throughout the proof $\varepsilon\in[-\varepsilon_0,\varepsilon_0]$ and $x\in\Omega$ are fixed. Let $\delta > 0$ and define 
	\[
	\Omega_\delta \defined \big(\pertDomain\big)_\delta \defined \big\{y\in\pertDomain \,|\, \dist \big(y,\partial\pertDomain\big) \geq \delta \big\} 
	\]
	as well as
	\[
	\tau_\delta \defined \inf \big\{t\geq 0 \,|\, \dist \big(X_t^x, \partial\pertDomain\big) < \delta \big\}.
	\]
	Then by construction $\overline{\Omega_\delta}\subseteq \Omega$ and $\tau_\delta \leq \pertExit$. With It\=o's formula, we compute
	\begin{align*}
		& (1-\beta_{\tau_\delta}^\varepsilon)u_\varepsilon^\dir\big(X_{\tau_\delta}^x\big) - (1-\beta_{0}^\varepsilon)u_\varepsilon^\dir\big(X_{0}^x\big) \\
		& \hspace*{1.5cm} = \int_{0}^{\tau_\delta} (1-\beta_s^\varepsilon)\InfGen [u_\varepsilon^\dir](X_s^x) - \beta_s^\varepsilon u_\varepsilon^\dir(X_s^x) \Diff_u f\big(\pertX_s,u(\pertX_s)\big)\de s \\
		& \hspace*{2.5cm} + \int_0^{\tau_\delta} (1 -\beta_s^\varepsilon)\Diff u_\varepsilon^\dir (X_s^x)\sigma(X_s^x)\de W_s.
	\end{align*}
	Using that $\sigma$ and $\beta^\varepsilon$ are globally bounded together with $\Diff u_\varepsilon^\dir$, we can invoke Lemma \ref{lemma:WaldLemmaForBrownianMotion} and obtain
	\begin{align*}
		\E\Big[\int_0^{\tau_\delta} (1 -\beta_s^\varepsilon)\Diff u_\varepsilon^\dir (X_s^x)\sigma(X_s^x)\de W_s\Big] = 0.
	\end{align*}
	Since $u_\varepsilon^\dir$ solves the PDE on $\pertDomain$, we have $\InfGen [u_\varepsilon^\dir] = - f(\cdot, u_\varepsilon^\dir)$ and hence
	\begin{align*}
		& \E\Big[(1-\beta_{\tau_\delta}^\varepsilon)u_\varepsilon^\dir\big(X_{\tau_\delta}^x\big) - (1-\beta_{0}^\varepsilon)u_\varepsilon^\dir\big(X_{0}^x\big)\Big] \\
		& \hspace*{1.5cm} =\E\Big[ \int_{0}^{\tau_\delta} (\beta_s^\varepsilon-1)f\big(X_s^x,u_\varepsilon^\dir (X_s^x)\big) - \beta_s^\varepsilon u_\varepsilon^\dir(X_s^x) \Diff_u f\big(\pertX_s,u(\pertX_s)\big)\de s\Big].
	\end{align*}
	Analogously to the proof of Lemma \ref{lemma:PropertiesPerturbedNew}, we can show that $\tau_\delta \rightarrow \pertExit$ a.s.~as $\delta\to 0$. With the continuity of $u_\varepsilon^\dir, f$ and $\Diff_u f$ on $\overline{\pertDomain}$ together with the continuity of $\beta^\varepsilon$ and its boundedness, we are in position to invoke Lebesgue's theorem and the proof is complete. 
\end{proof}

\section{Simulation Methodology}\label{Section:SimulationProbShape}
This section provides a simulation methodology based on the probabilistic representation of shape derivatives in Theorem~\ref{theorem:ProbabilisticRepresentationShapeFunctionalDerivative}.
Importantly, this approach does not necessarily require a mesh or discretization of the relevant domain. As a benchmark example, we consider a tracking type shape functional; 
\begin{equation}
    \trackingFunctional\colon u \mapsto \int_{\text{dom}(u)} \tfrac{1}{2} |u(x)-\trackingData(x)|^2 \dLebesgue{d}{x} 
\end{equation} 
where $\trackingData\colon\R^d\to\R$ is a given data map. Theorem~\ref{theorem:ProbabilisticRepresentationShapeFunctionalDerivative} implies that
\begin{multline*}
    \D\trackingFunctional[\dir] = \constant{+} \E\Big[\big\langle \dir,\nabla u-\nabla g\big\rangle \big(\enX^{X_0^+}\big)\Big] -\constant{-}\E\Big[\big\langle\dir ,\nabla u-\nabla g \big\rangle\big(\enX^{X_0^-}\big)\Big]\\
    - \int_{\partial\Omega}\tfrac{1}{2}\big|g(y)-\trackingData(y)\big|^2\,\big\langle\dir,n\big\rangle(y)\,\dSphere{d-1}{y} ,
\end{multline*}
where
\[
    \constant{\pm} = \pm\int_{\Omega^\pm} u(x)-\trackingData(x)\dLebesgue{d}{x}, \quad \Omega^+ = \bigl\{x\in\Omega \,\big|\, u(x)\geq \trackingData(x)\bigr\}, \quad \Omega^- = \Omega\setminus\Omega^+
\]
and
\[
    \mu^\pm [\de x] = \frac{1}{\pm\constant{\pm}} \big(u(x)-\trackingData(x)\big) \dLebesgue{d}{x}.
\]
The exit-kill random variables $\enX^{X_0^\pm}$ can be simulated via Algorithm \ref{Algorithm:ExitOrKilled}. Initial points $X_0^\pm$ can, for instance, be sampled using the acceptance-rejection method, see e.g.\ \cite[Section 2.2.2]{Glasserman}; the same method can also be employed for a mesh-free computation of the constants $\constant{\pm}$. We further wish to emphasize that the simulations of $X_0^\pm$ and, in fact, the exit-kill variables $\enX^{X_0^\pm}$ do not depend on the choice of distortion $\dir\in\pertSpace$, hence have to be carried out only once for each domain.

\medskip
\begin{algorithm}[H]
    \SetAlgoLined 
    \textbf{initialization} Choose time step size $\Delta$, simulate initial points $X_0^\pm$ and $E\sim \text{Exp}(1)$\\
    $x_0 = X_0^\pm$ \\
    $\Lambda_0 = 0$ \\
    \While{$x_k \in \Omega$}{
    	1. Update $\Lambda_k = \Lambda_{k-1} + \Diff_u f\big(x_{k-1}, u(x_{k-1})\big) \Delta$\\
        \If{$\Lambda_k\geq E$}{
            kill process, i.e.\ $x_k = \dagger$\\ 
            \textbf{break while}
        }
        2. Update $x_k$ [Euler-Maruyama Scheme]\\
    }
    \Return{$\widehat{x}^{x_0}$ that is the closest point to $x_k$ contained in $\partial\Omega$, i.e. $$\widehat{x}^{x_0} \defined \arg\min \big\{\|y-x_k\|\,\big|\, y\in\conv(x_k, x_{k-1})\cap\overline{\Omega}\big\}$$}
    \caption{Random Start Exit-kill Random Variables}\label{Algorithm:ExitOrKilled}
\end{algorithm}
\medskip

Several comments and remarks concerning Algorithm \ref{Algorithm:ExitOrKilled} are in order.
First, note that the decomposition of $\Omega$ is merely required for the simulation of the initial points. Second, while in continuous time killing is triggered when the integrated intensity exceeds the exponentially drawn threshold, in Algorithm \ref{Algorithm:ExitOrKilled} this quantity is approximated and satisfies\footnote{The last estimate is due to the definition of the Euler-Maruyama scheme, and holds under weaker assumptions than \hyperref[assumption:LadyzhenskayaPardoux]{\emph{(PDE)}}, see e.g.\ \cite[Theorem 10.2.2]{KloedenPlaten}.}

\begin{align*}
    &\E\bigg[\Big|\Lambda_k - \int_{0}^{\Delta k} \Diff_u f\big(X_s^{x_0}, u(X_s^{x_0})\big) \de s \Big|\bigg] \\
    & \hspace*{1.0cm} \leq \sum_{j=1}^k \int_{\Delta(j-1)}^{\Delta j} \E\Big[\big|\Diff_u f \big(x_{j-1},u(x_{j-1})\big) - \Diff_u f \big(X_s^{x_0},u(X_s^{x_0})\big)\big|\Big]\de s \\
    & \hspace*{1.0cm} \leq k\Delta \max_{\ell\in\{1\dots k\}}\,\E\Big[\sup_{s\in\big[\Delta (\ell-1),\Delta\ell\big]} \big|\Diff_u f \big(x_{\ell-1},u(x_{\ell-1})\big) - \Diff_u f \big(X_s^{x_0},u(X_s^{x_0})\big)\big|\Big] \\
    & \hspace*{1.0cm} \leq Ck \Delta^\frac{3}{2},
\end{align*}
where $C>0$ is a constant, for any $k\geq 1$. Finally, note that even in the special case when the Euler-Maruyama approximation is exact, the corresponding exit times in general do not coincide. Such issues can be studied via excursion theory, see \cite{YorExcursion}; here we refer to research on convergence rates of approximation schemes \cite{BGG_forward_exit, GobetMenozzi2010, Milstein1996}, and for possibilities to improve the simulation accuracy of first exit times, see e.g.\ \cite{Baldi1995, Buchmann2005, BuchmannPetersen2006, Yang2018}. Nevertheless, for the purposes of simulating the probabilistic representation of the shape derivative in Theorem~\ref{theorem:ProbabilisticRepresentationShapeFunctionalDerivative}, these issues can be minimized by choosing a sufficiently small step size $\Delta$. 

\section{Numerical Verification}\label{Section:TaylorTest}
In this section we present numerical results for a benchmark example with different perturbations.
Specifically, we compare a mesh-free simulation method based on our probabilistic representation of $\D\Phi$ as in Section~\ref{Section:SimulationProbShape} with classical methods based on finite elements.
Moreover, in accordance with the literature on shape calculus, we perform corresponding Taylor tests.
The code used for the numerical results of this section is publicly available on GitHub at {\footnotesize\url{https://github.com/max-wuer/ProbabilisticShape}}.

Specifically, similarly as in \cite{Gangl_Sturm_Taylor_Test}, we consider the unit sphere in $\R^2$ and the tracking type functional with target
\[
\trackingData\colon\R^2\rightarrow\R;\qquad (x_1,x_2)\mapsto x_1(1-x_1)x_2(1-x_2).
\]
The state equation is given by the PDE~\ref{eqnSemilinearPDE} with coefficients
\begin{align*}
	\mu = 0,\quad \sigma = \sqrt{2}\mathcal{I},\quad f = 1,\quad g = 0.
\end{align*}
We next describe the numerical representations of the shape derivative $\D\Phi$. We provide (a)~a mesh-free representation as in Section~\ref{Section:SimulationProbShape}, (b)~a classical volume formulation based on finite elements, (c)~a classical boundary formulation based on finite elements, and (d)~a hybrid representation based on Theorem~\ref{theorem:ProbabilisticRepresentationShapeFunctionalDerivative} and an adjoint PDE.

\textbf{(a) Mesh-free representation.} 
According to Theorem~\ref{theorem:ProbabilisticRepresentationShapeFunctionalDerivative} the shape derivative is given by 
\begin{equation}
	\D\Phi_{\text{free}}[\dir] = \constant{+} \E\Big[\big\langle \dir,\nabla u\big\rangle \big(\enX^{X_0^+}\big)\Big] -\constant{-}\E\Big[\big\langle\dir ,\nabla u\big\rangle \big(\enX^{X_0^-}\big)\Big]\\
	- \int_{\partial\Omega}\big\langle\dir,\phi\big(\cdot,u\big)n\big\rangle\dSph{d-1} .
\end{equation}
We achieve a mesh-free evaluation of this representation as in Section~\ref{Section:SimulationProbShape} by employing a neural PDE solver and Monte Carlo simulations to obtain the constants $\constant{\pm}$, the expected values of the exit-kill random variables, and the surface integral in the representation of $\D\Phi$.\footnote{The implementation is based on \textsc{NumPy} and \textsc{PyTorch}.}

\textbf{(b) Volume formulation.} The classical volume representation of $\D\Phi$ is given by, see e.g.\ \cite[Section 3.2]{SturmDissertation});
\begin{multline}\label{EQ:Tracking_Type_Volume_Formulation}
	\D \trackingFunctional_{\text{vol}}[\dir]=\int_\Omega \tfrac{1}{2}(u-\trackingData)^2\divergence(\dir) - \divergence (\dir)p - (u-\trackingData)(\nabla\trackingData)^\top\dir\dLeb{d}\\
	-\int_\Omega (\nabla u)^T\left[ \divergence(\dir)\mathbf{I}-\Diff \dir-\Diff \dir^T\right]\nabla p \dLeb{d}
\end{multline}
where the state equation is understood in a weak sense, i.e.\footnote{$H_0^1(\Omega)$ denotes the Lebesgue-Sobolev space of weakly differentiable functions compactly supported within $\Omega$.}
\begin{equation}
	\text{Find }u\in H_0^1(\Omega),\quad\text{s.t. }-\int_\Omega\nabla u^T\nabla v \de x = \int_\Omega v\de x \quad \text{for every } v\in H_0^1(\Omega)
\end{equation}
and the adjoint state $p\in H_0^1(\Omega)$ is the solution of
\begin{equation}
	\text{Find }p\in H_0^1(\Omega),\quad\text{s.t. }-\int_\Omega\nabla p^T\nabla v \dLeb{d} = -\int_\Omega (u-\trackingData)v\dLeb{d} \quad \text{for all } v\in H_0^1(\Omega).
\end{equation}
The relevant PDEs are solved numerically using finite element methods.\footnote{Implementations of finite element methods are based on \textsc{FEniCS}.}

\textbf{(c) Boundary formulation.} The classical boundary formulation is obtained via Hadamard's structure theorem, see e.g.\ \cite[Theorem 2.27]{Sokolowski_Shape};
\begin{equation}\label{EQ:Tracking_Type_Boundary_Formulation}
	\D \trackingFunctional_{\text{bdry}}[\dir]=\int_{\partial\Omega}\langle\nabla u, n \rangle \langle\nabla p, n \rangle \langle \dir, n \rangle \dSph{d-1} + \tfrac{1}{2} \int_{\partial\Omega} \trackingData^2 \langle\dir, n\rangle \dSph{d-1}.
\end{equation}
As in the volume formulation, the relevant PDEs are solved by finite element methods.

\textbf{(d) Feynman-Kac representation.} The Feynman-Kac formulation of the shape derivative is given by
\begin{equation*}
	\D\trackingFunctional_{FK}[\dir] = \int_\Omega\Dphiu\phi(\cdot,u) \adjointFK \dLeb{d} - \int_{\partial\Omega} \big\langle\dir, \phi(\cdot,u)n \big\rangle\dSph{d-1}
\end{equation*}
where $\adjointFK$ denotes the $\C^2(\Omega)$ solution\footnote{Existence and uniqueness are ensured due to e.g.\ \cite[Theorem 6.13]{GilbargTrudinger}.} of the probabilistic adjoint equation
\begin{equation}\label{eqn:ProbabilisticAdjointPDE}
	\Delta \adjointFK = 0 \quad\text{on }\Omega,\qquad \adjointFK = \langle \nabla u -\nabla g, \dir\rangle \quad\text{on }\partial\Omega .
\end{equation}
This formulation is justified by Theorem~\ref{theorem:ProbabilisticRepresentationShapeDerivative} and, importantly, makes it possible to evaluate the probabilistic representation of Theorem~\ref{theorem:ProbabilisticRepresentationShapeDerivative} using purely deterministic means.
As such, it represents a hybrid between the mesh-free probabilistic and the classical volume formulation.\footnote{Note that the Feynman-Kac formulation merely applies in the absence of killing, i.e.\ whenever $\Diff_u f=0$. Nonetheless, the formulation can be helpful to validate the results of Monte Carlo simulations since
	\begin{equation}
		\E\Big[\big\langle \dir,\nabla u-\nabla g\big\rangle \big(\enX^{X_0^\pm}\big)\Big] = \int_{\Omega} \tfrac{1}{\pm\constant{\pm}} (u-\trackingData) \adjointFK \dLeb{d} .
\end{equation}}
As in the classical formulations, the adjoint equation is solved numerically using finite elements.
\begin{table}[H]
    \centering
    \def\arraystretch{1.1}
    \begin{tabular}{ 
                        l | 
                        S[table-format = 3.4] 
                        l | 
                        S[table-format = 3.4] 
                        S[table-format = 3.4] 
                        S[table-format = 3.4] 
                    }
    \toprule
     & \multicolumn{2}{c}{$\D\trackingFunctional_{\text{free}}$} & \multicolumn{1}{c}{$\D\trackingFunctional_{\text{FK}}$} & \multicolumn{1}{c}{$\D\trackingFunctional_{\text{vol}}$} & \multicolumn{1}{c}{$\D\trackingFunctional_{\text{bdry}}$} \\
    \midrule
    $V_1 (x) = x$ & -0.9883820192222084 & \hspace*{-0.25cm}(0.0013) & -0.9897413057026561 & 0.989603598769655 & 0.9899758739058729 \\
    $V_2 (x) = (x_1-x_2, x_2-x_1)^\top$ & -0.5224728727378143 & \hspace*{-0.25cm}(0.0008) & -0.5319122640686904 & 0.5316396030934368 & 0.5322138452070018 \\
    $V_3 (x) = \big(\cos(x_1), \sin(x_2)\big)^\top$ & -0.10060805778779433 & \hspace*{-0.25cm}(0.0005) & -0.10275440391804858 & 0.10264456106433673 & 0.10287157268894756 \\
    $V_4 (x) = (x_1x_2,x_2)^\top$ & -0.29920424342449126 & \hspace*{-0.25cm}(0.0006) & -0.3027205231289931 & 0.3026404133170715 & 0.3028343473683559 \\
    $V_5 (x) = (1,0)^\top$ & 0.47721721024256303 & \hspace*{-0.25cm}(0.0010) & 0.4742946966853197 & -0.47436324746592073 & -0.47426996884789974\\
    $V_6 (x) = x_1\ x$ & 0.4760821965242726 & \hspace*{-0.25cm}(0.0010) & 0.4742499559034586 & -0.4743185430429145 & -0.474225251057891 \\
    $V_7 (x) = (0.3, 0.2)^\top - x$ & 1.2295162349459852 & \hspace*{-0.25cm}(0.0016) & 1.2268886540453143 & -1.2267852225026052 & -1.2271108583298223 \\
    $V_8 (x) = \sin\big(6\arctan(\tfrac{x_1}{x_2})\big) x$ & 0.196902122225113 & \hspace*{-0.25cm}(0.0008) & 0.19627531858764263 & -0.19621570730989138 & -0.19627531858764666 \\
    \bottomrule
    \end{tabular}
    \caption{Shape Derivative Value Comparison}
    \label{tab:TestResults}
\end{table}

Table~\ref{tab:TestResults} presents a comparison between the mesh-free computed values of the probabilistic shape derivative $\D\Phi_{\text{free}}$ and the finite element method based formulations $\D\Phi_{\text{vol}}, \D\Phi_{\text{bdry}}$ and $\D\Phi_{\text{FK}}$. The shape derivative is evaluated for a variety of directions, including classical perturbations in outer normal direction, obliquely pointing outward, contraction to a single point, and perturbations pointing in- and outward.


Adapting\footnote{Differences in signs are due to our definition of $\pertDomain$ as a pre-image. Our definition can be identified with that in \cite{Gangl_Sturm_Taylor_Test} for an appropriate $\widetilde{\dir}\in\pertSpace$ satisfying $\pertDomain = T_\varepsilon^{\widetilde{\dir}}(\Omega)$.} \cite[Section 7.1]{Gangl_Sturm_Taylor_Test} to our setting, a successful test result can be defined as follows.
\begin{definition}[Taylor test]
	Let $\Omega\subset\R^d$ be a bounded domain and $\dir\in\pertSpace$. Let $\mathcal{J}$ denote a shape functional as given in \eqref{Def:Shape_Functional} with shape derivative $\D \mathcal{J}$ as in Definition \ref{Def:ShapeDerivative} and set
	\begin{align}
		\mathcal{E}(\mathcal{J};\varepsilon)\defined\big|\mathcal{J}\big(\Omega_\epsilon^\dir\big)-\mathcal{J}(\Omega)+\varepsilon \D \mathcal{J}(\Omega)[\dir]\big|\quad\text{for }\varepsilon \in [-\varepsilon_0,\varepsilon_0].
	\end{align}
	Then $\D \mathcal{J}$ is said to satisfy the Taylor test for $\Omega$ and $\dir$ if 
	\[
	\mathcal{E}(\mathcal{J}; \varepsilon) = \mathcal{O}(\varepsilon^2) \quad\text{as }\varepsilon\to 0. \closeEqn
	\]
\end{definition}

Figure~\ref{Figure:TaylorTest} displays the results of the Taylor test for the directions in Table~\ref{tab:TestResults}, as well as the domain partition induced by the support of the initial point distributions. 
%
%
%
%
%
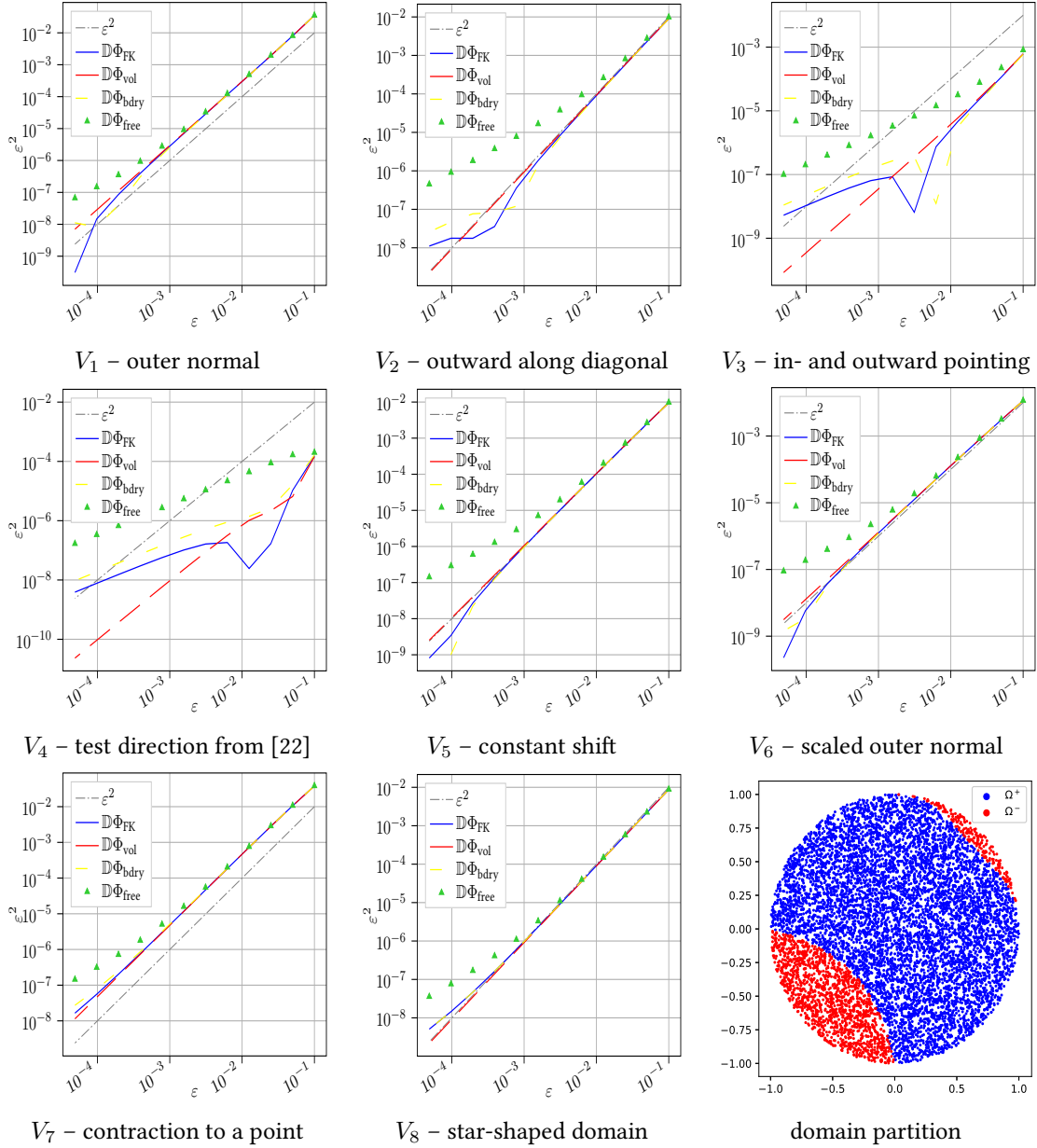
\begin{figure}[H]
	\begin{subfigure}{0.32\textwidth}
		\resizebox{\textwidth}{\textwidth}{
\begin{tikzpicture}

\definecolor{darkgray176}{RGB}{176,176,176}
\definecolor{gray}{RGB}{128,128,128}
\definecolor{lightgray204}{RGB}{204,204,204}
\definecolor{limegreen}{RGB}{50,205,50}
\definecolor{yellow}{RGB}{255,255,0}

\begin{axis}[
legend cell align={left},
legend style={
  fill opacity=0.8,
  draw opacity=1,
  text opacity=1,
  at={(0.03,0.97)},
  anchor=north west,
  draw=lightgray204
},
log basis x={10},
log basis y={10},
tick align=outside,
tick pos=left,
title={},
x grid style={darkgray176},
xlabel={\(\displaystyle \varepsilon\)},
xmajorgrids,
xmin=3.33505922059178e-05, xmax=0.146408569594563,
xmode=log,
xtick style={color=black},
xtick={1e-06,1e-05,0.0001,0.001,0.01,0.1,1,10},
xticklabel style={rotate=30.0,anchor=east},
xticklabels={
  \(\displaystyle {10^{-6}}\),
  \(\displaystyle {10^{-5}}\),
  \(\displaystyle {10^{-4}}\),
  \(\displaystyle {10^{-3}}\),
  \(\displaystyle {10^{-2}}\),
  \(\displaystyle {10^{-1}}\),
  \(\displaystyle {10^{0}}\),
  \(\displaystyle {10^{1}}\)
},
y grid style={darkgray176},
ylabel={\(\displaystyle \varepsilon^2\)},
ymajorgrids,
ymin=1.21645534749323e-10, ymax=0.089492176345068,
ymode=log,
ytick style={color=black},
ytick={1e-11,1e-10,1e-09,1e-08,1e-07,1e-06,1e-05,0.0001,0.001,0.01,0.1,1},
yticklabels={
  \(\displaystyle {10^{-11}}\),
  \(\displaystyle {10^{-10}}\),
  \(\displaystyle {10^{-9}}\),
  \(\displaystyle {10^{-8}}\),
  \(\displaystyle {10^{-7}}\),
  \(\displaystyle {10^{-6}}\),
  \(\displaystyle {10^{-5}}\),
  \(\displaystyle {10^{-4}}\),
  \(\displaystyle {10^{-3}}\),
  \(\displaystyle {10^{-2}}\),
  \(\displaystyle {10^{-1}}\),
  \(\displaystyle {10^{0}}\)
}
]
\addplot [line width=0.44pt, gray, dash pattern=on 7.04pt off 1.76pt on 1.1pt off 1.76pt]
table {%
0.1 0.01
0.05 0.0025
0.025 0.000625
0.0125 0.00015625
0.00625 3.90625e-05
0.003125 9.765625e-06
0.0015625 2.44140625e-06
0.00078125 6.103515625e-07
0.000390625 1.52587890625e-07
0.0001953125 3.814697265625e-08
9.765625e-05 9.5367431640625e-09
4.8828125e-05 2.38418579101563e-09
};
\addlegendentry{$\varepsilon^2$}
\addplot [line width=0.7pt, blue]
table {%
0.1 0.0352437395961115
0.05 0.00804420661325708
0.025 0.00192185640868384
0.0125 0.000469131508596066
0.00625 0.000115571517017174
0.003125 2.85193372286455e-05
0.0015625 7.00258052233407e-06
0.00078125 1.69441145559009e-06
0.000390625 3.96407289785829e-07
0.0001953125 8.56190140248632e-08
9.765625e-05 1.46777103841259e-08
4.8828125e-05 3.07699992330331e-10
};
\addlegendentry{$\D\Phi_{\text{FK}}$}
\addplot [line width=0.7pt, red, dash pattern=on 17.5pt off 8.75pt]
table {%
0.1 0.0352575102894116
0.05 0.00805109195990714
0.025 0.00192529908200886
0.0125 0.000470852845258579
0.00625 0.000116432185348431
0.003125 2.89496713942738e-05
0.0015625 7.21774760514826e-06
0.00078125 1.80199499699718e-06
0.000390625 4.50199060489376e-07
0.0001953125 1.12514899376637e-07
9.765625e-05 2.81256530600127e-08
4.8828125e-05 7.03167133027371e-09
};
\addlegendentry{$\D\Phi_{\text{vol}}$}
\addplot [line width=0.7pt, yellow, dash pattern=on 8.75pt off 17.5pt]
table {%
0.1 0.0352202827757898
0.05 0.00803247820309624
0.025 0.00191599220360341
0.0125 0.000466199406055854
0.00625 0.000114105465747068
0.003125 2.77863115935926e-05
0.0015625 6.63606770480766e-06
0.00078125 1.51115504682688e-06
0.000390625 3.04779085404225e-07
0.0001953125 3.98049118340612e-08
9.765625e-05 8.22934071127506e-09
4.8828125e-05 1.11458255553702e-08
};
\addlegendentry{$\D\Phi_{\text{bdry}}$}
\addplot [draw=limegreen, fill=limegreen, mark=triangle*, only marks]
table{%
x  y
0.1 0.0353796682441562
0.05 0.00811217093727946
0.025 0.00195583857069503
0.0125 0.000486122589601661
0.00625 0.000124067057519972
0.003125 3.27671074800443e-05
0.0015625 9.12646564803351e-06
0.00078125 2.75635401843981e-06
0.000390625 9.27378571210689e-07
0.0001953125 3.51104654737293e-07
9.765625e-05 1.47420530740341e-07
4.8828125e-05 6.66791101704378e-08
};
\addlegendentry{$\D\Phi_{\text{free}}$}
\end{axis}

\end{tikzpicture}
		}
		\subcaption*{$V_1$ -- outer normal}
	\end{subfigure}
	\hfill
	\begin{subfigure}{0.32\textwidth}
		\centering
		\resizebox{\textwidth}{\textwidth}{
\begin{tikzpicture}

\definecolor{darkgray176}{RGB}{176,176,176}
\definecolor{gray}{RGB}{128,128,128}
\definecolor{lightgray204}{RGB}{204,204,204}
\definecolor{limegreen}{RGB}{50,205,50}
\definecolor{yellow}{RGB}{255,255,0}

\begin{axis}[
legend cell align={left},
legend style={
  fill opacity=0.8,
  draw opacity=1,
  text opacity=1,
  at={(0.03,0.97)},
  anchor=north west,
  draw=lightgray204
},
log basis x={10},
log basis y={10},
tick align=outside,
tick pos=left,
title={},
x grid style={darkgray176},
xlabel={\(\displaystyle \varepsilon\)},
xmajorgrids,
xmin=3.33505922059178e-05, xmax=0.146408569594563,
xmode=log,
xtick style={color=black},
xtick={1e-06,1e-05,0.0001,0.001,0.01,0.1,1,10},
xticklabel style={rotate=30.0,anchor=east},
xticklabels={
  \(\displaystyle {10^{-6}}\),
  \(\displaystyle {10^{-5}}\),
  \(\displaystyle {10^{-4}}\),
  \(\displaystyle {10^{-3}}\),
  \(\displaystyle {10^{-2}}\),
  \(\displaystyle {10^{-1}}\),
  \(\displaystyle {10^{0}}\),
  \(\displaystyle {10^{1}}\)
},
y grid style={darkgray176},
ylabel={\(\displaystyle \varepsilon^2\)},
ymajorgrids,
ymin=1.03370707433198e-09, ymax=0.0215103629614632,
ymode=log,
ytick style={color=black},
ytick={1e-10,1e-09,1e-08,1e-07,1e-06,1e-05,0.0001,0.001,0.01,0.1,1},
yticklabels={
  \(\displaystyle {10^{-10}}\),
  \(\displaystyle {10^{-9}}\),
  \(\displaystyle {10^{-8}}\),
  \(\displaystyle {10^{-7}}\),
  \(\displaystyle {10^{-6}}\),
  \(\displaystyle {10^{-5}}\),
  \(\displaystyle {10^{-4}}\),
  \(\displaystyle {10^{-3}}\),
  \(\displaystyle {10^{-2}}\),
  \(\displaystyle {10^{-1}}\),
  \(\displaystyle {10^{0}}\)
}
]
\addplot [line width=0.44pt, gray, dash pattern=on 7.04pt off 1.76pt on 1.1pt off 1.76pt]
table {%
0.1 0.01
0.05 0.0025
0.025 0.000625
0.0125 0.00015625
0.00625 3.90625e-05
0.003125 9.765625e-06
0.0015625 2.44140625e-06
0.00078125 6.103515625e-07
0.000390625 1.52587890625e-07
0.0001953125 3.814697265625e-08
9.765625e-05 9.5367431640625e-09
4.8828125e-05 2.38418579101563e-09
};
\addlegendentry{$\varepsilon^2$}
\addplot [line width=0.7pt, blue]
table {%
0.1 0.00894379795784506
0.05 0.00228797916483724
0.025 0.000573341222251188
0.0125 0.000142030871878637
0.00625 3.46927881762615e-05
0.003125 8.25099494833944e-06
0.0015625 1.85016351477189e-06
0.00078125 3.56085332561948e-07
0.000390625 3.57746120408604e-08
0.0001953125 1.76819874623765e-08
9.765625e-05 1.77335961078749e-08
4.8828125e-05 1.10899827458338e-08
};
\addlegendentry{$\D\Phi_{\text{FK}}$}
\addplot [line width=0.7pt, red, dash pattern=on 17.5pt off 8.75pt]
table {%
0.1 0.00897106405537042
0.05 0.00230161221359992
0.025 0.000580157746632528
0.0125 0.000145439134069307
0.00625 3.63969192715965e-05
0.003125 9.10306049600696e-06
0.0015625 2.27619628860565e-06
0.00078125 5.69101719478829e-07
0.000390625 1.42282805499301e-07
0.0001953125 3.55721092668438e-08
9.765625e-05 8.89345225673526e-09
4.8828125e-05 2.22354143647132e-09
};
\addlegendentry{$\D\Phi_{\text{vol}}$}
\addplot [line width=0.7pt, yellow, dash pattern=on 8.75pt off 17.5pt]
table {%
0.1 0.00891363984401392
0.05 0.00227290010792167
0.025 0.000565801693793404
0.0125 0.000138261107649745
0.00625 3.28079060618156e-05
0.003125 7.30855389111649e-06
0.0015625 1.37894298616042e-06
0.00078125 1.2047506825621e-07
0.000390625 8.20305201120082e-08
0.0001953125 7.65845535388108e-08
9.765625e-05 4.71848791460921e-08
4.8828125e-05 2.58156242649423e-08
};
\addlegendentry{$\D\Phi_{\text{bdry}}$}
\addplot [draw=limegreen, fill=limegreen, mark=triangle*, only marks]
table{%
x  y
0.1 0.00988773709093267
0.05 0.00275994873138104
0.025 0.000809326005523092
0.0125 0.000260023263514588
0.00625 9.36889839942374e-05
0.003125 3.77490928573274e-05
0.0015625 1.65992124692659e-05
0.00078125 7.73060980980894e-06
0.000390625 3.72303685066436e-06
0.0001953125 1.82594913184937e-06
9.765625e-05 9.04081963547999e-07
4.8828125e-05 4.49817797082103e-07
};
\addlegendentry{$\D\Phi_{\text{free}}$}
\end{axis}
\end{tikzpicture}
		}
		\subcaption*{$V_2$ -- outward along diagonal}
	\end{subfigure}
	\hfill
	\begin{subfigure}{0.32\textwidth}
		\centering
		\resizebox{\textwidth}{\textwidth}{
\begin{tikzpicture}

\definecolor{darkgray176}{RGB}{176,176,176}
\definecolor{gray}{RGB}{128,128,128}
\definecolor{lightgray204}{RGB}{204,204,204}
\definecolor{limegreen}{RGB}{50,205,50}
\definecolor{yellow}{RGB}{255,255,0}

\begin{axis}[
legend cell align={left},
legend style={
  fill opacity=0.8,
  draw opacity=1,
  text opacity=1,
  at={(0.03,0.97)},
  anchor=north west,
  draw=lightgray204
},
log basis x={10},
log basis y={10},
tick align=outside,
tick pos=left,
title={},
x grid style={darkgray176},
xlabel={\(\displaystyle \varepsilon\)},
xmajorgrids,
xmin=3.33505922059178e-05, xmax=0.146408569594563,
xmode=log,
xtick style={color=black},
xtick={1e-06,1e-05,0.0001,0.001,0.01,0.1,1,10},
xticklabel style={rotate=30.0,anchor=east},
xticklabels={
  \(\displaystyle {10^{-6}}\),
  \(\displaystyle {10^{-5}}\),
  \(\displaystyle {10^{-4}}\),
  \(\displaystyle {10^{-3}}\),
  \(\displaystyle {10^{-2}}\),
  \(\displaystyle {10^{-1}}\),
  \(\displaystyle {10^{0}}\),
  \(\displaystyle {10^{1}}\)
},
y grid style={darkgray176},
ylabel={\(\displaystyle \varepsilon^2\)},
ymajorgrids,
ymin=3.32113908383791e-11, ymax=0.0253365941654032,
ymode=log,
ytick style={color=black},
ytick={1e-13,1e-11,1e-09,1e-07,1e-05,0.001,0.1,10},
yticklabels={
  \(\displaystyle {10^{-13}}\),
  \(\displaystyle {10^{-11}}\),
  \(\displaystyle {10^{-9}}\),
  \(\displaystyle {10^{-7}}\),
  \(\displaystyle {10^{-5}}\),
  \(\displaystyle {10^{-3}}\),
  \(\displaystyle {10^{-1}}\),
  \(\displaystyle {10^{1}}\)
}
]
\addplot [line width=0.44pt, gray, dash pattern=on 7.04pt off 1.76pt on 1.1pt off 1.76pt]
table {%
0.1 0.01
0.05 0.0025
0.025 0.000625
0.0125 0.00015625
0.00625 3.90625e-05
0.003125 9.765625e-06
0.0015625 2.44140625e-06
0.00078125 6.103515625e-07
0.000390625 1.52587890625e-07
0.0001953125 3.814697265625e-08
9.765625e-05 9.5367431640625e-09
4.8828125e-05 2.38418579101563e-09
};
\addlegendentry{$\varepsilon^2$}
\addplot [line width=0.7pt, blue]
table {%
0.1 0.000600019818797265
0.05 0.00011476676083812
0.025 2.32227296851849e-05
0.0125 4.60696601811412e-06
0.00625 7.44532189537995e-07
0.003125 6.53051865291991e-09
0.0015625 8.51702583463314e-08
0.00078125 6.43193568962793e-08
0.000390625 3.75464134090057e-08
0.0001953125 2.01140379838423e-08
9.765625e-05 1.03914958270388e-08
4.8828125e-05 5.27927423826504e-09
};
\addlegendentry{$\D\Phi_{\text{FK}}$}
\addplot [line width=0.7pt, red, dash pattern=on 17.5pt off 8.75pt]
table {%
0.1 0.00061100410416845
0.05 0.000120258903523712
0.025 2.59688010279813e-05
0.0125 5.98000168951229e-06
0.00625 1.43105002523708e-06
0.003125 3.49789436502462e-07
0.0015625 8.64592005784398e-08
0.00078125 2.14953725661063e-08
0.000390625 5.36095132218709e-09
0.0001953125 1.33964438175406e-09
9.765625e-05 3.3534535575936e-10
4.8828125e-05 8.41463531340604e-11
};
\addlegendentry{$\D\Phi_{\text{vol}}$}
\addplot [line width=0.7pt, yellow, dash pattern=on 8.75pt off 17.5pt]
table {%
0.1 0.000588302941707365
0.05 0.00010890832229317
0.025 2.02935104127101e-05
0.0125 3.14235638187673e-06
0.00625 1.2227371419299e-08
0.003125 3.59621890406428e-07
0.0015625 2.68246462876005e-07
0.00078125 1.55857459161116e-07
0.000390625 8.33154645414242e-08
0.0001953125 4.29985635500516e-08
9.765625e-05 2.18337586101435e-08
4.8828125e-05 1.10004056298173e-08
};
\addlegendentry{$\D\Phi_{\text{bdry}}$}
\addplot [draw=limegreen, fill=limegreen, mark=triangle*, only marks]
table{%
x  y
0.1 0.000814654431822688
0.05 0.000222084067350832
0.025 7.68813829415409e-05
0.0125 3.14362926462921e-05
0.00625 1.4159195503627e-05
0.003125 6.71386217569741e-06
0.0015625 3.26849557017592e-06
0.00078125 1.61251355736484e-06
0.000390625 8.00870043721556e-07
0.0001953125 3.99094190581439e-07
9.765625e-05 1.99212618455602e-07
4.8828125e-05 9.95227829030552e-08
};
\addlegendentry{$\D\Phi_{\text{free}}$}
\end{axis}

\end{tikzpicture}
		}
		\subcaption*{$V_3$ -- in- and outward pointing }
	\end{subfigure}
	
	\smallskip
	
	\begin{subfigure}{0.32\textwidth}
		\centering
		\resizebox{\textwidth}{\textwidth}{
			 

\begin{tikzpicture}

\definecolor{darkgray176}{RGB}{176,176,176}
\definecolor{gray}{RGB}{128,128,128}
\definecolor{lightgray204}{RGB}{204,204,204}
\definecolor{limegreen}{RGB}{50,205,50}
\definecolor{yellow}{RGB}{255,255,0}

\begin{axis}[
legend cell align={left},
legend style={
  fill opacity=0.8,
  draw opacity=1,
  text opacity=1,
  at={(0.03,0.97)},
  anchor=north west,
  draw=lightgray204
},
log basis x={10},
log basis y={10},
tick align=outside,
tick pos=left,
title={},
x grid style={darkgray176},
xlabel={\(\displaystyle \varepsilon\)},
xmajorgrids,
xmin=3.33505922059178e-05, xmax=0.146408569594563,
xmode=log,
xtick style={color=black},
xtick={1e-06,1e-05,0.0001,0.001,0.01,0.1,1,10},
xticklabel style={rotate=30.0,anchor=east},
xticklabels={
  \(\displaystyle {10^{-6}}\),
  \(\displaystyle {10^{-5}}\),
  \(\displaystyle {10^{-4}}\),
  \(\displaystyle {10^{-3}}\),
  \(\displaystyle {10^{-2}}\),
  \(\displaystyle {10^{-1}}\),
  \(\displaystyle {10^{0}}\),
  \(\displaystyle {10^{1}}\)
},
y grid style={darkgray176},
ylabel={\(\displaystyle \varepsilon^2\)},
ymajorgrids,
ymin=8.5886140502052e-12, ymax=0.0270220285334539,
ymode=log,
ytick style={color=black},
ytick={1e-14,1e-12,1e-10,1e-08,1e-06,0.0001,0.01,1,100},
yticklabels={
  \(\displaystyle {10^{-14}}\),
  \(\displaystyle {10^{-12}}\),
  \(\displaystyle {10^{-10}}\),
  \(\displaystyle {10^{-8}}\),
  \(\displaystyle {10^{-6}}\),
  \(\displaystyle {10^{-4}}\),
  \(\displaystyle {10^{-2}}\),
  \(\displaystyle {10^{0}}\),
  \(\displaystyle {10^{2}}\)
}
]
\addplot [line width=0.44pt, gray, dash pattern=on 7.04pt off 1.76pt on 1.1pt off 1.76pt]
table {%
0.1 0.01
0.05 0.0025
0.025 0.000625
0.0125 0.00015625
0.00625 3.90625e-05
0.003125 9.765625e-06
0.0015625 2.44140625e-06
0.00078125 6.103515625e-07
0.000390625 1.52587890625e-07
0.0001953125 3.814697265625e-08
9.765625e-05 9.5367431640625e-09
4.8828125e-05 2.38418579101563e-09
};
\addlegendentry{$\varepsilon^2$}
\addplot [line width=0.7pt, blue]
table {%
0.1 0.000151977920402752
0.05 1.05605947748333e-05
0.025 1.69109871439588e-07
0.0125 2.43841570685133e-08
0.00625 1.83540614200062e-07
0.003125 1.634443252789e-07
0.0015625 1.0249215712764e-07
0.00078125 5.67955450751431e-08
0.000390625 2.98297923181055e-08
0.0001953125 1.52784846239439e-08
9.765625e-05 7.73083732132692e-09
4.8828125e-05 3.88840373284237e-09
};
\addlegendentry{$\D\Phi_{\text{FK}}$}
\addplot [line width=0.7pt, red, dash pattern=on 17.5pt off 8.75pt]
table {%
0.1 0.00014396693921059
0.05 6.55510417875253e-06
0.025 2.17185516947997e-06
0.0125 1.0257568060887e-06
0.00625 3.17145710310033e-07
0.003125 8.68988369761474e-08
0.0015625 2.26794239998837e-08
0.00078125 5.79024548861867e-09
0.000390625 1.46310296377536e-09
0.0001953125 3.67963016996529e-10
9.765625e-05 9.23864991433032e-11
4.8828125e-05 2.32081773927467e-11
};
\addlegendentry{$\D\Phi_{\text{vol}}$}
\addplot [line width=0.7pt, yellow, dash pattern=on 8.75pt off 17.5pt]
table {%
0.1 0.000163360344339031
0.05 1.62518067429729e-05
0.025 2.67649611263021e-06
0.0125 1.39841883496639e-06
0.00625 8.94942110217511e-07
0.003125 5.19145073287625e-07
0.0015625 2.80342531132002e-07
0.00078125 1.45720732077324e-07
0.000390625 7.42923858191961e-08
0.0001953125 3.75097813744892e-08
9.765625e-05 1.88464856965996e-08
4.8828125e-05 9.44622792047869e-09
};
\addlegendentry{$\D\Phi_{\text{bdry}}$}
\addplot [draw=limegreen, fill=limegreen, mark=triangle*, only marks]
table{%
x  y
0.1 0.000199650050047434
0.05 0.00016525339045026
0.025 8.8076102483986e-05
0.0125 4.39778804633417e-05
0.00625 2.17932075389365e-05
0.003125 1.08249297512894e-05
0.0015625 5.39169488115651e-06
0.00078125 2.69029797406693e-06
0.000390625 1.34371696725293e-06
0.0001953125 6.71494895161575e-07
9.765625e-05 3.35655852571433e-07
4.8828125e-05 1.67804941213537e-07
};
\addlegendentry{$\D\Phi_{\text{free}}$}
\end{axis}

\end{tikzpicture}
		}
		\subcaption*{$V_4$ -- test direction from \cite{Gangl_Sturm_Taylor_Test}}
	\end{subfigure}
	\hfill
	\begin{subfigure}{0.32\textwidth}
		\centering
		\resizebox{\textwidth}{\textwidth}{
\begin{tikzpicture}

\definecolor{darkgray176}{RGB}{176,176,176}
\definecolor{gray}{RGB}{128,128,128}
\definecolor{lightgray204}{RGB}{204,204,204}
\definecolor{limegreen}{RGB}{50,205,50}
\definecolor{yellow}{RGB}{255,255,0}

\begin{axis}[
legend cell align={left},
legend style={
  fill opacity=0.8,
  draw opacity=1,
  text opacity=1,
  at={(0.03,0.97)},
  anchor=north west,
  draw=lightgray204
},
log basis x={10},
log basis y={10},
tick align=outside,
tick pos=left,
title={},
x grid style={darkgray176},
xlabel={\(\displaystyle \varepsilon\)},
xmajorgrids,
xmin=3.33505922059178e-05, xmax=0.146408569594563,
xmode=log,
xtick style={color=black},
xtick={1e-06,1e-05,0.0001,0.001,0.01,0.1,1,10},
xticklabel style={rotate=30.0,anchor=east},
xticklabels={
  \(\displaystyle {10^{-6}}\),
  \(\displaystyle {10^{-5}}\),
  \(\displaystyle {10^{-4}}\),
  \(\displaystyle {10^{-3}}\),
  \(\displaystyle {10^{-2}}\),
  \(\displaystyle {10^{-1}}\),
  \(\displaystyle {10^{0}}\),
  \(\displaystyle {10^{1}}\)
},
y grid style={darkgray176},
ylabel={\(\displaystyle \varepsilon^2\)},
ymajorgrids,
ymin=3.59279772007365e-10, ymax=0.02262055215466,
ymode=log,
ytick style={color=black},
ytick={1e-11,1e-10,1e-09,1e-08,1e-07,1e-06,1e-05,0.0001,0.001,0.01,0.1,1},
yticklabels={
  \(\displaystyle {10^{-11}}\),
  \(\displaystyle {10^{-10}}\),
  \(\displaystyle {10^{-9}}\),
  \(\displaystyle {10^{-8}}\),
  \(\displaystyle {10^{-7}}\),
  \(\displaystyle {10^{-6}}\),
  \(\displaystyle {10^{-5}}\),
  \(\displaystyle {10^{-4}}\),
  \(\displaystyle {10^{-3}}\),
  \(\displaystyle {10^{-2}}\),
  \(\displaystyle {10^{-1}}\),
  \(\displaystyle {10^{0}}\)
}
]
\addplot [line width=0.44pt, gray, dash pattern=on 7.04pt off 1.76pt on 1.1pt off 1.76pt]
table {%
0.1 0.01
0.05 0.0025
0.025 0.000625
0.0125 0.00015625
0.00625 3.90625e-05
0.003125 9.765625e-06
0.0015625 2.44140625e-06
0.00078125 6.103515625e-07
0.000390625 1.52587890625e-07
0.0001953125 3.814697265625e-08
9.765625e-05 9.5367431640625e-09
4.8828125e-05 2.38418579101563e-09
};
\addlegendentry{$\varepsilon^2$}
\addplot [line width=0.7pt, blue]
table {%
0.1 0.00950701953457835
0.05 0.00251122838489344
0.025 0.000644663677002852
0.0125 0.000162994752460728
0.00625 4.08193015753572e-05
0.003125 1.01334772948717e-05
0.0015625 2.48428970530139e-06
0.00078125 5.94852314733739e-07
0.000390625 1.35392696454278e-07
0.0001953125 2.71616754976746e-08
9.765625e-05 3.44387933727102e-09
4.8828125e-05 8.12710682080695e-10
};
\addlegendentry{$\D\Phi_{\text{FK}}$}
\addplot [line width=0.7pt, red, dash pattern=on 17.5pt off 8.75pt]
table {%
0.1 0.00951387461263845
0.05 0.00251465592392349
0.025 0.000646377446517877
0.0125 0.00016385163721824
0.00625 4.12477439541135e-05
0.003125 1.03476984842498e-05
0.0015625 2.59140029999045e-06
0.00078125 6.48407612078271e-07
0.000390625 1.62170345126544e-07
0.0001953125 4.05504998338076e-08
9.765625e-05 1.01382915053375e-08
4.8828125e-05 2.53449540195256e-09
};
\addlegendentry{$\D\Phi_{\text{vol}}$}
\addplot [line width=0.7pt, yellow, dash pattern=on 8.75pt off 17.5pt]
table {%
0.1 0.00950454675083635
0.05 0.00250999199302244
0.025 0.000644045481067352
0.0125 0.000162685654492978
0.00625 4.06647525914823e-05
0.003125 1.00562028029342e-05
0.0015625 2.44565245933266e-06
0.00078125 5.75533691749377e-07
0.000390625 1.25733384962097e-07
0.0001953125 2.23320197515842e-08
9.765625e-05 1.02905146422583e-09
4.8828125e-05 2.02012461860329e-09
};
\addlegendentry{$\D\Phi_{\text{bdry}}$}
\addplot [draw=limegreen, fill=limegreen, mark=triangle*, only marks]
table{%
x  y
0.1 0.00979927089030268
0.05 0.0026573540627556
0.025 0.000717726515933935
0.0125 0.000199526171926269
0.00625 5.9085011308128e-05
0.003125 1.9266332161257e-05
0.0015625 7.05071713849408e-06
0.00078125 2.87806603133008e-06
0.000390625 1.27699955475245e-06
0.0001953125 5.97965104646761e-07
9.765625e-05 2.88845593911814e-07
4.8828125e-05 1.41888146605191e-07
};
\addlegendentry{$\D\Phi_{\text{free}}$}
\end{axis}

\end{tikzpicture}
		}
		\subcaption*{$V_5$ -- constant shift }
	\end{subfigure}
	\hfill    
	\begin{subfigure}{0.32\textwidth}
		\centering
		\resizebox{\textwidth}{\textwidth}{
\begin{tikzpicture}

\definecolor{darkgray176}{RGB}{176,176,176}
\definecolor{gray}{RGB}{128,128,128}
\definecolor{lightgray204}{RGB}{204,204,204}
\definecolor{limegreen}{RGB}{50,205,50}
\definecolor{yellow}{RGB}{255,255,0}

\begin{axis}[
legend cell align={left},
legend style={
  fill opacity=0.8,
  draw opacity=1,
  text opacity=1,
  at={(0.03,0.97)},
  anchor=north west,
  draw=lightgray204
},
log basis x={10},
log basis y={10},
tick align=outside,
tick pos=left,
title={},
x grid style={darkgray176},
xlabel={\(\displaystyle \varepsilon\)},
xmajorgrids,
xmin=3.33505922059178e-05, xmax=0.146408569594563,
xmode=log,
xtick style={color=black},
xtick={1e-06,1e-05,0.0001,0.001,0.01,0.1,1,10},
xticklabel style={rotate=30.0,anchor=east},
xticklabels={
  \(\displaystyle {10^{-6}}\),
  \(\displaystyle {10^{-5}}\),
  \(\displaystyle {10^{-4}}\),
  \(\displaystyle {10^{-3}}\),
  \(\displaystyle {10^{-2}}\),
  \(\displaystyle {10^{-1}}\),
  \(\displaystyle {10^{0}}\),
  \(\displaystyle {10^{1}}\)
},
y grid style={darkgray176},
ylabel={\(\displaystyle \varepsilon^2\)},
ymajorgrids,
ymin=9.31933952382929e-11, ymax=0.0282679889551096,
ymode=log,
ytick style={color=black},
ytick={1e-13,1e-11,1e-09,1e-07,1e-05,0.001,0.1,10},
yticklabels={
  \(\displaystyle {10^{-13}}\),
  \(\displaystyle {10^{-11}}\),
  \(\displaystyle {10^{-9}}\),
  \(\displaystyle {10^{-7}}\),
  \(\displaystyle {10^{-5}}\),
  \(\displaystyle {10^{-3}}\),
  \(\displaystyle {10^{-1}}\),
  \(\displaystyle {10^{1}}\)
}
]
\addplot [line width=0.44pt, gray, dash pattern=on 7.04pt off 1.76pt on 1.1pt off 1.76pt]
table {%
0.1 0.01
0.05 0.0025
0.025 0.000625
0.0125 0.00015625
0.00625 3.90625e-05
0.003125 9.765625e-06
0.0015625 2.44140625e-06
0.00078125 6.103515625e-07
0.000390625 1.52587890625e-07
0.0001953125 3.814697265625e-08
9.765625e-05 9.5367431640625e-09
4.8828125e-05 2.38418579101563e-09
};
\addlegendentry{$\varepsilon^2$}
\addplot [line width=0.7pt, blue]
table {%
0.1 0.0114514322669706
0.05 0.00305685288244662
0.025 0.000789668848257987
0.0125 0.000200399389373663
0.00625 5.03196197597787e-05
0.003125 1.25274604966879e-05
0.0015625 3.08513496147602e-06
0.00078125 7.4534152998308e-07
0.000390625 1.73041282158288e-07
0.0001953125 3.65728419740894e-08
9.765625e-05 5.79441447450137e-09
4.8828125e-05 2.26426100847268e-10
};
\addlegendentry{$\D\Phi_{\text{FK}}$}
\addplot [line width=0.7pt, red, dash pattern=on 17.5pt off 8.75pt]
table {%
0.1 0.0114582909809162
0.05 0.00306028223941941
0.025 0.000791383526744384
0.0125 0.000201256728616862
0.00625 5.0748289381378e-05
0.003125 1.27417953074875e-05
0.0015625 3.19230236687585e-06
0.00078125 7.98925232682992e-07
0.000390625 1.99833133508243e-07
0.0001953125 4.99687676490672e-08
9.765625e-05 1.24923773119903e-08
4.8828125e-05 3.12255531789718e-09
};
\addlegendentry{$\D\Phi_{\text{vol}}$}
\addplot [line width=0.7pt, yellow, dash pattern=on 8.75pt off 17.5pt]
table {%
0.1 0.0114489617824138
0.05 0.00305561764016824
0.025 0.000789051227118797
0.0125 0.000200090578804068
0.00625 5.01652144749812e-05
0.003125 1.24502578542891e-05
0.0015625 3.04653364027665e-06
0.00078125 7.26040869383391e-07
0.000390625 1.63390951858443e-07
0.0001953125 3.17476768241671e-08
9.765625e-05 3.38183189954019e-09
4.8828125e-05 1.43271738832786e-09
};
\addlegendentry{$\D\Phi_{\text{bdry}}$}
\addplot [draw=limegreen, fill=limegreen, mark=triangle*, only marks]
table{%
x  y
0.1 0.011634656329052
0.05 0.00314846491348732
0.025 0.000835474863778336
0.0125 0.000223302397133838
0.00625 6.17711236398661e-05
0.003125 1.82532124367316e-05
0.0015625 5.94801093149787e-06
0.00078125 2.176779514994e-06
0.000390625 8.8876027466375e-07
0.0001953125 3.9443233822682e-07
9.765625e-05 1.84724162600867e-07
4.8828125e-05 8.92384479623355e-08
};
\addlegendentry{$\D\Phi_{\text{free}}$}
\end{axis}

\end{tikzpicture}
		}
		\subcaption*{$V_6$ -- scaled outer normal}
	\end{subfigure}
	
	\smallskip
	
	\begin{subfigure}{0.32\textwidth}
		\centering
		\resizebox{\textwidth}{\textwidth}{
\begin{tikzpicture}

\definecolor{darkgray176}{RGB}{176,176,176}
\definecolor{gray}{RGB}{128,128,128}
\definecolor{lightgray204}{RGB}{204,204,204}
\definecolor{limegreen}{RGB}{50,205,50}
\definecolor{yellow}{RGB}{255,255,0}

\begin{axis}[
legend cell align={left},
legend style={
  fill opacity=0.8,
  draw opacity=1,
  text opacity=1,
  at={(0.03,0.97)},
  anchor=north west,
  draw=lightgray204
},
log basis x={10},
log basis y={10},
tick align=outside,
tick pos=left,
title={},
x grid style={darkgray176},
xlabel={\(\displaystyle \varepsilon\)},
xmajorgrids,
xmin=3.33505922059178e-05, xmax=0.146408569594563,
xmode=log,
xtick style={color=black},
xtick={1e-06,1e-05,0.0001,0.001,0.01,0.1,1,10},
xticklabel style={rotate=30.0,anchor=east},
xticklabels={
  \(\displaystyle {10^{-6}}\),
  \(\displaystyle {10^{-5}}\),
  \(\displaystyle {10^{-4}}\),
  \(\displaystyle {10^{-3}}\),
  \(\displaystyle {10^{-2}}\),
  \(\displaystyle {10^{-1}}\),
  \(\displaystyle {10^{0}}\),
  \(\displaystyle {10^{1}}\)
},
y grid style={darkgray176},
ylabel={\(\displaystyle \varepsilon^2\)},
ymajorgrids,
ymin=1.04003477871596e-09, ymax=0.0877964568545007,
ymode=log,
ytick style={color=black},
ytick={1e-10,1e-09,1e-08,1e-07,1e-06,1e-05,0.0001,0.001,0.01,0.1,1},
yticklabels={
  \(\displaystyle {10^{-10}}\),
  \(\displaystyle {10^{-9}}\),
  \(\displaystyle {10^{-8}}\),
  \(\displaystyle {10^{-7}}\),
  \(\displaystyle {10^{-6}}\),
  \(\displaystyle {10^{-5}}\),
  \(\displaystyle {10^{-4}}\),
  \(\displaystyle {10^{-3}}\),
  \(\displaystyle {10^{-2}}\),
  \(\displaystyle {10^{-1}}\),
  \(\displaystyle {10^{0}}\)
}
]
\addplot [line width=0.44pt, gray, dash pattern=on 7.04pt off 1.76pt on 1.1pt off 1.76pt]
table {%
0.1 0.01
0.05 0.0025
0.025 0.000625
0.0125 0.00015625
0.00625 3.90625e-05
0.003125 9.765625e-06
0.0015625 2.44140625e-06
0.00078125 6.103515625e-07
0.000390625 1.52587890625e-07
0.0001953125 3.814697265625e-08
9.765625e-05 9.5367431640625e-09
4.8828125e-05 2.38418579101563e-09
};
\addlegendentry{$\varepsilon^2$}
\addplot [line width=0.7pt, blue]
table {%
0.1 0.0380360057566123
0.05 0.0106309333313251
0.025 0.00281638447840771
0.0125 0.000725612261107005
0.00625 0.000184413609018146
0.003125 4.66059212179456e-05
0.0015625 1.17751951513199e-05
0.00078125 2.98957381653911e-06
0.000390625 7.68261906165758e-07
0.0001953125 2.02246769840492e-07
9.765625e-05 5.56206885177798e-08
4.8828125e-05 1.64307117309374e-08
};
\addlegendentry{$\D\Phi_{\text{FK}}$}
\addplot [line width=0.7pt, red, dash pattern=on 17.5pt off 8.75pt]
table {%
0.1 0.0380256626023414
0.05 0.0106257617541897
0.025 0.00281379868983998
0.0125 0.000724319366823141
0.00625 0.000183767161876214
0.003125 4.62826976469796e-05
0.0015625 1.16135833658369e-05
0.00078125 2.90876792379761e-06
0.000390625 7.2785895979501e-07
0.0001953125 1.82045296655118e-07
9.765625e-05 4.55199519250928e-08
4.8828125e-05 1.13803434345939e-08
};
\addlegendentry{$\D\Phi_{\text{vol}}$}
\addplot [line width=0.7pt, yellow, dash pattern=on 8.75pt off 17.5pt]
table {%
0.1 0.0380582261850631
0.05 0.0106420435455505
0.025 0.00282193958552041
0.0125 0.000728389814663354
0.00625 0.000185802385796321
0.003125 4.73003096070328e-05
0.0015625 1.21223893458635e-05
0.00078125 3.16317091381092e-06
0.000390625 8.55060454801665e-07
0.0001953125 2.45646044158446e-07
9.765625e-05 7.73203256767565e-08
4.8828125e-05 2.72805303104258e-08
};
\addlegendentry{$\D\Phi_{\text{bdry}}$}
\addplot [draw=limegreen, fill=limegreen, mark=triangle*, only marks]
table{%
x  y
0.1 0.0382987638466794
0.05 0.0107623123763587
0.025 0.00288207400092448
0.0125 0.00075845702236539
0.00625 0.000200835989647339
0.003125 5.48171115325419e-05
0.0015625 1.5880790308618e-05
0.00078125 5.04237139518817e-06
0.000390625 1.79466069549029e-06
0.0001953125 7.15446164502759e-07
9.765625e-05 3.12220385848913e-07
4.8828125e-05 1.44730560396504e-07
};
\addlegendentry{$\D\Phi_{\text{free}}$}
\end{axis}

\end{tikzpicture}
		}
		\subcaption*{$V_7$ -- contraction to a point}
	\end{subfigure}
	\hfill
	\begin{subfigure}{0.32\textwidth}
		\centering
		\resizebox{\textwidth}{\textwidth}{
\begin{tikzpicture}

\definecolor{darkgray176}{RGB}{176,176,176}
\definecolor{gray}{RGB}{128,128,128}
\definecolor{lightgray204}{RGB}{204,204,204}
\definecolor{limegreen}{RGB}{50,205,50}
\definecolor{yellow}{RGB}{255,255,0}

\begin{axis}[
legend cell align={left},
legend style={
  fill opacity=0.8,
  draw opacity=1,
  text opacity=1,
  at={(0.03,0.97)},
  anchor=north west,
  draw=lightgray204
},
log basis x={10},
log basis y={10},
tick align=outside,
tick pos=left,
title={},
x grid style={darkgray176},
xlabel={\(\displaystyle \varepsilon\)},
xmajorgrids,
xmin=3.33505922059178e-05, xmax=0.146408569594563,
xmode=log,
xtick style={color=black},
xtick={1e-06,1e-05,0.0001,0.001,0.01,0.1,1,10},
xticklabel style={rotate=30.0,anchor=east},
xticklabels={
  \(\displaystyle {10^{-6}}\),
  \(\displaystyle {10^{-5}}\),
  \(\displaystyle {10^{-4}}\),
  \(\displaystyle {10^{-3}}\),
  \(\displaystyle {10^{-2}}\),
  \(\displaystyle {10^{-1}}\),
  \(\displaystyle {10^{0}}\),
  \(\displaystyle {10^{1}}\)
},
y grid style={darkgray176},
ylabel={\(\displaystyle \varepsilon^2\)},
ymajorgrids,
ymin=1.01056645869231e-09, ymax=0.021533566115542,
ymode=log,
ytick style={color=black},
ytick={1e-10,1e-09,1e-08,1e-07,1e-06,1e-05,0.0001,0.001,0.01,0.1,1},
yticklabels={
  \(\displaystyle {10^{-10}}\),
  \(\displaystyle {10^{-9}}\),
  \(\displaystyle {10^{-8}}\),
  \(\displaystyle {10^{-7}}\),
  \(\displaystyle {10^{-6}}\),
  \(\displaystyle {10^{-5}}\),
  \(\displaystyle {10^{-4}}\),
  \(\displaystyle {10^{-3}}\),
  \(\displaystyle {10^{-2}}\),
  \(\displaystyle {10^{-1}}\),
  \(\displaystyle {10^{0}}\)
}
]
\addplot [line width=0.44pt, gray, dash pattern=on 7.04pt off 1.76pt on 1.1pt off 1.76pt]
table {%
0.1 0.01
0.05 0.0025
0.025 0.000625
0.0125 0.00015625
0.00625 3.90625e-05
0.003125 9.765625e-06
0.0015625 2.44140625e-06
0.00078125 6.103515625e-07
0.000390625 1.52587890625e-07
0.0001953125 3.814697265625e-08
9.765625e-05 9.5367431640625e-09
4.8828125e-05 2.38418579101563e-09
};
\addlegendentry{$\varepsilon^2$}
\addplot [line width=0.7pt, blue]
table {%
0.1 0.0088035242166131
0.05 0.00220294818199391
0.025 0.000558736956626131
0.0125 0.000141499629533062
0.00625 3.57767729605341e-05
0.003125 9.06653080582688e-06
0.0015625 2.31700202767582e-06
0.00078125 6.03024026555601e-07
0.000390625 1.6246254210442e-07
0.0001953125 4.64461357244794e-08
9.765625e-05 1.45239233463951e-08
4.8828125e-05 5.08681688668827e-09
};
\addlegendentry{$\D\Phi_{\text{FK}}$}
\addplot [line width=0.7pt, red, dash pattern=on 17.5pt off 8.75pt]
table {%
0.1 0.00879756308883798
0.05 0.00219996761810634
0.025 0.00055724667468235
0.0125 0.000140754488561171
0.00625 3.54042024745887e-05
0.003125 8.88024556285419e-06
0.0015625 2.22385940618948e-06
0.00078125 5.56452715812428e-07
0.000390625 1.39176886732834e-07
0.0001953125 3.48033080386863e-08
9.765625e-05 8.70250950349853e-09
4.8828125e-05 2.17610996524e-09
};
\addlegendentry{$\D\Phi_{\text{vol}}$}
\addplot [line width=0.7pt, yellow, dash pattern=on 8.75pt off 17.5pt]
table {%
0.1 0.0088035242166135
0.05 0.00220294818199411
0.025 0.000558736956626232
0.0125 0.000141499629533112
0.00625 3.57767729605592e-05
0.003125 9.06653080583945e-06
0.0015625 2.31700202768211e-06
0.00078125 6.03024026558745e-07
0.000390625 1.62462542105992e-07
0.0001953125 4.64461357252655e-08
9.765625e-05 1.45239233467881e-08
4.8828125e-05 5.08681688688478e-09
};
\addlegendentry{$\D\Phi_{\text{bdry}}$}
\addplot [draw=limegreen, fill=limegreen, mark=triangle*, only marks]
table{%
x  y
0.1 0.00886620458036014
0.05 0.00223428836386742
0.025 0.00057440704756289
0.0125 0.000149334675001441
0.00625 3.96942956947238e-05
0.003125 1.10252921729217e-05
0.0015625 3.29638271122324e-06
0.00078125 1.09271436832931e-06
0.000390625 4.07307712991275e-07
0.0001953125 1.68868721167907e-07
9.765625e-05 7.57352160681089e-08
4.8828125e-05 3.56924632475452e-08
};
\addlegendentry{$\D\Phi_{\text{free}}$}
\end{axis}

\end{tikzpicture}
		}
		\subcaption*{$V_8$ -- star-shaped domain}
	\end{subfigure}
	\hfill
	\begin{subfigure}{0.32\textwidth}
		\centering
		\resizebox{\textwidth}{\textwidth}{
			\includegraphics{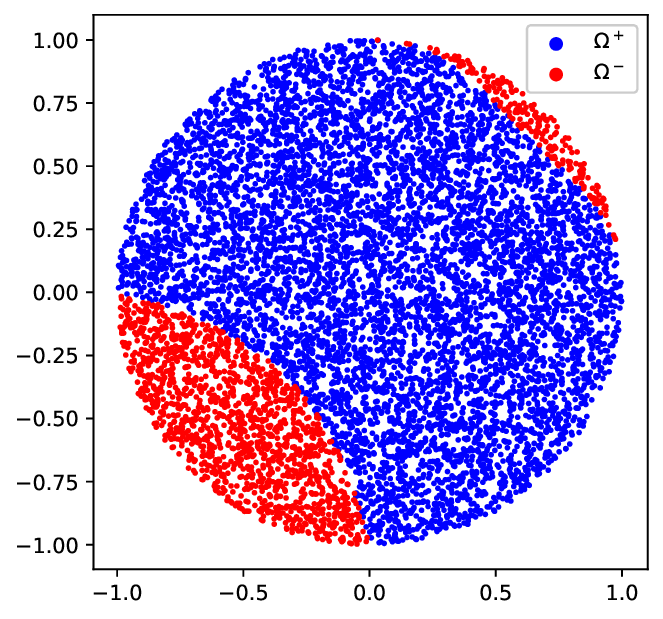}
		}
		\subcaption*{domain partition}
	\end{subfigure}
	\caption{
		\small{
			Taylor Test Results -- Perturbations are introduced in Table~\ref{tab:TestResults}. For readability the legends simply indicate the type of shape derivative used to compute the error $\mathcal{E}(\D\Phi;\varepsilon)$.
		}
	}
	\label{Figure:TaylorTest}%
\end{figure}

The results show that the mesh-free representation is consistent and competitive with classical approaches and generally performs similarly as the corresponding boundary formulation.
We wish to stress that our benchmark implementation is limited in the sense that $\D\Phi_{\text{free}}$ is evaluated mesh-free on the (exact) unit sphere, whereas the representations $\D\Phi_{\text{vol}}, \D\Phi_{\text{bdry}}$ and $\D\Phi_{\text{FK}}$ as well as the shape functional differences $\Phi\big(\Omega_\varepsilon^\dir\big)-\Phi(\Omega)$ are evaluated using finite elements and, in particular, a mesh discretization of the unit sphere. This leads to a systematic numerical error that becomes increasingly significant for smaller distortion factors ($\varepsilon\to 0$).

Further analysis of possible improvements in the implementation of the probabilistic representation of the shape derivative, as well as possible applications in the context of a stochastic gradient scheme for shape optimization, are left for future research.
\bibliographystyle{plain}
\bibliography{references}

\appendix
\section{Stochastic Gronwall Inequality with Stopping Times}\label{section:StochasticGronwall}
This Appendix provides a stochastic Gronwall inequality for random time horizons, i.e.\ up to a stopping time. The result can be obtained as a special case of \cite[Theorem 1]{korea_gronwall}; since for the applications in this paper slightly stronger integrability conditions hold, we can give a more direct and shorter proof. For deterministic time horizons, stochastic Gronwall inequalities on general (not necessarily Brownian) filtered probability spaces can be found in \cite[Corollary B1]{Duffie_Stochastic_Differential_Utility} and \cite[Theorem 1.8]{Antonelli_Stability_of_BSDEs}.

\begin{lemma}[Stochastic Gronwall Inequality for Stopping Times]\label{lemma:StochasticGronwallRandomTimes}
    Let $p>1$.
    Suppose that $\alpha$ and $Y$ are non-negative progressively measurable processes with
    \begin{align*}
        \E\bigg[\Big(\int_0^\tau |\alpha_s Y_s|^p \de s\Big)^2 \bigg] < \infty \qquad\text{and}\qquad \E\bigg[\exp\Big(\frac{2p}{2p-1}\int_0^\tau \alpha_s\de s\Big)\bigg] < \infty,
    \end{align*}
    and that $\tau$ is a stopping time satisfying $\E[\exp(\rho\tau)]<\infty$ for some $\rho>0$.\footnote{The proof shows that it would be sufficient to require $\E[\tau^{2(p-1)}]<\infty$.}
    If $\xi\in L^p (\F_\tau)$ and
    \begin{equation*}
        Y_t \leq \E \Big[\xi + \int_{t}^{\tau} \alpha_s Y_s \de s \condExpectation{\F_t}\Big]\quad\text{on }\{t\leq\tau\}
    \end{equation*}
    then
    \begin{align*}
        Y_t \leq \E\Big[\exp\Big(\int_t^\tau \alpha_s \de s\Big)\xi \Big| \F_t\Big]\quad\text{on }\{t\leq\tau\}.\closeEqn
    \end{align*}
\end{lemma}
\begin{proof}
    Define
    \begin{align*}
        \eta \defined \xi + \int_0^\tau \alpha_sY_s \de s.
    \end{align*}
    Using the elementary inequality $|a+b|^p\leq 2^{p-1}(|a|^p+|b|^p)$ for $a,b\in\R$ and the Jensen and H\"older inequalities we obtain
    \begin{align*}
        \E \big[|\eta|^p\big] &\leq 2^{p-1} \E\Big[|\xi|^p + \tau^{p-1}\int_0^\tau {|\alpha_s Y_s|}^p\de s\Big]\\
        &\leq 2^{p-1}\E\big[|\xi|^p\big] + 2^{p-1} {\E\big[\tau^{2(p-1)}\big]}^{\frac{1}{2}} {\E\bigg[\Big(\int_0^\tau |\alpha_s Y_s|^p \de s\Big)^2 \bigg]}^{\frac{1}{2}} < \infty
    \end{align*}
    and conclude that $\eta \in L^p(\F_\tau)$. By the martingale representation theorem, see \cite[Corollary 2.44]{Pardoux2014}, there exists a progressively measurable process $Z$ such that 
    \begin{align*}
        \eta = \E[\eta] + \int_0^\tau Z_s\de W_s
    \end{align*}
    with $Z_t=0$ on $\{t\geq\tau\}$ and
    \begin{align}\label{eqn:ProofStochasticGronwallNormOfRepresentant}
        \E\bigg[\Big(\int_0^\infty Z_s^2\de s\Big)^{\frac{p}{2}}\bigg] < \infty.
    \end{align}
    Next, we introduce the process $\beta\defined\exp(\int_0^{\cdot\wedge\tau} \alpha_s\de s)$ and observe that
    \begin{align*}
        \E\Big[\int_0^t\big|\beta_sZ_s\big|^2\de s\Big] \leq \E\Big[\beta_\tau\int_0^\infty\big|Z_s\big|^2\de s\Big]\leq \E\big[\beta_\tau^q\big]^\frac{1}{q} \E\bigg[\Big(\int_0^\infty\big|Z_s\big|^2\de s\Big)^\frac{p}{2} \bigg]^\frac{2}{p} <\infty
    \end{align*}
    for any $t\geq 0$, where $q \defined 2p(2p-1)^{-1}$.
    In particular
    \begin{align*}
        \E\Big[\int_0^t\big|Z_s\big|^2\de s\Big] \leq \E\bigg[\Big(\int_0^\infty\big|Z_s\big|^2\de s\Big)^\frac{p}{2} \bigg]^\frac{2}{p} <\infty
    \end{align*}
    so both
    \begin{align}\label{eqn:ProofStochasticGronwallUIMartingales}
        \bigg(\int_0^t \beta_sZ_s \de W_s\bigg)_{t\geq0} \qquad\text{and}\qquad \bigg(\int_0^t Z_s \de W_s\bigg)_{t\geq0}
    \end{align}
    are uniformly integrable martingales. 
    Defining the auxiliary process
    \begin{align*}
        B_t \defined \E\Big[\xi +\int_{t\wedge\tau}^\tau \alpha_sY_s\de s \Big| \F_t\Big],\quad t\geq 0
    \end{align*}
    we have $B_\tau = \xi$ and, since by definition $\xi=\eta-\int_0^\tau \alpha_s Y_s\de s$,
    \begin{align*}
        B_t = \E\Big[\eta - \int_0^{t\wedge\tau}\alpha_sY_s\de s\condExpectation{\F_t}\Big] = \E[\eta] - \int_0^{t\wedge\tau}\alpha_sY_s\de s + \int_0^{t\wedge\tau} Z_s\de W_s.
    \end{align*}
    Here we use the martingale representation of $\eta$ and \eqref{eqn:ProofStochasticGronwallUIMartingales}. Thus $B$ is a semimartingale and the It\=o product formula yields
    \begin{align*}
        \beta_\tau B_\tau - \beta_t B_t = \int_t^\tau \beta_s \big(-\alpha_sY_s\big) \de s + \int_t^\tau \beta_s Z_s\de W_s + \int_t^\tau B_s \beta_s \alpha_s \de s.
    \end{align*}
    Rearranging and using \eqref{eqn:ProofStochasticGronwallUIMartingales} we obtain
    \begin{align*}
        \beta_t B_t = \E\Big[\beta_\tau B_\tau + \int_t^\tau \beta_s\alpha_s \big(Y_s-B_s\big)\de s\Big| \F_t \Big] \leq \E\big[\beta_\tau \xi\big| \F_t\big]
    \end{align*}
    since $Y_s\leq B_s$ on $\{s\leq\tau\}$ and $B_\tau = \xi$. It follows that
    \begin{align*}
        Y_t \leq B_t \leq \E\Big[\frac{\beta_\tau}{\beta_t} \xi\Big|\F_t\Big]
    \end{align*}
    and the proof is complete. 
\end{proof}

\section{Supplements}\label{section:Supplements}
\begin{definition}[Hölder Space]\label{Definition Hoelder Space}
	Let $k\in\nat_0$ and $\gamma\in(0,1)$. The Hölder space $\HoelderSpace{k}{\gamma}(\Omega)$ is
	\begin{align}
		\HoelderSpace{k}{\gamma}(\Omega)\defined\big\{f\in\C^k(\Omega)\,\big|\,\|f\|_{\HoelderSpace{k}{\gamma} (\Omega)}<\infty\big\},
	\end{align}
	where
	\[
	\|f\|_{\HoelderSpace{k}{\gamma}(\Omega)}\defined\sum_{|\beta|\leq k}{\|D^\beta f\|}_\infty +\sum_{|\beta|=k}\langle D^\beta f\rangle_\gamma
	\]
	and
	\[
	\langle h \rangle_\gamma \defined \sup_{x,y\in\Omega,\, x\neq y} \frac{|h(x)-h(y)|}{\|x-y\|^\gamma}. \closeEqn
	\]
\end{definition}

\begin{lemma}\label{lemma:UniformConvergenceDirectionalDerivative}
	Let $A\subseteq \mathcal{O}\subseteq\R^d$ where $\mathcal{O}$ is a bounded domain and $A$ is closed. If $f\colon \mathcal{O}\rightarrow\R$ is of class $\C^2 (\mathcal{O})$ and $\Diff^2 f$ is uniformly bounded, then 
	\begin{equation*}
		\lim_{\varepsilon\to 0}\sup \bigg\{ \Big|\frac{1}{\varepsilon} \Big(f \big(x_\varepsilon+\varepsilon\dir (x_\varepsilon)\big) - f(x_\varepsilon)\Big) - \Diff f(x)\dir (x)\Big|\,\bigg|\, x\in A, (x_\varepsilon)\subseteq \mathcal{O} \text{ with } x_\varepsilon\to x \bigg\} = 0
	\end{equation*}
	for any $\dir \in\pertSpace$. \close
\end{lemma}
\begin{proof}
	Let $\dir^*\defined \sup_{x\in A}\|\dir(x)\|$.
	For each $x\in A$ there is $\varepsilon_0>0$ such that the open ball with radius $\varepsilon_0 \dir^*$ is contained in $\mathcal{O}$. By Taylor expansion
	\begin{align*}
		\frac{1}{\varepsilon} f \big(x_\varepsilon+\varepsilon\dir (x_\varepsilon)\big) - f(x_\varepsilon) = \Diff f (x_\varepsilon)\dir (x_\varepsilon) + \varepsilon \Big\langle\dir (x_\varepsilon), \Diff^2 f\big(x_\varepsilon+\alpha\varepsilon\dir (x_\varepsilon) \big)\dir(x_\varepsilon)\Big\rangle,
	\end{align*}
	for some $\alpha\in[0,1]$. 
	Since $\dir$ is of class $\C^2(\R^d)$ and $\Diff^2 f$ is bounded by assumption we have
	\begin{align*}
		\bigg|\varepsilon \Big\langle\dir (x_\varepsilon), \Diff^2 f\big(x_\varepsilon+\alpha\varepsilon\dir (x_\varepsilon) \big)\dir(x_\varepsilon)\Big\rangle\bigg| \leq \varepsilon \big\|\dir(x_\varepsilon)\big\|^2 \Big\|\Diff^2 f\big(x_\varepsilon+\varepsilon\dir (x_\varepsilon)\big)\Big\| \leq C \varepsilon.
	\end{align*}
	Moreover, $\Diff f$ is Lipschitz on $\mathcal{O}$ since $\Diff^2 f$ is bounded; since $\mathcal{O}$ is bounded, so is $\Diff f$ and consequently $\Diff f \cdot \dir$ is Lipschitz on $\mathcal{O}$; hence
	\begin{align*}
		&\lim_{\varepsilon\to 0}\sup\Big\{\big|\Diff f(x_\varepsilon)\dir(x_\varepsilon) - \Diff f(x)\dir (x)\big|\;\Big|\, x\in A, (x_\varepsilon)\subseteq O\text{ with } x_\varepsilon \to x \Big\} = 0. \qedhere
	\end{align*}
\end{proof}

\begin{lemma}\label{lemma:BoundaryPreservation}
	Let $A\subseteq\mathcal{O}\subseteq\R^d$ where $\mathcal{O}$ is open and $A$ is closed.
	If $T\colon \mathcal{O}\rightarrow\R^d$ is bi-continuous, i.e.\ continuous with continuous inverse, then 
	\[
	T(\partial A)=\partial T(A).\closeEqn
	\]
\end{lemma}
\begin{proof}
	First, we prove $T(\partial A)\subseteq\partial T(A)$. For this, let $x\in\partial A$ and set $y\defined T(x) \in T(\partial A)$. Let $U_y\subseteq\R^d$ be some arbitrary open neighborhood of $y$. Then by continuity of $T$, the set $T^{-1}(U_y)$ is open and $x\in T^{-1}(U_y)$. Since $x\in\partial A$, we have $T^{-1}(U_y)\cap (\R^d\setminus A)\neq \emptyset$, as well as $T^{-1}(U_y)\cap A\neq\emptyset$, where we made explicit use of the openness of $T^{-1}(U_y)$. Thus, for every open neighborhood $U_y$ of $y$ there is at least one point contained in $T(A)$ and one in $T(\R^d\setminus A) = T(\R^d)\setminus T( A)$. Hence $y\in\partial T(A)$.\\
	On the other hand, by continuity of $T^{-1}\colon T(\mathcal{O})\rightarrow \mathcal{O}$ and the first part, we obtain 
	\[
	T^{-1}(\partial(T(A))) \subseteq \partial T^{-1}(T(A))=\partial A,
	\]
	and hence, $\partial(T(A))\subseteq T(\partial A)$.
\end{proof}

\begin{lemma}[Conditional Wald Lemma for Brownian Motion]\label{lemma:WaldLemmaForBrownianMotion}
	Let $H$ be a bounded progressively measurable process.
	If $\tau$ is a stopping time with $\E[\tau]<\infty$ then for every $t\geq 0$
	\[
	\E_t\Big[\int_{t\wedge\tau}^\tau H_s\de W_s \Big] = 0. \closeEqn
	\]
\end{lemma}
\begin{proof}
	The local martingale $M$, $M_t\defined\int_0^{t\wedge\tau}H_s\de W_s$ satisfies $\E[\langle M\rangle_\infty]=\E[\int_0^\tau H_s^2\de s]\leq C\E[\tau]<\infty$, hence is a uniformly integrable martingale.
	Thus by optional stopping
	\begin{equation*}
		0 = \E_t[M_\tau-M_t] = \E_t\Big[ \int_{t\wedge\tau}^\tau H_s\de W_s \Big].\qedhere
	\end{equation*}
\end{proof}
%
%
\MW{
	\begin{lemma}\label{lemma:elementaryBoundEXP}
		For $k\in\nat$ and $\alpha > 2$ it holds that $x^k e^{2x} \leq \tfrac{k^k}{e^k(\alpha-2)^k} e^{\alpha x}$ for all $x\geq 0$. \close 
	\end{lemma}
	\begin{proof}
		The continuous map $h\colon [0,\infty)\to [0,\infty);\ x\mapsto x^k \exp((2-\alpha)x)$ has a global maximum at $x^* \defined \tfrac{k}{\alpha-2}$.
		For this notice $h^\prime (x) = (kx^{k-1}+(2-\alpha)x^k)\exp((2-\alpha)x)$ vanishes only at $x^*$, and moreover $h(0)=0$ and $h(x)\to 0$ as $x\to\infty$. Now the claim follows from
		\begin{align}
			x^k \exp(2x) & = h(x) \exp(\alpha x) \leq h(x^*) \exp(\alpha x) \quad\text{for all } x\geq 0.  \qedhere
		\end{align}
	\end{proof}
}
\subsection*{Proof of Proposition \ref{Prop:ShapeDerivativeSplitt}}\label{section:proofTheoremShapeFunctional}
\begin{lemma}\label{lemma:IntegralConvergenceForDerivativeSplitt}
	Let $\dir\in\pertSpace$. Then
	\begin{equation*}
		\sup_{\varepsilon\in [-\varepsilon_0,\varepsilon_0]}\, \sup_{x\in\overline{\pertDomain}}\, \big|u_\varepsilon^\dir (x)\big| < \infty.
	\end{equation*}
	Moreover, for the extension $\uExt\colon\HoldAll\rightarrow\R$, we have
	\begin{align*}
		\frac{1}{\varepsilon} \int_{\pertDomain} \big|u_\varepsilon^\dir(x) - \uExt(x)\big|\, \1_{\pertDomain\setminus\Omega}(x)\, \dLebesgue{d}{x} \rightarrow 0\quad \text{as }\varepsilon\to 0. \closeEqn
	\end{align*}
\end{lemma}
\begin{proof}
	Let $x\in\pertDomain$. By Lemma~\ref{lemma:ProofStep3} (for $t=0$) and using \eqref{eqn:proofStrongConvergenceUniformExitBound}, there is a constant $C>0$ (not depending on $x$ or $\varepsilon$) such that
	\begin{align*}
		\big|u_\varepsilon^\dir (x) - \uExt (x)\big| \leq C \varepsilon\, \E\big[ (\pertExit +1)\exp(\pertExit)\big] \leq C \varepsilon.
	\end{align*}
	Using this bound, we conclude that
	\begin{align*}
		\big|u_\varepsilon^\dir (x)\big| \leq \big|u_\varepsilon^\dir (x) - \uExt (x)\big| + \big|\uExt (x)\big| \leq C \varepsilon_0 + \sup_{y\in\overline{\pertDomain}}\big|\uExt (y)\big| < \infty
	\end{align*}
	and
	\begin{equation*}
		\frac{1}{\varepsilon} \int_{\pertDomain} \big|u_\varepsilon^\dir(x) - \uExt(x)\big|\, \1_{\pertDomain\setminus\Omega}(x)\, \dLebesgue{d}{x} \leq C \lambda^d \big[\pertDomain\setminus\Omega\big].\quad \qedhere
	\end{equation*}
\end{proof}

\begin{proof}[Proof of Proposition~\ref{Prop:ShapeDerivativeSplitt}]
	Throughout the proof $\dir\in\pertSpace$ is fixed. As in the proof of Theorem \ref{theorem:StronConvergenceShapeDerivative}, let $\Tilde{u}\colon\HoldAll\rightarrow\R$ denote a $\C^2$ extension of $u$, where $\HoldAll$ is defined in Remark \ref{remark:Epsilon0}. We have
	\begin{align*}
		\D\Phi[\dir] & = \lim_{\varepsilon\rightarrow 0} \frac{1}{\varepsilon} \Big( \int_{\pertDomain} \phi(\cdot, u_\varepsilon^\dir)\dLeb{d}-\int_\Omega\phi(\cdot, u)\dLeb{d} \Big)\\
		& = \lim_{\varepsilon\rightarrow 0} \frac{1}{\varepsilon} \int_{\pertDomain} \phi (\cdot, u_\varepsilon^\dir) - \phi(T_\varepsilon^\dir, u\circ T_\varepsilon^\dir)\big|\det (\Diff T_\varepsilon^\dir)\big| \dLeb{d} \\
		& =\lim_{\varepsilon\rightarrow 0} \int_{\pertDomain} \frac{1}{\varepsilon} \Bigl( \phi (\cdot, u_\varepsilon^\dir) - \phi (\cdot, \Tilde{u}) \Bigr) \dLeb{d}\\
		& \hspace*{2.0cm} + \int_{\pertDomain} \frac{1}{\varepsilon} \phi (\cdot, \Tilde{u})\Bigl( 1 - \big|\det (\Diff T_\varepsilon^\dir)\big| \Bigr) \dLeb{d}\\
		& \hspace*{3.0cm} + \int_{\pertDomain} \frac{1}{\varepsilon} \big(\phi(\cdot, \Tilde{u}) - \phi (\cdot ,u\circ T_\varepsilon^\dir) \big)\big|\det (\Diff T_\varepsilon^\dir)\big|\; \dLeb{d}\\
		& \hspace*{4.0cm} + \int_{\pertDomain} \frac{1}{\varepsilon} \big(\phi (\cdot ,u\circ T_\varepsilon^\dir) - \phi(T_\varepsilon^\dir, u\circ T_\varepsilon^\dir)\big)\big|\det (\Diff T_\varepsilon^\dir)\big|\; \dLeb{d}.
	\end{align*}
	We first discuss the behavior of $\pertDomain$ as $\varepsilon\to 0$. For $x\in\HoldAll$ define $\delta^x \defined \dist(x,\partial\Omega ) / \|\dir \|_\infty$. If $x\in\Omega$ then $\delta^x>0$ and whenever $\varepsilon < \delta^x$ also $T_\varepsilon^\dir (x)\in\Omega$, hence $x\in\pertDomain$. Conversely, if $x\in\HoldAll\setminus\overline{\Omega}$ then $\delta^x > 0$ and again for any $\varepsilon < \delta^x$ the distortion cannot push back into $\Omega$, i.e.\ $T_\varepsilon^\dir (x) \notin \Omega$ or equivalently $x\notin\pertDomain$. Thus
	\begin{equation}\label{proof:ConnectionShapeDerivativesIndicatorConvergence}
		\lim_{\varepsilon\to 0} \1_{\pertDomain} (x) = \1_\Omega (x) \qquad \text{ $\lambda^d$-a.e.}
	\end{equation}
	We now investigate the limits of each of the four integrals. 
	For the first, observe that by \eqref{proof:ConnectionShapeDerivativesIndicatorConvergence} for a.e.\ $x\in\HoldAll\setminus\overline{\Omega}$
	\begin{align*}
		\lim_{\varepsilon\rightarrow 0} \frac{1}{\varepsilon} \Bigl( \phi \big(x, u_\varepsilon^\dir(x)\big) - \phi \big(x, \Tilde{u}(x)\big) \Bigr) \1_{\pertDomain}(x) = 0
	\end{align*}
	while for a.e.\ $x\in \Omega$
	\begin{align*}
		\lim_{\varepsilon\rightarrow 0} \frac{1}{\varepsilon} \Bigl( \phi \big(x, u_\varepsilon^\dir(x)) - \phi \big(x, \tilde u(x)\big) \Bigr) \1_{\pertDomain}(x) = \Diff_u \phi\big(x,u(x)\big) \D u[\dir](x).
	\end{align*}
	Moreover, for $x\in\pertDomain$ we have by the mean value theorem and Lemma~\ref{lemma:IntegralConvergenceForDerivativeSplitt}
	\begin{align*}
		\frac{1}{\varepsilon} \big|\phi \big(x, u_\varepsilon^\dir(x)) - \phi \big(x, \tilde u(x)\big)\big|\, \1_{\pertDomain}(x) \leq C\frac{1}{\varepsilon}\big|\big( u_\varepsilon^\dir(x) - \uExt(x)\big)\big|\, \big(\1_{\pertDomain\cap\Omega}(x) + \1_{\pertDomain\setminus\Omega}(x)\big).
	\end{align*}
	Here
	\begin{multline*}
		\frac{1}{\varepsilon}\big|\big( u_\varepsilon^\dir(x) - u(x)\big)\big|\, \1_{\pertDomain\cap\Omega}(x) \\ \leq \big|\D u[\dir](x)\big|\1_\Omega(x) + \sup_{\delta\in[-\varepsilon_0,\varepsilon_0]\setminus\{0\}} \Big|\frac{1}{\delta}\big( u_\delta^\dir(x) - u(x)\big)- \D u[\dir](x)\Big|\, \1_{\Omega_\delta^\dir\cap\Omega}(x)
	\end{multline*}
	where the first summand is uniformly bounded by Remark~\ref{remark:ShapeDerivativeContinuous} and the second is uniformly bounded for sufficiently small $\varepsilon$ by Theorem~\ref{theorem:StronConvergenceShapeDerivative};
	on the other hand Lemma~\ref{lemma:IntegralConvergenceForDerivativeSplitt} implies that
	\begin{align*}
		\int_{\pertDomain} \frac{1}{\varepsilon}\big|\big( u_\varepsilon^\dir(x) - \uExt(x)\big)\big|\, \1_{\pertDomain\setminus\Omega}(x) \dLebesgue{d}{x} \rightarrow 0 \qquad\text{as }\varepsilon\to 0.
	\end{align*}
	Thus dominated convergence yields
	\begin{align*}
		&\lim_{\varepsilon\rightarrow 0} \int_{\pertDomain} \frac{1}{\varepsilon} \Bigl( \phi (\cdot, u_\varepsilon^\dir) - \phi (\cdot, \Tilde{u}) \Bigr) \dLeb{d} \\
		&\hspace*{0.5cm}= \lim_{\varepsilon\rightarrow 0} \Big(\int_{\HoldAll\cap\Omega} \frac{1}{\varepsilon} \big(\phi (\cdot, u_\varepsilon^\dir) - \phi (\cdot, \tilde{u})\big) \1_{\pertDomain}\dLeb{d} + \int_{\HoldAll\setminus\overline{\Omega}} \frac{1}{\varepsilon} \big(\phi (\cdot, u_\varepsilon^\dir) - \phi (\cdot, \Tilde{u})\big) \1_{\pertDomain}\dLeb{d}\Big)\\
		&\hspace*{0.5cm}= \int_{\Omega} \Dphiu \phi(\cdot, u) \D u[\dir] \dLeb{d}.
	\end{align*}
	
	Turning to the limits of the remaining integrals, note first that there is a function $\varphi\colon\HoldAll\times\R\rightarrow\R$ with $\sup_{x\in\HoldAll}\varepsilon^{-1}|\varphi(x,\varepsilon)| \rightarrow 0$ as $\varepsilon \rightarrow 0$ such that
	\begin{equation}\label{proof:ConnectionShapeDerivativesDeterminantDerivative}
		\det \big(\Diff T_\varepsilon^\dir(x)\big) = \det\big(\mathcal{I}_d-\varepsilon\Diff\dir (x)\big)=1+\varepsilon\trace \big(\Diff \dir (x)\big) + \varphi(x,\varepsilon),\qquad x\in\pertDomain.
	\end{equation}
	In particular $\det (\Diff T_\varepsilon^\dir)\to 1$ as $\varepsilon\to 0$ and
	\[
	\lim_{\varepsilon\rightarrow 0}\frac{1}{\varepsilon}\Big(1 - \big|\det(\Diff T_\varepsilon^\dir)\big|\Big) = -\trace(\Diff\dir ) = -\divergence(\dir).
	\]
	It follows that for a.e.\ $x\in\HoldAll$ we have
	\begin{align*}
		\lim_{\varepsilon\to 0}\frac{1}{\varepsilon} \phi \big(x, \Tilde{u}(x)\big)\Big(1-\big|\det (\Diff T_\varepsilon^\dir)\big|\Big)\1_{\pertDomain} (x) = -\phi \big(x,u(x)\big) \divergence (\dir)(x) \1_\Omega (x)
	\end{align*}
	and it is clear from \eqref{proof:ConnectionShapeDerivativesDeterminantDerivative} and the fact that $\dir$ is of class $\C^2$ and $\HoldAll$ is bounded that there is an integrable majorant.
	Hence
	\[
	\lim_{\varepsilon\rightarrow 0} \int_{\pertDomain} \frac{1}{\varepsilon} \phi (\cdot, \Tilde{u})\Bigl( 1 - \big|\det (\Diff T_\varepsilon^\dir)\big| \Bigr) \dLeb{d}
	= - \int_{\Omega} \phi (\cdot,u) \divergence(\dir) \dLeb{d}.
	\]
	Furthermore, we have
	\begin{multline*}
		\Big(\phi\big(x, \Tilde{u}(x)\big) - \phi \big(x , (u\circ T_\varepsilon^\dir)(x)\big) \Big)\big|\det \big(\Diff T_\varepsilon^\dir(x)\big)\big|\;\1_{\pertDomain}(x) \\ \rightarrow - \Dphiu\phi\big(x,u(x)\big) \big\langle \nabla u,\dir\big\rangle (x)\; \1_\Omega (x)
	\end{multline*}
	as $\varepsilon\to 0$, and again an integrable majorant exists by the mean value theorem so
	\[
	\lim_{\varepsilon\rightarrow 0} \int_{\pertDomain} \frac{1}{\varepsilon} \big(\phi(\cdot, \Tilde{u}) - \phi (\cdot ,u\circ T_\varepsilon^\dir) \big)\big|\det (\Diff T_\varepsilon^\dir)\big|\; \dLeb{d}
	= - \int_{\Omega} \Dphiu\phi\big(\cdot,u\big) \langle \nabla u,\dir\rangle \dLeb{d}.
	\]
	Finally, we have
	\begin{multline}
		\Big(\phi\big(x,(u\circ T_\varepsilon^\dir)(x)\big) - \phi\big(T_\varepsilon^\dir (x), (u\circ T_\varepsilon^\dir)(x)\big)\Big)\big|\det (\Diff T_\varepsilon^\dir (x))\big|\;\1_{\pertDomain} (x)\\ \rightarrow - \big\langle\Dphix\phi (\cdot,u),\dir\big\rangle (x)\;\1_\Omega (x)
	\end{multline}
	as $\varepsilon\to 0$, bounded by an integrable majorant, and thus
	\[
	\lim_{\varepsilon\rightarrow 0} \int_{\pertDomain} \frac{1}{\varepsilon} \big(\phi (\cdot ,u\circ T_\varepsilon^\dir) - \phi(T_\varepsilon^\dir, u\circ T_\varepsilon^\dir)\big)\big|\det (\Diff T_\varepsilon^\dir)\big|\; \dLeb{d}
	= - \int_{\Omega} \big\langle\Dphix\phi(\cdot,u),\dir\big\rangle \dLeb{d}.
	\]
	Combining the preceding three integrals and noting that
	\[
	\divergence\big(\phi(\cdot,u)\dir\big) = \phi(\cdot, u)\divergence(\dir) + \Dphiu\phi(\cdot,u)\big\langle\nabla u,\dir\big\rangle + \big\langle\Dphix\phi(\cdot,u),\dir\big\rangle
	\]
	we conclude that 
	\begin{align*}
		\D\Phi [\dir] = \int_\Omega - \divergence\big(\phi(\cdot,u)\dir\big) \dLeb{d} + \int_\Omega \Dphiu\phi (\cdot, u)\D u[\dir] \dLeb{d}.
	\end{align*}
	Now Gauss's divergence theorem, see e.g.\ \cite[Theorem 10.41]{LeeSmoothManifolds}, yields the claim.
\end{proof}

\end{document}